\documentclass[12pt]{amsart}
\usepackage[mathcal]{eucal}
\usepackage{amssymb, labelfig, graphics, amsxtra, amsthm}
\usepackage{hyperref}

\newtheorem{thm}{Theorem}[section]
\newtheorem{prop}[thm]{Proposition}
\newtheorem{lem}[thm]{Lemma}

\newtheorem{cor}[thm]{Corollary}
\newtheorem{comp}[thm]{Complement}

\theoremstyle{remark}
\newtheorem{rem}[thm]{Remark}

\theoremstyle{definition}

\newcommand{\R}{\mathbb R}
\newcommand{\Z}{\mathbb Z}
\newcommand{\N}{\mathbb N}

\newcommand{\HH}{\mathbb H}
\newcommand{\SL}{\mathrm{SL_n(\R)}}
\newcommand{\PSL}{\mathrm{PSL_n(\R)}}
\newcommand{\GL}{\mathrm{GL_n(\R)}}
\newcommand{\PGL}{\mathrm{PGL_n(\R)}}
\newcommand{\Flag}{\mathrm{Flag}(\R^n)}
\newcommand{\F}{\mathcal F}
\newcommand{\E}{\mathrm{e}}

\newcommand{\leftn}{\left \vert \kern -1.5pt \left  \vert \kern -1.5pt \left \vert}
\newcommand{\rightn}{\right \vert \kern -1.5pt \right \vert \kern -1.5pt \right \vert}
\newcommand{\ML}{\mathcal{ML}}

\newcommand{\Hit}{\mathrm{Hit}_{\mathrm{n}}}
\newcommand{\Id}{\mathrm{Id}}
\newcommand{\CC}{\mathcal C}
\newcommand{\CH}{\mathcal C^{\mathrm{H{\ddot o}l}}}
\newcommand{\slits}{\mathrm{slits}}
\newcommand{\delv}{\partial_{\mathrm v}}
\newcommand{\delh}{\partial_{\mathrm h}}

\newcommand{\db}{/\kern -4pt/}

\renewcommand{\leq}{\leqslant}
\renewcommand{\geq}{\geqslant}
\renewcommand{\phi}{\varphi}
\renewcommand{\epsilon}{\varepsilon}

\title
{Hitchin characters and geodesic laminations}

\author{Francis Bonahon}
\address {Department
of Mathematics,  University of
Southern California, KAP 108, Los Angeles,
CA~90089-2532, U.S.A.}
\email{fbonahon@math.usc.edu}

\author{Guillaume Dreyer}
\address {Department
of Mathematics,  University of
Notre Dame, 255 Hurley Hall, Notre Dame, IN~46556, U.S.A.}
\email{dreyfactor@gmail.com}

\thanks{This research was partially supported by the grants  DMS-0604866 and DMS-1105402 from the U.S. National Science Foundation, and by a  Fellowship from the Simons Foundation (grant 301050). In addition, the authors gratefully acknowledge support from the NSF grants DMS-1107452, 1107263 and 1107367 ``RNMS: GEometric structures And Representation varieties'' (the GEAR Network).}
\date{\today}

\begin{document}

\begin{abstract}
For a closed surface $S$, the Hitchin component $\Hit(S)$ is a preferred component of the character variety consisting of group homomorphisms from the fundamental group $\pi_1(S)$  to the Lie group $\PSL$. We construct a parametrization of the Hitchin component that is well-adapted to a geodesic lamination $\lambda$ on the surface. This is a natural extension of Thurston's parametrization of the Teichm\"uller space $\mathcal T(S)$ by shear coordinates associated to $\lambda$, corresponding to the case $n=2$. However, significantly new ideas are needed in this higher dimensional case. The article concludes with a few applications. 
\end{abstract}

\maketitle

\begin{tableofcontents}

\end{tableofcontents}

\section*{Introduction}
\label{bigsect:Intro}

\subsection{Background and motivation}
For a closed, connected, oriented surface $S$ of genus $g>1$, the \emph{Hitchin component} $\Hit(S)$ is a preferred component of the character variety
$$
\mathcal X_{\PSL}(S) = \{ \text{homomorphisms } \rho \colon \pi_1(S) \to \PSL \} \db \PSL
$$
consisting of group homomorphisms $\rho \colon \pi_1(S) \to \PSL$ from the fundamental group $\pi_1(S)$ to the Lie group $\PSL$ (equal to the special linear group $\SL$ if $n$ is odd, and to  $\SL/\{\pm \Id\}$ if $n$ is even), where $\PSL$ acts on these homomorphisms by conjugation. The quotient should normally be taken in the sense of geometric invariant theory \cite{Mum}, but this subtlety is irrelevant here as this quotient construction coincides with the usual topological quotient on the Hitchin component. 

When $n=2$, the Lie group $\mathrm{PSL}_2(\R)$ is also the orientation-preserving isometry group of the hyperbolic plane $\HH^2$, and the Hitchin component $\mathrm{Hit}_2(S)$   of $\mathcal X_{\mathrm{PSL}_2(\R)}(S)$ consists of all  characters represented by injective homomorphisms $\rho\colon \pi_1(S) \to \mathrm{PSL}_2(\R)$ whose image $\rho \bigl( \pi_1(S) \bigr)$ is discrete in $\mathrm{PSL}_2(\R)$ and for which the natural homotopy equivalence $S \to \HH^2/\rho \bigl( \pi_1(S) \bigr)$ has degree $+1$. The Hitchin component $\mathrm{Hit}_2(S)$ is in this case  called the \emph{Teichm\"uller component}, and can also be described as the space of isotopy classes of hyperbolic metrics on $S$. 

When $n>2$, there is a preferred homomorphism $\mathrm{PSL}_2(\R) \to \PSL$ coming from the unique $n$--dimensional representation of $\mathrm{SL}_2(\R)$ (or, equivalently, from the natural action of $\mathrm{SL}_2(\R) $ on the vector space $\R[X,Y]_{n-1}\cong \R^n$ of homogeneous polynomials  of degree $n-1$ in two variables). This provides a natural map $\mathcal X_{\mathrm{PSL}_2(\R)}(S) \to \mathcal X_{\PSL}(S) $, and the \emph{Hitchin component} $\Hit(S)$ is the component of $\mathcal X_{\PSL}(S)$ that contains the image of  $\mathrm{Hit}_2(S)\subset \mathcal X_{\mathrm{PSL}_2(\R)}(S) $. 
The terminology is motivated by the following fundamental result of Hitchin \cite{Hit}, who was the first to single out this component. 

\begin{thm} [Hitchin]
\label{thm:Hitchin}
The  Hitchin component $\Hit(S)$ is diffeomorphic to $\R^{2(g-1)(n^2-1)}$. 
\end{thm}

A \emph{Hitchin character} is an element of the Hitchin component $\Hit(S)$, and a \emph{Hitchin homomorphism} is a homomorphism $\rho \colon \pi_1(S) \to \PSL$ representing a Hitchin character.  We will use the same letter to represent the Hitchin homomorphism $\rho \colon \pi_1(S) \to \PSL$  and the corresponding Hitchin character $\rho \in \Hit(S)$. 

About 15 years after \cite{Hit}, Labourie \cite{Lab1} showed that Hitchin homomorphisms satisfy many important geometric and dynamical properties, and in particular are injective with discrete image; see also \cite{FoG1}.

Hitchin's construction of the parametrization of  $\Hit(S)$ given by Theorem~\ref{thm:Hitchin} is based on geometric analysis techniques that provide little information on the geometry of the Hitchin homomorphisms themselves; see \cite{Loft, Lab3, Lab2} for  different geometric analytic parametrizations when $n=3$. The current article is devoted to developing another parametrization of the Hitchin component $\Hit(S)$ which is much more geometric, and has the additional advantage of being well-behaved with respect to a geodesic lamination. Geodesic laminations were introduced by Thurston to develop a continuous calculus for simple closed curves on the surface $S$, and provide very powerful tools for many topological and geometric problems in dimensions 2 and 3. See \S\S \ref{bigsect:PseudoAnosov} and \ref{bigsect:LengthMeasLam}  for two simple applications of our parametrization, one to the dynamics of the action of a pseudo-Anosov homomorphism of $S$ on the Hitchin component, and another one to the length functions defined by a Hitchin character on Thurston's space $\ML(S)$ of measured laminations on $S$. 

Our construction is a natural extension of Thurston's parametrization of the Teichm\"uller component by shear coordinates \cite{Thu1, Bon96}. It draws its inspiration from this classical case where $n=2$, but also from work of Fock-Goncharov \cite{FoG1} on a variant of the Hitchin component where the surface $S$ has punctures, and where these punctures are endowed with additional information. As in the classical case when $n=2$, the situation is conceptually and analytically much more complicated for a closed surface than in the case considered in \cite{FoG1}. Many arguments, such as those of \S\S \ref{subsect:Slithering}, \ref{subsect:ParamInjective} and \ref{subsect:PositiveIntersectionRevisited}, are new even for the case $n=2$.  

The companion article \cite{BonDre} is devoted to a special case of our parametrization, when the geodesic lamination has only finitely many leaves. The situation is much simpler in that case, and in particular the arguments of \cite{BonDre} tend to be very combinatorial in nature. The current article has a much more analytic flavor. It is also  more conceptual, and provides a homological interpretation of some of the invariants and phenomena that were developed in a purely computational way in \cite{BonDre}. And of course the framework of general geodesic laminations, possibly with uncountably many leaves, considered in this article is better suited for applications. 

The article \cite{Dre2} was developed, to a large extent, as a first step towards the more general results of the current paper. It investigates all deformations of a Hitchin character $\rho\in \Hit(S)$ that respect its triangle invariants, as discussed in the next section.

\subsection{Main results}
We can now be more specific. Let $\lambda$ be a maximal geodesic lamination in $S$. See \S \ref{bigsect:GeodLam} for precise definitions. What we need to know here is just that, for an arbitrary auxiliary metric of negative curvature on the surface, $\lambda$ is decomposed as  a union of disjoint geodesic leaves, and that its complement $S-\lambda$ consists of $4(g-1)$ infinite triangles with geodesic boundary. Some maximal geodesic laminations, such as the ones considered in \cite{BonDre},  have only a finite number of leaves, but generic examples have uncountably many leaves. 

Given a Hitchin character  $\rho\in \Hit(S)$, the rich dynamical structure for $\rho$ discovered by Labourie \cite{Lab1} associates a triple $(E,F,G)$ of three flags of $\R^n$ to each triangle component $T_j$ of $S-\lambda$. In addition, Fock and Goncharov \cite{FoG1} prove that this flag triple $(E,F,G)$ is positive, in a sense discussed in \S \ref{subsect:Positivity}, and is determined by $\frac{(n-1)(n-2)}2$ invariants $\tau_{abc}^\rho(E,F,G) \in \R$. Since $S-\lambda$ has $4(g-1)$ components, these flag triple invariants  can be collected into a single \emph{triangle invariant}  $\tau^\rho \in \R^{2(g-1)(n-1)(n-2)}$. 

The really new feature introduced in this article describes how to glue these flag triples across the (possibly uncountably many) leaves of the lamination, and simultaneously involves analytic and combinatorial arguments. The analytic part of this analysis is based on the slithering map constructed in \S \ref{subsect:Slithering}, which is a higher dimensional analogue of the horocyclic foliation  that is at the basis of the case $n=2$ \cite{Thu1, Bon96}.  This slithering map enables us to control the gluing by elements of the homology of a train track neighborhood $U$ for $\lambda$, which we now briefly describe. The precise definition of train track neighborhoods can be found in \S \ref{subsect:TrainTracks} (and is familiar to experts); at this point, it suffices to say that $U$ is obtained from $S$ by removing $2(g-1)$ disjoint disks, one in each component of $S-\lambda$; in addition, the boundary $\partial U$ is decomposed into a \emph{horizontal boundary} $\delh U$ and a \emph{vertical boundary} $\delv U$, in such a way that each component of $\partial U$ is a hexagon made up of three arc components of $\delh U$ and three arc components of $\delv U$. 

The geodesic lamination has a well-defined 2--fold \emph{orientation cover} $\widehat\lambda$, whose leaves are continuously oriented, and the covering map $\widehat\lambda \to \lambda$ uniquely extends to a 2--fold cover $\widehat U \to U$. In particular, $\widehat\lambda$ is a geodesic lamination in the surface $\widehat U$. 

Our new invariant for a Hitchin character $\rho \in \Hit(S)$ is a certain \emph{shearing class} $[\sigma^\rho] \in H_1(\widehat U,  \delv \widehat U; \R^{n-1})$. This shearing class has the property that $\iota_* \bigl([\sigma^\rho]\bigr) = - \overline{[ \sigma^\rho]}$, for the covering involution $\iota$ of the cover $\widehat U \to U$ and for the involution $x\mapsto \overline x$ of $\R^{n-1}$ that associates $\overline x = (x_{n-1}, x_{n-2}, \dots, x_1)$ to $x=(x_1, x_2, \dots, x_{n-1})$. In particular, $[\sigma^\rho]$ can also be interpreted as a twisted homology class $[ \sigma^\rho] \in H_1(U,  \delv U; \widetilde\R^{n-1})$  valued in a suitable coefficient bundle $\widetilde\R^{n-1}$ over $U$ with fiber $\R^{n-1}$.

The triangle invariant $\tau^\rho \in  \R^{2(g-1)(n-1)(n-2)}$ and shearing class $[\sigma^\rho] \in H_1(\widehat U,  \delv \widehat U; \R^{n-1})$ satisfy two types of constraints. The first constraint is a homological equality. 

\begin{prop}[Shearing Cycle  Boundary Condition]
\label{prop:ShearingCycleBdryConditionIntro}
The boundary $\partial [\sigma^\rho] \in H_0(\delv \widehat U; \R^{n-1})$ of the shearing  class $[\sigma^\rho] \in H_1(\widehat U,  \delv \widehat U; \R^{n-1})$ of a Hitchin character $\rho \in \Hit(S)$ is completely determined by the triangle invariant $\tau^\rho \in \R^{2(g-1)(n-1)(n-2)}$, by an explicit linear formula given in {\upshape \S \ref{subsect:ShearingCycle}}. 
\end{prop}

The second constraint is a positivity property, proved as Corollary~\ref{cor:TransMeasureHasPositiveIntersection} in \S \ref{subsect:ShearingAndLength}. Because the leaves of the orientation cover $\widehat \lambda$ are oriented,  a famous construction of Ruelle and Sullivan  \cite{RueSul} interprets every  transverse measure $\mu$ for the orientation cover $\widehat\lambda$ as a 1--dimensional de Rham current in $\widehat U$. In particular, such a transverse measure $\mu$ determines a homology class $[\mu] \in H_1(\widehat U; \R)$. 

\begin{prop}[Positive Intersection Condition]
\label{prop:PositiveIntersectionIntro}
For every  transverse measure $\mu$ for the orientation cover $\widehat\lambda$, the algebraic intersection vector $ [\mu]  \cdot [\sigma^\rho] \in \R^{n-1}$ of the shearing class $[\sigma^\rho] \in H_1(\widehat U,  \delv \widehat U; \R^{n-1})$ with  $[\mu] \in H_1(\widehat U; \R)$ is positive, in the sense that all its coordinates are positive. 
\end{prop}

The Shearing Cycle  Boundary and Positive Intersection Conditions restrict the pair $\bigl( \tau^\rho, [\sigma^\rho]\bigr )$ to a convex polyhedral cone $\mathcal P$ in $\R^{2(g-1)(n-1)(n-2)} \times H_1(\widehat U,  \delv \widehat U; \R^{n-1})$. The main result of the article, proved as Theorem~\ref{thm:InvariantsHomeomorphism} in \S \ref{subsect:RealizeInvariants}, shows that these are the only restrictions on the triangle and shearing invariants, and that these provide a parametrization of the Hitchin component $\Hit(S)$. 

\begin{thm}[Parametrization of the Hitchin component]
\label{thm:ParametrizationIntro}
The map $\Hit(S) \to \mathcal P$, which to a Hitchin character $\rho \in \Hit(S)$ associates the pair $\bigl( \tau^\rho, [\sigma^\rho]\bigr )$ formed by its triangle invariant $\tau^\rho \in \R^{2(g-1)(n-1)(n-2)}$ and its shearing class $[\sigma^\rho] \in H_1(\widehat U,  \delv \widehat U; \R^{n-1})$, is a homeomorphism. 
\end{thm}

The Shearing Cycle Boundary Condition provides some unexpected constraints on the triangle invariants of Hitchin characters, as well as on their shearing classes. The following two statements are abbreviated expressions of more specific computations given in \S \ref{subsect:ConstraintInvariantsBis}. These restrictions are somewhat surprising when one considers the relatively large dimension $2(g-1)(n^2-1)$ of $\Hit(S)$. 

\begin{prop}
\label{prop:RestrictTriangleInvIntro}
An element $\tau \in \R^{2(g-1)(n-1)(n-2)}$ is the triangle invariant $\tau^\rho$ of a Hitchin character $\rho \in \Hit(S)$ if and only if it belongs to a certain explicit subspace of codimension $\lfloor \frac {n-1}2 \rfloor$ of $\R^{2(g-1)(n-1)(n-2)}$. 
\end{prop}

\begin{prop}
\label{prop:RestrictShearingCycleIntro}
A relative homology class $[\sigma] \in H_1(\widehat U, \delv \widehat U; \R^{n-1})$ is the shearing class $[\sigma^\rho]$ of a Hitchin character $\rho \in \Hit(S)$ if and only if it belongs to a certain open convex polyhedral cone in an explicit linear subspace of dimension $6(g-1)(3n-7)$ if $n>3$, of dimension $16(g-1)$ if $n=3$, and of dimension $6(g-1)$ if $n=2$. 
\end{prop}

The dimensions in Proposition~\ref{prop:RestrictShearingCycleIntro} should be compared to the dimension $18(g-1)(n-1)$ of the twisted homology space $ H_1(U,  \delv U; \widetilde\R^{n-1})$, consisting of those  $\alpha \in H_1(\widehat U,  \delv \widehat U; \R^{n-1})$ such that $\iota_* (\alpha)= - \overline{\alpha}$.

At first, the relative homology group $H_1(\widehat U, \delv \widehat U; \R^{n-1})$ of a train track neighborhood $U$ may not appear very natural. In fact, although we  decided to privilege this more familiar point of view in this introduction, it occurs as a space $\CC(\widehat\lambda, \slits; \R^{n-1}) $ of tangent cycles for the orientation cover $\widehat\lambda$ relative to its slits, where the slits of $\widehat\lambda$ are lifts of the spikes of the complement $S-\lambda$; Proposition~\ref{prop:RelativeHomologyTangentCycles} then provides an isomorphism $\CC(\widehat\lambda, \slits; \R^{n-1}) \cong H_1(\widehat U, \delv \widehat U; \R^{n-1})$. A relative tangent cycle $\alpha\in \CC(\widehat\lambda, \slits; \R^{n-1})$ assigns a vector $\alpha(k) \in \R^{n-1}$ to each arc $k$ transverse to $\widehat \lambda$, in a quasi-additive way: If $k$ is split into two  subarcs $k_1$ and $k_2$, then $\alpha(k)$ is equal to the sum of $\alpha(k_1)$, $\alpha(k_2)$ and of a correction factor depending on the slit of $\widehat\lambda $ facing the point $k_1\cap k_2$ along which $k$ was split. In particular, $\CC(\widehat\lambda, \slits; \R^{n-1})$ depends only on the maximal geodesic lamination $\lambda$, and not on the train track neighborhood $U$. 

The lack of additivity of a relative tangent cycle $\alpha \in \CC(\widehat\lambda, \slits; \R^{n-1})$ has a nice expression in terms of the boundary map $\partial \colon H_1(\widehat U, \delv \widehat U; \R^{n-1}) \to H_0( \delv \widehat U; \R^{n-1})$, and is at the basis of  the Shearing Cycle Boundary Condition of Proposition~\ref{prop:ShearingCycleBdryConditionIntro}. In the classical case where $n=2$, the Shearing Cycle Boundary Condition says that the shearing class $[\sigma^\rho] \in H_1(\widehat U, \delv \widehat U; \R^{n-1}) $ has boundary 0, and in particular that  the corresponding tangent cycle $[\sigma^\rho] \in \CC(\widehat \lambda, \slits; \R)$ is additive with no correction factors; such objects were called ``transverse cocycles'' in \cite{Bon97a, Bon96}. 

This point of view enables us to shed some light on the Positive Intersection Condition of Proposition~\ref{prop:PositiveIntersectionIntro}. Given a Hitchin character $\rho \in \Hit(S)$, Labourie \cite{Lab1} shows that for every  nontrivial $\gamma \in \pi_1(S)$ the matrix $\rho(\gamma)\in \PSL$ is diagonalizable, and that its eigenvalues $m_a^\rho(\gamma)$ can be ordered in such a way that $|m_1^\rho(\gamma)|> |m_2^\rho(\gamma)|> \dots >|m_n^\rho(\gamma)| $. If we define $\ell^\rho (\gamma) \in \R^{n-1}$ by the property that its $a$--th coordinate is $\ell^\rho_a (\gamma) =\log \frac{|m_a^\rho(\gamma)|}{ |m_{a+1}^\rho(\gamma)|}$, the second author showed in  \cite{Dre1} that this formula admits a continuous linear extension $\ell^\rho  \colon \CH(S) \to \R^{n-1}$ to the space $\CH(S)$ of H\"older geodesic currents of $S$, a topological vector space that contains all conjugacy classes of $\pi_1(S)$ in a natural way; this continuous extension $\ell^\rho  \colon \CH(S) \to \R^{n-1}$ is unique on the subspaces of $\CH(S)$ that are of interest to us in this paper (see Remark~\ref{rem:LengthHolderGeodCurrentsUnique?}). 

In particular, an (additive) tangent cycle $\alpha \in \CC(\widehat \lambda; \R)$ defines a H\"older geodesic current $\alpha \in \CH(S)$ (see \cite{Bon97a}), and we can restrict  the length function of \cite{Dre1} to $\ell^\rho  \colon \CC(\widehat \lambda; \R) \to \R^{n-1}$. 

The following result, proved as Theorem~\ref{thm:ShearingAndLength} in \S \ref{subsect:ShearingAndLength}, relates the length vector $\ell^\rho(\alpha) \in \R^{n-1}$ to the shearing class $[\sigma^\rho] \in \CC(\widehat\lambda, \slits; \R^{n-1}) \cong H_1(\widehat U, \delv \widehat U; \R^{n-1})$. 

\begin{thm}[Length and Intersection Formula]
\label{thm:LengthIntersectionInto}
If  $[\sigma^\rho] \in \CC(\widehat\lambda, \slits; \R^{n-1}) \cong H_1(\widehat U, \delv \widehat U; \R^{n-1})$ is the shearing cycle of a Hitchin character $\rho \in \Hit(S)$, and if  $\alpha \in \CC(\widehat\lambda; \R)  \cong H_1(\widehat U; \R)$ is a  tangent cycle for the orientation cover $\widehat\lambda$, then
$$
\ell_a^\rho(\alpha) = [\alpha]\cdot [\sigma^\rho] \in \R^{n-1}
$$
is the algebraic intersection vector of the  homology classes $[\alpha]\in H_1(\widehat U; \R)$ and $[\sigma^\rho ] \in H_1(\widehat U, \delv \widehat U; \R^{n-1})$ in the train track neighborhood $\widehat U$ of $\widehat\lambda$. 
\end{thm}

In the special case where $\alpha$ is a transverse measure $\mu$ for $\widehat\lambda$, the Positive Intersection Condition of Proposition~\ref{prop:PositiveIntersectionIntro} is then equivalent to the property that all coordinates of the vector $\ell^\rho(\mu)$ are positive. In this version, this statement is an immediate consequence of the Anosov Property that is central to \cite{Lab1} (see Proposition~\ref{prop:MeasureHasPositiveLengths}). 

The article concludes, in \S\S \ref{bigsect:PseudoAnosov} and \ref{bigsect:LengthMeasLam}, with  two brief applications of Theorems~\ref{thm:ParametrizationIntro} and \ref{thm:LengthIntersectionInto}. The first one is concerned with the dynamics of the action of a pseudo-Anosov diffeomorphism $\phi \colon S \to S$ on the Hitchin component $\Hit(S)$; applying the parametrization of Theorem~\ref{thm:ParametrizationIntro} to the case of a maximal geodesic lamination $\lambda$ containing the stable lamination of $\phi$ shows that the dynamics of the action of $\phi$ on $\Hit(S)$ are concentrated on submanifolds of $\Hit(S)$ of relatively large codimension.  The second application considers the restriction of the length function $\ell^\rho  \colon \CH(S) \to \R^{n-1}$ to Thurston's space $\ML(S)$ of measured laminations on $S$;  a consequence of Theorem~\ref{thm:LengthIntersectionInto} is that, at each $\alpha \in \ML(S)$, the  tangent map $T_\alpha \ell^\rho  \colon T_\alpha \ML(S) \to \R^{n-1}$ is linear on each face of the piecewise linear structure of $\ML(S)$. 

These results can be put in a broader perspective. Indeed, the properties of the Hitchin component remain valid when the Lie group $\PSL$ is replaced by any split real  algebraic group $G$ \cite{Hit, Lab1, FoG1}. In this more general framework, our triangle invariant $\tau^\rho$  associates to each component of $S-\lambda$ a positive  triple in the flag space $B\backslash G$, where $B$ is a Borel subgroup. The shearing class is now a relative homology class $[\sigma^\rho] \in H_1(\widehat U, \delv \widehat U; \mathfrak h)$ valued in the Cartan algebra $\mathfrak h$ of $G$, and equivariant with respect to the covering involution $\iota \colon \widehat U \to \widehat U$ and to minus the opposition involution of $\mathfrak h$. The 
Shearing Cycle Boundary Condition then states that the boundary $\partial [\sigma^\rho] \in H_0 (\delv \widehat U; \mathfrak h)$ is completely determined by the triangle invariant $\tau^\rho \in (B\backslash G)^{4(g-1)}$, while the Positive Intersection Condition requires that the algebraic intersection vector $[\mu]\cdot [\sigma^\rho] \in \mathfrak h$ belong to the principal Weyl chamber of $\mathfrak h$. The output of these constructions is perhaps not as explicit as in the case of $\PSL$, but extending the proofs to this more general context is only a matter of using the right vocabulary.

\medskip\noindent
\textbf{Acknowledgement:} The authors are very pleased to  acknowledge very helpful conversations with Antonin Guilloux and Anne Parreau, at a time when they (the authors) were very confused. They are also grateful to Giuseppe Martone for many useful comments on the manuscript.

\section{Generic configurations of flags}
\label{bigsect:FlagConf}

Flags in $\R^n$  play a fundamental r\^ole in our construction of invariants of Hitchin characters. This section is devoted to certain invariants of finite families of flags, borrowed from \cite{FoG1}.  

\subsection{Flags}
\label{subsect:Flags}

A \emph{flag} in $\R^n$ is a family $F$ of nested linear subspaces $F^{(0)} \subset F^{(1)}\subset \dots\subset F^{(n-1)} \subset F^{(n)}$ of $\R^n$ where each $F^{(a)}$ has dimension $a$. 

A pair of flags $(E,F)$ is \emph{generic} if every subspace $E^{(a)}$ of $E$ is transverse to every subspace $F^{(b)}$ of $F$. This is equivalent to the property that $E^{(a)}\cap F^{(n-a)}=0$ for every $a$. 

Similarly, a triple of flags $(E,F,G)$ is \emph{generic} if each triple of subspaces $E^{(a)}$, $F^{(b)}$, $G^{(c)}$, respectively  in $E$, $F$, $G$, meets transversely. Again, this is equivalent to the property that $E^{(a)} \cap F^{(b)} \cap G^{(c)} =0$ for every $a$, $b$, $c$ with $a+b+c=n$. 

\subsection{Wedge-product invariants of generic flag triples}
\label{subsect:WedgeInvariants}
Elementary linear algebra shows that, for any two generic flag pairs $(E,F)$ and $(E', F')$, there is a linear isomorphism $\R^n \to \R^n$ sending $E$ to $E'$ and $F$ to $F'$. However, the same is not true for generic flag triples. Indeed, there is a whole moduli space of generic flag triples modulo the action of $\GL$, and this moduli space can be parametrized by invariants that we now describe. These invariants are expressed in terms of the exterior algebra $\Lambda^\bullet (\R^n)$ of $\R^n$. 

Consider the discrete triangle
$$
\Theta_n = \{ (a,b,c) \in \Z^3; a+b+c=n \text{ and } a,b,c \geq 0 \}. 
$$
represented in Figure~\ref{fig:DiscTriangle}.  

\begin{figure}[htbp]

\SetLabels
\R(- 0*-.15) $(n,0,0) $ \\
(.5* 1.08) $ (0,n,0)$ \\
\L( 1* -.15) $(0,0,n) $ \\
( * ) $ $ \\
( * ) $ $ \\
\endSetLabels
\vskip 10pt
\centerline{\AffixLabels{\includegraphics{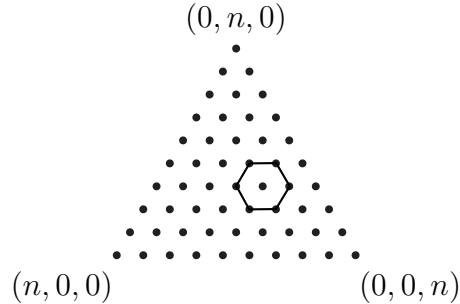}}}
\vskip 10pt
\caption{The discrete triangle $\Theta_n$, with a hexagon cycle}
\label{fig:DiscTriangle}
\end{figure}

A function $\phi \colon \Theta_n \to \Z$ is \emph{balanced} if, for every $a_0$, $b_0$, $c_0$,
$$
\sum_{(a_0,b,c) \in \Theta_n} \phi(a_0, b,c) = \sum_{(a,b_0,c) \in \Theta_n} \phi(a,b_0,c) = \sum_{(a,b,c_0) \in \Theta_n} \phi(a,b,c_0) = 0,
$$
namely if the sum of the $\phi(a,b,c)$ over each line  parallel to one side of the triangle $\Theta_n$ is equal to 0.

Such a balanced function $\phi$ defines an invariant of a generic flag triple $(E,F,G)$ as follows. For each $a$, $b$, $c$ between $0$ and $n$, the spaces $ \Lambda^a\bigl(E^{(a)}\bigr)$, $\Lambda^b\bigl(F^{(b)}\bigr)$ and $\Lambda^c\bigl(G^{(c)}\bigr)$ are each isomorphic to $\R$. Choose non-zero elements   $ e^{(a)}  \in \Lambda^a\bigl(E^{(a)}\bigr)$,  $ f^{(b)}  \in \Lambda^b\bigr(F^{(b)}\bigr)$ and  $ g^{(c)} \in \Lambda^c\bigl(G^{(c)}\bigr)$. We will use the same letters to denote their images $ e^{(a)} \in \Lambda^a(\R^n)$, $ f^{(b)} \in \Lambda^b(\R^n)$ and $ g^{(c)} \in \Lambda^c (\R^n )$. We then define
$$
\Phi(E,F,G) = \prod_{(a,b,c) \in \Theta_n}  \bigl(  e^{(a)} \wedge  f^{(b)} \wedge  g^{(c)}  \bigr)^{\phi(a, b, c)}  \in \R,
$$
where we choose an isomorphism $\Lambda^n(\R^n) \cong \R$ to interpret each term in the product as a real number. The fact that the flag triple is generic guarantees that these numbers are non-zero, while the property that $\phi$ is balanced is exactly what is needed to make sure that this product is independent of the choices of the elements $ e^{(a)}  \in \Lambda^a\bigl(E^{(a)}\bigr)$,  $ f^{(b)}  \in \Lambda^b\bigr(F^{(b)}\bigr)$ and  $ g^{(c)} \in \Lambda^c\bigl(G^{(c)}\bigr)$ and of the isomorphism  $\Lambda^n(\R^n) \cong \R$.  We say that $\Phi$ is the \emph{wedge-product  invariant} of generic flag triples associated to the balanced function $\phi\colon \Theta\to \Z$. 

We now consider a fundamental special case. For  $a$, $b$, $c\geq 1$ with $a+b+c=n$, namely for a point $(a,b,c)$ in the interior of the triangle $\Theta_n$, the \emph{$(a,b,c)$--hexagon cycle} is the balanced function $\phi_{abc} \colon \Theta_n \to \Z$ defined by
$$
\phi_{abc} = \delta_{(a+1, b, c-1)} - \delta_{(a-1, b, c+1)} + \delta_{(a, b-1, c+1)} -\delta_{(a, b+1, c-1)} + \delta_{(a-1, b+1, c)} -\delta_{(a+1, b-1, c)},
$$
where $\delta_{(a,b,c)} \colon \Theta_n \to \Z$ denotes the Kronecker function such that $\delta_{(a,b,c)}(a', b', c') =1$ if $(a,b,c)=(a', b', c')$ and $\delta_{(a,b,c)}(a', b', c') =0$ otherwise. The terminology is explained by the fact that the support of $\phi_{abc} $ is a small hexagon in the discrete triangle $\Theta_n$, centered at the point $(a,b,c)$; see Figure~\ref{fig:DiscTriangle} for the case where $n=9$ and $(a,b,c)=(2,3,4)$.  The  wedge-product  invariant associated to the hexagon cycle $\phi_{abc} $ is
the \emph{$(a,b,c)$--triple ratio}  
\begin{align*}
T_{abc} (E,F,G) &=
\frac
{ e^{(a+1)} \wedge  f^{(b)} \wedge  g^{(c-1)}}
{ e^{(a-1)} \wedge  f^{(b)} \wedge  g^{(c+1)}}
\\
&\qquad\qquad\qquad\qquad
\frac
{ e^{(a)} \wedge  f^{(b-1)} \wedge  g^{(c+1)}}
{ e^{(a)} \wedge  f^{(b+1)} \wedge  g^{(c-1)}}\  \ 
\frac
{ e^{(a-1)} \wedge  f^{(b+1)} \wedge  g^{(c)}}
{ e^{(a+1)} \wedge  f^{(b-1)} \wedge  g^{(c)}}.
\end{align*}

Note the elementary property of triple ratios under permutation of the flags. 

\begin{lem}
\label{lem:SymmetriesTripleRatios}
\pushQED{\qed}
\begin{equation*}
 T_{abc} (E,F,G ) =  T_{bca} (F,G, E)= T_{bac} (F,E,G)^{-1}. \qedhere
\end{equation*}
\end{lem}

The natural action of the linear group $\GL$ on the flag variety $\Flag$ descends to an action of the projective linear group $\PGL$, quotient of $\GL$ by its center $\bigl(\R- \{0\}\bigr) \Id$ consisting of all non-zero scalar multiples of the identity. Note that the projective special linear group $\PSL$ is equal to $\PGL$ if $n$ is odd, and is an index 2 subgroup of $\PGL$ otherwise.

\begin{prop}
\label{prop:TripRatiosDetermineFlagTriples}
Two generic flag triples $(E,F,G)$ and $(E', F', G')$ are equivalent under the action of $\PGL$ if and only if $ T_{abc} (E,F,G)=T_{abc} (E',F',G')$ for every $a$, $b$, $c\geq 1$ with $a+b+c=n$. 

In addition, for any set of non-zero numbers $t_{abc}\in \R - \{0\}$, there exists a generic flag triple $(E,F,G)$ such that $ T_{abc} (E,F,G)=t_{abc}$ for every $a$, $b$, $c\geq 1$ with $a+b+c=n$. 

\end{prop}

\begin{proof}
See \cite[\S9]{FoG1}.
\end{proof}

In particular, the moduli space of generic flag triples $(E,F,G)$  under the action of $\PGL$  is homeomorphic to $\bigl (\R-\{0\} \bigr)^{\frac{(n-1)(n-2)}2} $.

Corollary~\ref{cor:WedgeInvarExpressedTripleRatios} below partially  accounts for the important r\^ole played by the triple ratios $T_{abc}$ in Proposition~\ref{prop:TripRatiosDetermineFlagTriples}.  We will not really need this property, but it explains why we will always be able to express in terms of triple ratios $T_{abc}$ the various wedge-product invariants that we will encounter in the paper. 

\begin{lem}
\label{lem:BalancedFunctionsHexagonCycles}
The hexagon cycles $\{ \phi_{abc}; a,b,c\geq 1, a+b+c=n \}$ form a basis for the free abelian group consisting of all balanced function $\phi \colon \Theta_n \to \Z$. 
\end{lem}

\begin{proof}
The proof is elementary, by induction on $n$. 
\end{proof}
Lemma~\ref{lem:BalancedFunctionsHexagonCycles} immediately implies:

\begin{cor}
\label{cor:WedgeInvarExpressedTripleRatios}
Every wedge-product  invariant can be uniquely expressed as a product of integer powers of triple ratios. \qed
\end{cor}

\subsection{Quadruple ratios}
\label{subsect:QuadRatios}
In addition to triple ratios, the following wedge-product  invariants of generic flag triples will play a very important r\^ole in this article. 

For $a=1$, 2, \dots, $n-1$, the \emph{$a$--th quadruple ratio} of the generic flag triple $(E,F,G)$ is the wedge-product invariant
\begin{align*}
Q_a(E, F, G)
= 
\frac
{e^{(a-1)}\wedge f^{(n-a)} \wedge g^{(1)}}
{e^{(a)}\wedge f^{(n-a-1)} \wedge g^{(1)} }\,
&\frac
{e^{(a)}\wedge f^{(1)} \wedge g^{(n-a-1)}}
{e^{(a-1)}\wedge f^{(1)} \wedge g^{(n-a)}}
  \\
&\qquad\qquad\qquad 
 \frac
{ e^{(a+1)}\wedge f^{(n-a-1)} }
{ e^{(a+1)}\wedge g^{(n-a-1)} }
\,
\frac
{e^{(a)}\wedge g^{(n-a)} }
{e^{(a)}\wedge f^{(n-a)}  }
 \end{align*}
 where, as usual, we consider arbitrary non-zero elements $ e^{(b)}  \in \Lambda^b\bigl(E^{(b)}\bigr)$,  $ f^{(b)}  \in \Lambda^b\bigr(F^{(b)}\bigr)$ and  $ g^{(b)} \in \Lambda^b \bigl(G^{(b)}\bigr)$, and where the ratios are  computed in $\Lambda^n (\R^n) \cong \R$. 

 Note that $
 Q_a(E,G,F) = Q_a(E,F,G)^{-1}
 $, but that this quadruple ratio usually does not behave well under the other permutations of the flags $E$, $F$ and $G$, as  $E$ plays a special r\^ole in $Q_a(E,F,G)$.  
 
 For this wedge-product invariant, we can explicitly determine the formula predicted by Corollary~\ref{cor:WedgeInvarExpressedTripleRatios}. 
  
\begin{lem}
\label{lem:ExpressQuadrupleRatioTripleRatios}
For $a=1$, $2$, \dots, $n-1$,
$$
Q_a(E,F,G) = \prod_{b+c=n-a} T_{abc}(E,F,G)
$$
where the product is over all integers $b$, $c\geq 1$ with $b+c=n-a$. In particular, $Q_{n-1}(E,F,G) = 1$ and $Q_{n-2}(E,F,G) = T_{(n-2)11}(E,F,G)$. 
\end{lem}
\begin{proof}
When computing the right hand side of the equation, most terms $e^{(a')} \wedge f^{(b')} \wedge g^{(c')}$  cancel out and we are left with the eight terms of $Q_a(E,F,G)$.
\end{proof}

\subsection{Double ratios}
\label{subsect:DoubleRatios}
We now consider quadruples $(E,F,G,H)$ of flags $E$, $F$, $G$, $H\in \Flag$.  Such a flag quadruple is \emph{generic} if each quadruple of subspaces $E^{(a)}$, $F^{(b)}$, $G^{(c)}$, $H^{(d)}$ meets transversely. As usual, we can restrict attention to the cases where $a+b+c+d=n$.

For $1\leq a\leq n-1$, the \emph{$a$--th double ratio} of the generic flag quadruple $(E,F,G,H)$ is
$$
D_a (E,F,G,H) = -
\frac
{ e^{(a)} \wedge  f^{(n-a-1)}\wedge  g^{(1)}} 
{ e^{(a)} \wedge  f^{(n-a-1)}\wedge  h^{(1)}}
\frac
{ e^{(a-1)} \wedge  f^{(n-a)}\wedge  h^{(1)}}
{ e^{(a-1)}  \wedge  f^{(n-a)}\wedge  g^{(1)}} 
$$
where we choose arbitrary non-zero elements $ e^{(a')} \in \Lambda^{a'}(E^{(a')})$, $f^{(b')} \in \Lambda^1(F^{(b')})$, $g^{(1)} \in \Lambda^{1}(G^{(1)})$ and $h^{(1)}\in \Lambda^1(H^{(1)})$. As usual,  $D_a (E,F,G,H) $ is independent of these choices. 

The following computation gives a better feeling of what is actually measured by this double ratio. 

\begin{lem}
\label{lem:ComputeDoubleRatios}
For a generic flag quadruple $(E,F,G,H)$, consider the decomposition $\R^n = \bigoplus_{a=1}^n L_a$ where $L_a = E^{(a)} \cap F^{(n-a+1)}$. For arbitrary non-zero vectors $g\in G^{(1)}$ and $h\in H^{(1)}$, let $g_a$, $h_a\in L_a$ be the respective projections of $g$ and $h$ to the line $L_a$ parallel to the other lines $L_b$ with $b\neq a$. Then
\pushQED{\qed}
\begin{equation*}
D_a (E,F,G,H) = -
\frac
{ g_{a+1}} 
{ h_{a+1}}
\frac
{ h_a}
{ g_a}
\qedhere
\end{equation*}
where the ratios $\frac{g_b}{h_b}\in \R$ are measured in the lines $L_b$. 
\end{lem}
 Note that $D_a (E,F,G,H)$ does not really depend on the whole flags $G$ and $H$, but only on the lines $G^{(1)}$ and $H^{(1)}$. The following elementary properties indicate how it behaves under transposition of $E$ and $F$, or of $G$ and $H$. 
\begin{lem}
\label{lem:RelationsDoubleRatios}
\pushQED{\qed}
\begin{align*}
 D_a(E,F,H, G) &= D_a(E,F,G,H)^{-1} \\ 
  D_a(F,E,G,H) &= D_{n-a}(E,F,G,H)^{-1},\\
 \text{and }
 D_a(E,F,G,K) &= -D_a(E,F,G,H) D_a(E,F,H,K).  \qedhere
\end{align*}
\end{lem}
The minus sign in the definition of $D_a (E,F,G,H)$ is justified by the positivity property of the next section, and in particular by Proposition~\ref{prop:DihedralPermutPreservesPositivity}. 

\subsection{Positivity} 
\label{subsect:Positivity}
An ordered family of flags $(E_1, E_2, \dots, E_m)\in \Flag^m$ is \emph{positive} if:
\begin{enumerate}
\item for every distinct $i$, $j$, $k$ and for every $a$, $b$, $c\geq 1$ with $a+b+c=n$, the triple ratio $T_{abc}(E_i, E_j, E_k)$ is positive. 
\item for every distinct $i$, $j$, $k$, $l$ with $ i < k< j<l $ or $ k<i < l< j$, and for every $1\leq a\leq n-1$, the double ratio $D_a(E_i, E_j, E_k, E_l)$ is positive. 
\end{enumerate}

Fock and Goncharov \cite[\S5]{FoG1} give a much more conceptual definition of positivity, building on earlier work of Lusztig \cite{Lusz1, Lusz2}. In particular, they prove the following result.

\begin{prop}[{\cite{FoG1}}]
\label{prop:DihedralPermutPreservesPositivity}
If the flag $m$--tuple $(E_1, E_2, \dots, E_m)$ is positive, any flag $m$--tuple obtained by dihedral permutation of the $E_i$ is also positive.
\qed
\end{prop}

Recall that a \emph{dihedral permutation} is, either a cyclic permutation, or the composition of the order reversal $(E_1, E_2, \dots, E_m) \mapsto (E_m, E_{m-1}, \dots, E_1)$ with a cyclic permutation. 

\section{Geodesic laminations}
\label{bigsect:GeodLam}

Geodesic laminations are a now very classical tool in 2--dimensional topology and geometry. They occur in many different contexts, for instance when one takes limits of sequences of simple closed curves. We state here a few basic definitions and facts, and refer to  \cite{Thu0, CasBlei, PenH, BonLam} for proofs and background.

To define geodesic laminations, one first chooses a metric $m$ of negative curvature on the surface $S$. 

An \emph{$m$--geodesic lamination} is a closed subset $\lambda \subset S$ that can be decomposed as a disjoint union of simple complete $m$--geodesics, called its \emph{leaves}. Recall that a geodesic is \emph{complete} if it cannot be extended to a longer geodesic, and it is \emph{simple} if it has no transverse  self-intersection point. The leaves of a geodesic laminations can be closed or bi-infinite. A geodesic lamination can have finitely many leaves (as in the case considered in \cite{BonDre}), or uncountably many leaves.

An $m$--geodesic lamination has measure 0, and in fact Hausdorff dimension 1 \cite{BirSer}, and its decomposition as a union of leaves is unique. The complement $S-\lambda$ of an $m$--geodesic lamination $\lambda$ is a surface of finite topological type, bounded by finitely many leaves of $\lambda$. The completion of $S-\lambda$ for the path metric induced by $m$ is a finite area surface with geodesic boundary; it is the union of a compact part and of finitely many spikes homeomorphic to $[0,1] \times \left[ 0, \infty \right[$, where $\{0,1\} \times \left[ 0, \infty \right[$ is contained in two  leaves of $\lambda$. The width of these spikes decreases exponentially in the sense that the parametrization by $[0,1] \times \left[ 0, \infty \right[$ can be chosen so that its restriction to each $\{ x\} \times \left[ 0, \infty \right[$ has speed 1 and so that the length of each arc $[0,1] \times \{ t\}$ decreases exponentially with $t$.  

Because the leaves of $\lambda$ are disjoint, every point of $S$ has a neighborhood $U$ homeomorphic to $[0,1] \times [0,1]$ for which the intersection $U\in \lambda$ corresponds to $K\times[0,1]$ for some totally disconnected compact subset $K\subset [0,1]$; beware that, in general, the homeomorphism cannot be made differentiable, only H\"older bicontinuous.  

We will make heavy use of \emph{transverse arcs} for $\lambda$. These are arcs differentiably immersed in $S$ that are transverse to the leaves of $\lambda$. In addition, we require that the endpoints of such a transverse arc be disjoint from $\lambda$.

The notion of geodesic lamination is  independent of the choice of the negatively curved metric $m$ in the  sense that, if $m'$ is another negatively curved metric on $S$, there is a natural one-to-one correspondence between $m$--geodesic laminations and $m'$--geodesic laminations. 

A geodesic lamination $\lambda$ is \emph{maximal} if it is contained in no other geodesic lamination. This is equivalent to the property that each component of its  complement $S-\lambda$ is a triangle, bounded by three infinite leaves of $\lambda$ and containing three spikes of $S-\lambda$. If the surface $S$ has genus $g$, an Euler characteristic argument shows that the number of triangle components of the complement $S-\lambda$ of a maximal geodesic lamination is equal to $4(g-1)$. 

Every geodesic lamination is contained in a maximal geodesic lamination. 

We can think of maximal geodesic laminations as some kind of triangulations of the surface $S$, where the edges are geodesic and where the vertices have been pushed to infinity. This point of view explains why maximal geodesic laminations are powerful tools for many problems, such as the ones considered in the current article.

\section{Triangle invariants}
\label{bigsect:TriangleInv}
Let $\rho \colon \pi_1(S) \to \PSL$ be a Hitchin homomorphism. We will use a maximal geodesic lamination $\lambda$ to construct invariants of the corresponding character $\rho\in \Hit(S)$. 

\subsection{The flag curve} 
\label{subsect:FlagCurve}
The key to the definition of these invariants is the following construction of Labourie \cite{Lab1}. 

Let $T^1S$ and $T^1\widetilde S$ be   the unit tangent bundles of the surface $S$ and of its universal cover $\widetilde S$, respectively.  For convenience, lift the homomorphism  $\rho \colon \pi_1(S) \to \PSL$ to a homomorphism $\rho' \colon \pi_1(S) \to \SL$. The fact that such a lift exists is classical when $n=2$, and therefore when $\rho \colon \pi_1(S) \to \PSL$ comes from a discrete representation $\pi_1(S) \to \mathrm{PSL}_2(\R)$; the existence of the lift in the general case follows by connectedness of the Hitchin component $\Hit(S)$, and by homotopy invariance of the obstruction to lift.  We can then consider the twisted product
 $$
T^1S \times_{\rho'} \R^n = (T^1\widetilde S \times \R^n )/\pi_1(S)
$$
where the fundamental group $\pi_1(S)$ acts on $T^1\widetilde S$ by its usual action on the universal cover $\widetilde S$, and acts on $\R^n$ by $\rho'$. The natural projection $T^1S \times_{\rho'} \R^n \to T^1S$ presents $T^1S \times_{\rho'} \R^n$ as a vector bundle over $T^1S$ with fiber $\R^n$. 

Endow the surface $S$ with an arbitrary metric of negative curvature. This defines a circle at infinity $\partial_\infty \widetilde S$ for the universal cover $\widetilde S$, and a geodesic flow on the unit tangent bundle $T^1S$. It is well known (see for instance \cite{Gro, BridH, GhysHarp}) that these objects are actually independent of the choice of the negatively curved metric, at least if we do not care about the actual parametrization of the geodesic flow (which is the case here).

The geodesic flow $(g_t)_{t\in \R}$ of $T^1S$ has a natural flat lift  to a flow $(G_t)_{t\in \R}$ on the total space $T^1S \times_{\rho'} \R^n$. The flatness property here just means that the flow $(G_t)_{t\in \R}$ is the projection of the flow $(\widetilde{G}_t)_{t\in \R}$ on $T^1\widetilde S \times \R^n$ that acts by the geodesic flow $(\widetilde{g}_t)_{t\in \R}$ of $T^1\widetilde S$ on the first factor, and by the identity $\Id_{\R^n}$ on the  second factor. 

Endow each fiber of the vector bundle $T^1S \times_{\rho'} \R^n \to T^1S$ with a norm $\Vert \ \Vert$  depending continuously on the corresponding point of $T^1S$. 

\begin{thm}[Labourie {\cite{Lab1}}]
\label{thm:AnosovProperty}
If $\rho \colon \pi_1(S) \to \PSL$ is a Hitchin homomorphism, the  vector bundle $T^1S \times_{\rho'} \R^n \to T^1S$ admits a unique decomposition as a  direct sum $L_1 \oplus L_2 \oplus \dots \oplus L_n$ of $n$ line subbundles $L_a \to T^1S$ such that:
\begin{enumerate}
\item each line bundle $L_a$ is invariant under the lift $(G_t)_{t\in \R}$ of the geodesic flow;
\item for every $a>b$, there exist constants $A_{ab}$, $B_{ab}>0$ such that, for every $v_a\in L_a$ and $v_b \in L_b$ in the same fiber of $T^1S \times_{\rho'} \R^n$ and for every $t\geq0$, \pushQED{\qed}
\begin{equation*}
\frac{\Vert G_t(v_b)\Vert} {\Vert v_b \Vert} \leq
A_{ab}
\frac{\Vert G_t(v_a) \Vert} {\Vert v_a\Vert}  \E^{-B_{ab}t}. \qedhere
\end{equation*}

\end{enumerate}
\end{thm}

The second property is clearly independent of the choice of the norm $\Vert \ \Vert$. It is referred to as the \emph{Anosov property} of the Hitchin homomorphism $\rho$. This relative property does not say anything about whether the flow $(G_t)_{t \in \R^n}$ expands or contracts the fibers of any individual subbundle $L_a$ but states that, when $a<b$, the flow $(G_t)_{t \in \R^n}$ contracts the fibers of $L_b$ much more than those of $L_a$. Writing this in a more intrinsic way, this means that $(G_t)_{t\in \R^n}$ induces on the line bundle $\mathrm{Hom}(L_a, L_b)$ a flow that is uniformly contracting when $a>b$.

Lift the subbundles $L_a$ of $T^1S \times_{\rho'} \R^n = (T^1\widetilde S \times \R^n)/\pi_1(S)$ to subbundles $\widetilde L_a$ of $T^1\widetilde S \times \R^n$. 
Because the line subbundles $L_a$ are invariant under the lift $(G_t)_{t\in \R}$ of the geodesic flow, the fiber of $\widetilde L_a$ over $\widetilde x \in \widetilde S$ is of the form $\{\widetilde x \} \times \widetilde L_a(g)$ for some line $\widetilde L_a(g) \subset \R^n$ depending only on the orbit $g$ of $\widetilde x$ for the geodesic flow of $T^1\widetilde S$. 

The line $\widetilde L_a(g)\subset \R^n $ depends on the orbit $g$ of the geodesic flow of $T^1 \widetilde S$ or, equivalently, on the corresponding oriented geodesic $g$ of $\widetilde S$. The Anosov property has the following relatively easy consequence. Define a flag $E(g)\in \Flag$ by the property that $E(g)^{(a)} = \widetilde L_1(g) \oplus\widetilde  L_2(g) \oplus \dots \oplus \widetilde L_a(g)$; then $E(g)$ depends only on the positive endpoint of $g$. More precisely:

\begin{prop}[Labourie \cite{Lab1}]
\label{prop:FlagCurve}
For a Hitchin homomorphism $\rho \colon \pi_1(S) \to \PSL$, there exists a unique map $\F_\rho \colon \partial_\infty \widetilde S \to \Flag$ such that
\begin{enumerate}
\item $\F_\rho$ is H\"older continuous;
\item for every oriented geodesic $g$ of $\widetilde S$ with positive  endpoint $\widetilde x_+ \in \partial_\infty \widetilde S$, the image $ \F_\rho(\widetilde x_+)$ is equal to the flag $ E(g)$ defined above;
\end{enumerate}
In addition, $\F_\rho$ is $\rho$--equivariant in the sense that $\F_\rho(\gamma \widetilde x) = \rho(\gamma)\bigl(\F_\rho(\widetilde x)\bigr)$ for every $\widetilde x \in \partial_\infty \widetilde S$ and $\gamma \in \pi_1(S)$.   \qed
\end{prop}

By definition, this map  $\F_\rho \colon \partial_\infty \widetilde S \to \Flag$ is the \emph{flag curve} of the Hitchin homomorphism $\rho \colon \pi_1(S) \to \PSL$. It is independent of the choice of the lift $\rho '\colon \pi_1(S) \to \SL$ of $\rho \colon \pi_1(S) \to \PSL$, and of the negatively curved metric on $S$ used to define  the geodesic flow of the unit tangent bundle $T^1S$. 

The flag curve $\F_\rho$ has the following important positivity property.

\begin{thm}
[Fock-Goncharov {\cite{FoG1}}]
\label{thm:FlagCurvePositive}
For every finite set of distinct points $x_1$, $x_2$, \dots, $x_k \in \partial_\infty \widetilde S$ occurring in this order on the circle at infinity $\partial_\infty \widetilde S$, the flag $k$--tuple $\bigl( \F_\rho(x_1),  \F_\rho(x_2), \dots ,  \F_\rho(x_k) \bigr)$ is positive in the sense of {\upshape \S \ref{subsect:Positivity}}. \qed
\end{thm}

\subsection{Triangle invariants of Hitchin characters} 
\label{subsect:TriangleInv}
We now define a first set of invariants for the Hitchin character represented by a homomorphism $\rho \colon \pi_1(S) \to \PSL$.

The complement of the maximal geodesic lamination $\lambda$ consists of finitely many infinite triangles $T_1$, $T_2$, \dots, $T_m$, each with three spikes. 

Consider such a triangle  component $T$ of $S-\lambda$, and select one of its spikes $s$.   Lift $T$ to  an ideal triangle $\widetilde T$ in the universal cover $\widetilde S$, and let $\widetilde s$  be the spike of $\widetilde T$ corresponding to $s$. The spike $\widetilde s$ uniquely determines a point of the circle at infinity $ \partial_\infty \widetilde S$, which we will also denote by $\widetilde s $.   

Label the spikes of $ T$ as $ s$, $ s'$ and $ s''$ in counterclockwise order around $ T$, and let  $\widetilde s$, $\widetilde s'$ and $\widetilde s'' \in \partial_\infty \widetilde S$ be the corresponding points of the circle at infinity. The  flag triple $\bigl( \mathcal F_\rho(\widetilde s), \mathcal F_\rho(\widetilde s') , \mathcal F_\rho(\widetilde s'')\bigr)$, associated to  $\widetilde s$, $\widetilde s'$ and $\widetilde s'' \in \partial_\infty \widetilde S$ by the flag curve $\F_\rho \colon \partial_\infty \widetilde S \to \Flag$,  is positive by Theorem~\ref{thm:FlagCurvePositive}. We can therefore consider the  logarithms
$$
\tau_{abc}^\rho (s) = \log T_{abc}
\bigl( \mathcal F_\rho(\widetilde s), \mathcal F_\rho(\widetilde s') , \mathcal F_\rho(\widetilde s'')\bigr)
$$
of its triple ratios, defined for  every $a$, $b$, $c\geq 1$ with $a+b+c=n$. By $\rho$--equivariance of the flag curve $\mathcal F_\rho$, these triple ratio   logarithms depend only on the triangle $T$ and on the spike $s$ of $T$, and not on  the choice of the lift $\widetilde T$.

Lemma~\ref{lem:SymmetriesTripleRatios} indicates how the invariant $\tau_{abc}^\rho (s)  \in \R$ changes if we choose a different vertex of the triangle $T$. 

\begin{lem}
\label{lem:RotateTriangleInvariants}
If $s$, $s'$ and $s''$ are the three spikes of the component $T$ of $S-\lambda$, indexed counterclockwise around $T$, then
\pushQED{\qed}
\begin{equation*}
\tau_{abc}^\rho (s)  = \tau_{bca}^\rho (s')  =\tau_{cab}^\rho (s'') .  \qedhere
\end{equation*}
\end{lem}

By invariance of triple ratios under the action of $\PGL$ on $\Flag$, it is immediate that the triangle invariants $\tau_{abc}^\rho (s) $ depend only on the character $\rho \in \Hit(S)$, and not  on the homomorphism $\rho \colon \pi_1(S) \to \PSL$ representing it. 

Because of Lemma~\ref{lem:RotateTriangleInvariants}, we can think of the invariant $\tau_{abc}^\rho(s)$ as mainly associated to the triangle component $T$ of $S-\lambda$ that has the slit $s$ as a vertex, since choosing a different vertex of $T$ only affects the order in which the indices $a$, $b$, $c$ are considered. For this reason, we will refer to the $\tau_{abc}^\rho(s)$ as the \emph{triangle invariants} of the Hitchin character $\rho \in \Hit(S)$.

\begin{rem}
The companion article \cite{BonDre} use a clockwise labeling convention for the vertices of a triangle. As a consequence, the triangle invariants of \cite{BonDre} are the opposite of those introduced here. 
\end{rem}

\section{Tangent cycles for a geodesic lamination}
\label{bigsect:TangentCycles}

The second type of invariants associated to a Hitchin character $\rho \in \Hit(S)$  are more closely tied to the geodesic lamination $\lambda$, and have a  homological flavor. This section is devoted to the definitions and basic properties of the corresponding objects.

\subsection{Tangent cycles}
\label{subsect:TangentCycles}

Let $\widehat \lambda$ be the orientation cover of the geodesic lamination $\lambda$, consisting of all pairs $(x,o)$ where $x\in \lambda$ and where $o$ is an orientation of the leaves of $\lambda$ near $x$. The map $(x,o) \mapsto x$ defines a 2--fold covering map $\widehat\lambda \to \lambda$. 

Intuitively, a  tangent cycle for $\widehat \lambda$ is a certain local multiplicity for the leaves of  $\widehat \lambda$, and defines a $1$--dimensional de Rham current supported in $\widehat\lambda$ as in \cite{RueSul}. This notion was called ``transverse cocycle'' in \cite{Bon97a} and in subsequent papers, with the discrepancy between cycles and cocycles explained by Poincar\'e duality. The change in terminology is motivated by the relative tangent cycles that will be introduced in \S \ref{subsect:RelativeTangentCycles}. 

Let $U$ be a neighborhood of the geodesic lamination $\lambda$ in $S$. If $U$ is small enough that it avoids at least one point of each component of $S-\lambda$, the cover $\widehat \lambda \to \lambda$ extends to a 2--fold cover $\widehat U \to U$  (not necessarily unique, according to the topology of $U$) for some surface $\widehat U$. 

A \emph{tangent cycle} $\alpha$ for the geodesic lamination $\widehat \lambda$ is the assignment of a number $\alpha(k) \in \R$ to each arc $k \subset \widehat U$ transverse to $\widehat \lambda$ such that:
\begin{enumerate}
\item $\alpha$ is \emph{finitely additive}, in the sense that $\alpha(k) = \alpha(k_1) + \alpha(k_2)$ whenever the arc $k$ is split into two transverse arcs $k_1$ and $k_2$;
\item $\alpha$ is \emph{invariant under homotopy respecting} $\widehat \lambda$, in the sense that $\alpha(k)= \alpha(k')$ whenever the transverse arcs $k$ and $k'$ are homotopic by a homotopy that keeps each point of $k\cap \widehat\lambda$ in the same leaf of $\widehat\lambda$. 
\end{enumerate} 

It easily follows from the above two conditions that $\alpha(k)=0$ for every arc $k$ disjoint from $\widehat\lambda$. As a consequence, the notion of tangent cycle is independent of the choice of the neighborhood $U$. 

A well-known example of tangent cycle are \emph{transverse measures} for $\widehat\lambda$. These can be defined as tangent cycles $\mu \in \CC(\widehat\lambda; \R)$ such that $\mu(k) \geq 0$ for every transverse arc $k$. Indeed, this positivity property enhances the finite additivity condition (1) to countable additivity. 

\subsection{Train track neighborhoods}
\label{subsect:TrainTracks}
To determine the space of tangent cycles for the geodesic lamination $\lambda$, we will use a very specific type of neighborhood $U$ for~$\lambda$.

A (trivalent)  \emph{train track neighborhood} for the geodesic lamination $\lambda$ is a closed neighborhood $U$ of $\lambda$ which can be decomposed as a union of finitely many rectangles $R_i$ such that
\begin{enumerate}
\item the boundary of each rectangle $R_i \cong [0,1] \times [0,1]$ is divided into a \emph{horizontal boundary} $\delh R_i = [0,1] \times \{0,1\} $ and a \emph{vertical boundary} $\delv R_i =  \{0,1\} \times [0,1]$;

\item  each component of the intersection $R_i \cap R_j$ of two distinct rectangles $R_i$ and $R_j$ is, either a component of $\delv R_i$ contained in $\delv R_j$ and containing one of the  endpoints of $\delv R_j$, or a component of $\delv R_j$ contained in $\delv R_i$ and containing one of the  endpoints of $\delv R_i$;

\item each of the four endpoints of $\delv R_i$ is contained in some rectangle $R_j$ different from $R_i$;

\item the leaves of $\lambda$ are transverse to the arcs $\{x \} \times [0,1]$ in each rectangle $R_i \cong [0,1] \times [0,1]$;

\item a fifth condition indicated below is satisfied.

\end{enumerate}

By construction, the boundary $\partial U$ of the train track neighborhood $U$ naturally splits into two pieces. The \emph{horizontal boundary} $\delh U$ is the union of the horizontal boundaries $\delh R_i$ of all rectangles $R_i$. The \emph{vertical boundary} consists of those points of  $\partial U$ that are contained in the vertical boundary $\delv R_i$ of some rectangle $R_i$. 

We can now state the missing condition. 

\begin{enumerate}
\setcounter{enumi}{4}
\item no component of $S-U$ is a disk with 0, 1 or 2 components of the vertical boundary $\delv U$ in its closure. 
\end{enumerate}

In particular, the arcs $\{x \} \times [0,1]$ of each rectangle $R_i \cong [0,1] \times [0,1]$ provide a foliation of  $U$, whose leaves are called the \emph{ties} of the train track neighborhood. A tie is \emph{generic} if it meets the boundary of $U$ only at its endpoints. Otherwise, it is \emph{singular}. 

The origin of the train track terminology should become apparent when $U$ is chosen so that its ties are relatively short. See Figure~\ref{fig:TrainTrack}. In particular, a singular tie is also often called a \emph{switch}, and the rectangles $R_i$ are the \emph{edges} of $U$. 

The definitions are such that a singular tie $t$ is adjacent to three edges $R_i$, $R_j$, $R_k$, in such a way that $t$ is equal to a component of the vertical boundary $\delv R_i$, and is also the union of a component of $\delv R_j$, of a component of $\delv R_k$ and of a component of $\delv U$. The rectangles $R_i$, $R_j$, $R_k$ are not necessarily distinct. 

\begin{figure}[htbp]

\SetLabels
( .72*.58) $\lambda$ \\
( .33* .55) $ R_i$ \\
( .53* .7) $ R_j$ \\
(.53 * .17) $R_k $ \\
\endSetLabels
\centerline{\AffixLabels{ \includegraphics{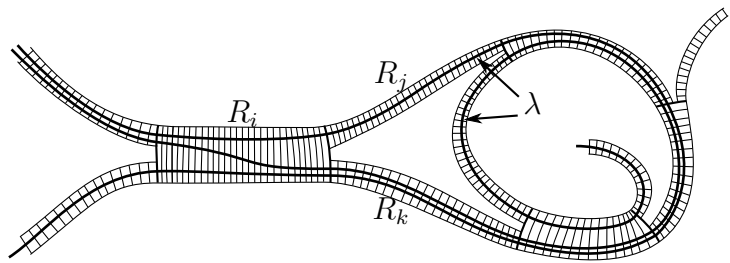} }}
\caption{A train track neighborhood}
\label{fig:TrainTrack}
\end{figure}

Every geodesic lamination admits a train track neighborhood. 

When the geodesic lamination $\lambda$ is maximal, there is a crucial property of its train track neighborhoods $U$ that we will use on a regular basis. Recall that the complement of $\lambda$ then consists of infinite triangles. The following property is easily proved by extending  the foliation of $U$ by its ties to a foliation of $S$ with saddle-type singularities, and by using an index computation on each component of the complement $S-\lambda$. 

\begin{prop}
\label{prop:TrainTrackMaxGeodLam}
Let $U$ be a train track neighborhood of the maximal geodesic lamination $\lambda$. Then, every component
$T$ of the complement $S-\lambda$ contains exactly one component $H=T-U$ of $S-U$;   this component $H$ is a hexagon, namely a disk whose boundary is the union of  $3$ components of the horizontal boundary $\delh U$ and $3 $ components of the vertical boundary $\delv U$. In addition, the foliation of $T\cap U$ by the ties of $U$ is as indicated in Figure~{\upshape\ref{fig:TrainTrackMaxGeodLam}}. \qed
\end{prop}

Incidentally, another index argument applied to the whole surface $S$ shows that the complement $S-U$ consists of $4(g-1)$ hexagons. In particular, this proves that the complement $S-\lambda$ consists of $4(g-1)$ triangles. 
\begin{figure}[htbp]

\SetLabels
(.5 *.65 ) $\lambda $ \\
(.62 * .48) $S-U $ \\
(.8 *.7 ) $U $ \\
(.91 *.33 ) $\delv U $ \\
( .89*.5) $ \delh U$ \\
\endSetLabels
\centerline{\AffixLabels{\includegraphics{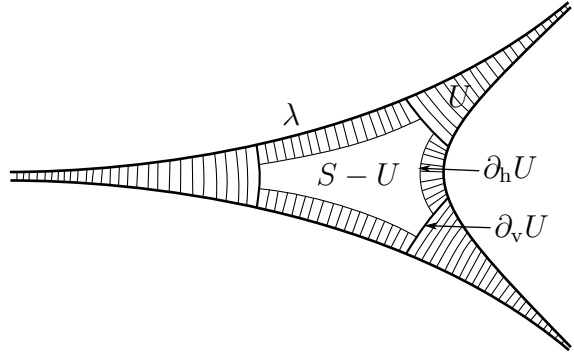}}}
\caption{Train track neighborhoods and maximal geodesic laminations}
\label{fig:TrainTrackMaxGeodLam}
\end{figure}

\subsection{Homological interpretation of tangent cycles}
\label{subsect:SpaceTangentCycles}

Train track neighborhoods provide a convenient tool to perform computations in the vector space  $\mathcal C(\widehat \lambda;\R)$ consisting of all  tangent cycles for the orientation cover $\widehat\lambda$ of $\lambda$. 

Let $U$ be a train track neighborhood of the maximal geodesic lamination $\lambda$. Using Proposition~\ref{prop:TrainTrackMaxGeodLam}, the orientation cover map $\widehat\lambda \to \lambda$ has a unique extension to a cover $\widehat U \to U$. Note that $\widehat\lambda$ is a geodesic lamination in the surface $\widehat U$, and that $\widehat U$ is a train track neighborhood of $\widehat\lambda$. Also, each component of $\widehat U - \widehat\lambda$ is an annulus bounded on one side by a chain of 6 leaves of $\widehat\lambda$, and on the other side by a dodecagon made up of 6 components of the horizontal boundary $\delh \widehat U$ and 6 components of the vertical boundary $\delv\widehat U$. 

The leaves of the orientation cover $\widehat \lambda$ are canonically oriented (use the orientation $o$ near the point $(x,o) \in \widehat\lambda$). This enables us to orient the ties of $\widehat U$ from left to right  with respect to this canonical orientation of $\widehat\lambda$. Indeed,  Proposition~\ref{prop:TrainTrackMaxGeodLam} guarantees that, for every tie $k$ of $\widehat U$,  the left-to-right orientation at the endpoints of a component  $d$ of $k-\widehat \lambda$ extends to an orientation of $d$.  

\begin{prop}
\label{prop:HomologyTangentCycles}

A tangent cycle $\alpha\in  \CC (\widehat \lambda ; \R) $ uniquely determines a homology class $[\alpha] \in  H_1(\widehat U; \R)$ by the property that 
$$
\alpha(k) = [k] \cdot [\alpha]
$$
for every generic tie $k$ of the train track neighborhood $\widehat U$, where $ [k] \cdot [\alpha]$ is the algebraic intersection number of $[\alpha] \in  H_1(\widehat U; \R)$  with the relative homology class
 $[k] \in H_1(\widehat U, \partial \widehat U; \R)$ defined by the tie $k$, endowed with the above left-to-right orientation.
 
 In addition, the rule $\alpha \mapsto [\alpha]$ defines a linear isomorphism $\CC (\widehat \lambda ; \R)  \to  H_1(\widehat U; \R)$. 
\end{prop}
\begin{proof} Because the geodesic lamination $\lambda$ is maximal, Proposition~\ref{prop:TrainTrackMaxGeodLam} shows that it is tightly carried by the train track $U$, in the sense that each component of $U-\lambda$ is an annulus. It follows that $\widehat \lambda$ is tightly carried by $\widehat U$. The result is then a consequence of \cite[Theorem~11]{Bon97a}. 
\end{proof}

\begin{lem}
\label{lem:DimensionTangentCycles}
If the surface $S$ has genus $g$, 
$$
\mathcal C(\widehat \lambda ; \R) \cong H_1(\widehat U; \R) \cong \R^{12g-11}.
$$
\end{lem}

\begin{proof}
Since the complement $S-\lambda$ consists of infinite triangles, the geodesic lamination $\lambda$ is non-orientable. This implies that $\widehat\lambda$ is connected, and therefore so is $\widehat U$. By definition of the Euler characteristic $\chi(\ )$, 
$$
\dim H_1(\widehat U; \R) = -\chi(\widehat U) + \dim H_0(\widehat U; \R) = -\chi(\widehat U)  + 1 = -2 \chi(U) +1 .
$$
We observed that the complement of $U$ in $S$ consists of $4(g-1)$ hexagons. Therefore,  $\chi(U) = \chi(S)-4(g-1)=-6(g-1)$. The result follows. 
\end{proof}

\subsection{Tangent cycles relative to the slits}
\label{subsect:RelativeTangentCycles}

We now relax the additivity condition for a tangent cycle.

Let $U$ be a neighborhood of $\lambda$ in $S$ that avoids at least one point of each component of $S-\lambda$. For instance, $U$ can be a train track neighborhood of $\lambda$. Extend the orientation cover $\widehat\lambda \to \lambda$ to a 2--fold cover $\widehat U \to U$. The complement $\widehat U - \widehat\lambda$ has a certain number of infinite spikes, in fact $24(g-1)$ spikes because the complement $S-\lambda$ consists of $4(g-1)$ infinite triangles and because each spike of $S-\lambda$ lifts to two spikes of $\widehat U - \widehat \lambda$. In particular, the spikes of $\widehat U - \widehat\lambda$  are really independent of the choice of the neighborhood $U$. For this reason, we will also refer to the spikes of  $\widehat U - \widehat\lambda$  as the \emph{slits of $\widehat\lambda$}.

We need  to restrict attention to a special class of transverse arcs for $\widehat\lambda$.  An arc $k \subset \widehat U$ is \emph{tightly transverse} to the geodesic lamination $\widehat\lambda$ if it is transverse to $\widehat\lambda$, if it has nonempty intersection with $\widehat\lambda$ and if, for every component $d$ of $k-\widehat\lambda$, one of the following holds:
\begin{itemize}
\item
 $d$ contains one of the endpoints of $k$;
 \item $d$  separates one of the spikes of $\widehat U - \widehat\lambda$ from the rest of $\widehat U - \widehat\lambda$. 
 \end{itemize}

A fundamental example arises when the geodesic lamination $\lambda$ is maximal and when $ U$ is a train track neighborhood of $\lambda$, so that its  lift $\widehat U$ is a train track neighborhood of $\widehat \lambda$. It then follows from Proposition~\ref{prop:TrainTrackMaxGeodLam} that every tie of $\widehat U$ is tightly transverse to $\widehat \lambda$. 

The slits of $\widehat\lambda$, namely the spikes of $\widehat U - \widehat\lambda$, come in two types because of the canonical orientation of the leaves of the orientation cover $\widehat\lambda$: the \emph{positive slits} $s$ where the two leaves of $\widehat\lambda$ that are adjacent to $s$ are oriented towards $s$ for the canonical orientation of $\widehat\lambda$, and the \emph{negative slits} where these two leaves are oriented away from $s$. Define the \emph{sign} of the slit $s$ of $\widehat U - \widehat\lambda$ as $\epsilon(s) =+1$ when $s$ is positive, and $\epsilon(s) =-1$ for a negative slit.

An $\R$--valued \emph{tangent cycle relative to the slits} for $\widehat \lambda$ assigns a number $\alpha(k)\in \R$ to each arc $k \subset \widehat U$ tightly transverse to $\widehat\lambda$ in such a way that:

\begin{enumerate}

\item $\alpha$ is, as before, \emph{invariant under homotopy respecting} $\widehat \lambda$ in the sense that $\alpha(k)= \alpha(k')$ whenever the transverse arcs $k$ and $k'$ are homotopic by a homotopy that keeps each point of $k\cap \widehat\lambda$ in the same leaf of $\widehat\lambda$;

\item $\alpha$ is  \emph{quasi-additive} in the following sense. There is a number $\partial\alpha(s)\in \R$ associated to each slit $s$ of $\widehat \lambda$ such that 
$$
\alpha(k) = \alpha(k_1) + \alpha(k_2) - \epsilon(s) \partial \alpha(s)
$$
whenever the arc $k \subset \widehat U$ is tightly transverse to $\widehat \lambda$, the arcs $k_1$ and $k_2$ are obtained by splitting $k$ at a point $x\in k-\widehat\lambda$ contained in a component $d$ of $k-\widehat\lambda$ that is disjoint from $\partial k$, and  $s$ is the spike separated from the rest of $\widehat U - \widehat \lambda$ by the component $d$. 


\end{enumerate} 

By definition, the function $\partial\alpha \colon \{ \text{slits of } \widehat\lambda\} \to \R$ is the \emph{boundary} of the relative cycle $\alpha$. We could have combined $\partial \alpha$ with the sign $\epsilon$ to create a single function $ \{ \text{slits of } \widehat\lambda\} \to \R$, but the current convention simplifies the homological interpretation of relative tangent cycles that is given below, in Proposition~\ref{prop:RelativeHomologyTangentCycles}. This homological interpretation also explains the boundary terminology.

We let $\CC(\widehat\lambda, \slits; \R)$ denote the space of tangent cycles relative to the slits for $\widehat\lambda$. 

Using the quasi-additivity property, one easily shows that the notion of tangent cycle relative to the slits is independent of the choice of the neighborhood $U$ of $\lambda$. 

These relative tangent cycles generalize the tangent cycles of \S \ref{subsect:TangentCycles}. 

\begin{lem}
\label{lem:RelativeTgtCyclesWithBdry0}
There is a natural correspondence between   the set $\CC(\widehat\lambda; \R)$ of tangent cycles for $\widehat \lambda$ and the set $\{ \alpha \in \CC(\widehat\lambda, \slits; \R); \partial \alpha=0\}$ of tangent cycles relative to the slits with boundary $0$. 
\end{lem}
\begin{proof}
A relative tangent cycle with boundary equal to 0 is additive. So the only point that requires some discussion is the fact that relative tangent cycles are restricted to arcs tightly transverse to $\widehat\lambda$, whereas the definition of tangent cycles involves all tangent arcs transverse to $\widehat\alpha$. 

However, every arc $k$ transverse to $\widehat\lambda$ can be split into the union of finitely many arcs $k_1$, $k_2$, \dots, $k_l$ that are tightly transverse to $\widehat\lambda$. It easily follows that every relative tangent cycle $\alpha \in \CC(\widehat\lambda, \slits; \R)$ with $\partial \alpha=0$ uniquely extends to a tangent cycle, by the property that $\alpha(k) = \sum_{i=1}^l \alpha(k_i)$ for every transverse arc $k$ split as above into finitely many tightly transverse arcs $k_i$. Indeed, the additivity property guarantees that this $\alpha(k)$ does not depend on the decomposition of $k$ into tightly transverse arcs. 
\end{proof}

\subsection{Homological interpretation of tangent cycles relative to the slits}
\label{subsect:RelativeHomologyTangentCycles}

We now focus on a train track neighborhood $U$ of the maximal geodesic laminations $\lambda$. As before, let $\widehat \lambda$ be the orientation cover of $\lambda$, and extend the covering map $\widehat\lambda \to \lambda$ to a cover $\widehat U \to U$. The  canonical orientation of the leaves of $\widehat\lambda$ provides a left-to-right orientation for the ties of $\widehat U$. 

By Proposition~\ref{prop:TrainTrackMaxGeodLam}, there is a one-to-one correspondence between the slits of $\widehat\lambda$ and the components of the vertical boundary $\delv \widehat U$. Indeed, each component $c$ of $\delv \widehat U$ faces a unique spike $s$ of $U- \widehat \lambda$ (= slit of $\widehat\lambda$) in the sense that, if $k$ is the singular tie of $\widehat U$ that contains $ c$ and if $d$ is the component of $k-\widehat\lambda$ that contains $c$, then $d$ separates $s$ from the rest of $\widehat U - \widehat\lambda$; see Figure~\ref{fig:TrainTrackMaxGeodLam}.

For a relative tangent cycle  $\alpha\in  \CC (\widehat \lambda, \slits ; \R) $, the boundary $\partial \alpha \colon \{\text{slits of } \widehat\lambda\} \to \R$ therefore assigns a multiplicity to each component of $\delv \widehat U$, and therefore can be interpreted as an element of $H_0(\delv\widehat U; \R) $. 

\begin{prop}
\label{prop:RelativeHomologyTangentCycles}
Let $U$ be a train track neighborhood of the maximal geodesic lamination $\lambda$, and let $\widehat U$ be its lift to a train track neighborhood of the orientation cover $\widehat\lambda$. 
A tangent cycle $\alpha\in  \CC (\widehat \lambda, \slits ; \R) $ relative to the slits of $\widehat\lambda$ uniquely determines a relative homology class $[\alpha] \in  H_1(\widehat U, \delv \widehat U; \R)$ by the property that 
$$
\alpha(k) = [k] \cdot [\alpha]
$$
for every generic tie $k$ of the train track neighborhood $\widehat U$, where $ [k] \cdot [\alpha]$ is the algebraic intersection number of $[\alpha] \in  H_1(\widehat U, \delv \widehat U; \R)$  with the relative homology class
 $[k] \in H_1(\widehat U, \delh \widehat U; \R)$ defined by the tie $k$, endowed with the above left-to-right orientation.
 
 In addition, the rule $\alpha \mapsto [\alpha]$ defines a linear isomorphism $\CC (\widehat \lambda, \slits ; \R)  \cong H_1(\widehat U, \delv \widehat U; \R) $, for which the boundary  $\partial \alpha \colon \{\text{slits of } \widehat\lambda\} \to \R$ of the relative tangent cycle $\alpha$ corresponds to the image of $[\alpha] \in H_1(\widehat U, \delv \widehat U; \R)$ under the boundary homomorphism $\partial \colon  H_1(\widehat U, \delv \widehat U; \R)  \to H_0(\delv\widehat U; \R) $. 
\end{prop}

\begin{proof} We  split the proof into a few steps to improve readability.

\smallskip
\noindent\textsc{Step 1.} Construct a linear map  $\phi \colon \CC(\widehat\lambda, \slits; \R) \to H_1(\widehat U, \delv \widehat U; \R)$.

Pick a generic tie $k_e$ in each edge $e$ of $\widehat U$. An easy homological computation shows that, as $e$ ranges over all edges of $\widehat U$, the relative homology classes $[k_e]$ form a basis for $H_1(\widehat U, \delh \widehat U; \R)$. The map $[k_e] \mapsto \alpha(k_e)$ therefore extends to a linear map $H_1(\widehat U, \delh \widehat U; \R) \to \R$. By Poincar\'e duality and since the boundary $\partial \widehat U$ is the union of $\delh \widehat U$ and $\delv \widehat U$, there consequently exists a unique class $[\alpha] \in H_1(\widehat U, \delv \widehat U; \R)$ such that $\alpha(k_e) = [k_e] \cdot [\alpha]$ for every edge $e$. 

An arbitrary generic tie $k$ of $\widehat U$ is contained in an edge $e$. Then, $[k] = [k_e]$ in $H_1(\widehat U, \delh \widehat U; \R)$, and $\alpha(k) = \alpha(k_e)$ by invariance of $\alpha$ under homotopy respecting $\widehat\lambda$. This proves that $
\alpha(k) = [k] \cdot [\alpha]$ for every generic tie $k$ of  $\widehat U$. As a consequence, $[\alpha]$ satisfies the properties indicated in the statement of Proposition~\ref{prop:RelativeHomologyTangentCycles}. 

This provides  a map $\phi \colon \CC(\widehat\lambda, \slits; \R) \to H_1(\widehat U, \delv \widehat U; \R)$, associating the above class $[\alpha]\in H_1(\widehat U, \delv \widehat U; \R)$ to $\alpha \in  \CC(\widehat\lambda, \slits; \R)$.

\smallskip
\noindent\textsc{Step 2.} Construct a linear map  $\psi \colon H_1(\widehat U, \delv \widehat U; \R) \to \CC(\widehat\lambda, \slits; \R)$. 

We first associate a homology class  $[k]\in H_1(\widehat U, \delh \widehat U; \R)$ to  each arc $k$ that is tightly transverse to  $\widehat\lambda$. 

A key observation is that the canonical orientation of the orientation cover $\widehat\lambda$ specifies  a natural orientation for $k$. Indeed the definition of tight transversality implies that, if the arc $k$ is tightly transverse to $\widehat \lambda$, the leaves of $\widehat\lambda$ passing through the endpoints of a component $d$ of $k-\widehat\lambda$ induce the same transverse orientation (namely an orientation of the normal bundle) for $k$. As a consequence, all leaves of $\widehat\lambda$ define the same transverse orientation for $k$. We can therefore orient every tightly transverse arc $k$ from left to right with respect to  the canonical orientation of the leaves of $\widehat \lambda$.

We now extend the tightly transverse arc $k$ to an arc $k'\subset \widehat U$ with $\partial k' \subset \delh \widehat U$. There is  a natural one-to-one correspondence between the components of the horizontal boundary $\delh \widehat U$ and the boundary leaves of $\widehat \lambda$ (namely those which are in the boundary of $\widehat U - \widehat\lambda$); indeed, Proposition~\ref{prop:TrainTrackMaxGeodLam} shows that all ties originating from a component of $\delh\widehat U$ leave $\widehat U - \widehat\lambda$ on the same boundary leaf of $\widehat\lambda$. For each component $d$ of $k-\widehat\lambda$ containing an endpoint of $k$, we can extend $d$ to an arc $d'\subset \widehat U - \widehat\lambda$ going from a boundary leaf to $\widehat\lambda$ to the corresponding component of $\delh \widehat U$, in the homotopy class specified by the arcs in ties of $\widehat U$ that connect this boundary leaf to $\delh \widehat U$. Performing this operation for each of the two components $d$ of $k-\widehat\lambda$ that contain an endpoint of $k$, we have extended $k$ to an oriented arc $k' \supset k$ whose boundary is contained in $\delh \widehat U$. There are many possible choices for $k'$ but all give the same relative homology class in $ H_1(\widehat U, \delh \widehat U; \R)$, which we denote by $[k]$. 


Given a relative homology class $c \in H_1(\widehat U, \delv \widehat U; \R) $ we can consider, for every arc $k$ tightly transverse to $\widehat\lambda$, the algebraic intersection number
$$
\alpha_c (k) =  [k] \cdot c \in \R
$$
of $c \in H_1(\widehat U, \delv \widehat U; \R) $ with the homology class $[k] \in H_1(\widehat U, \delh \widehat U; \R)$ associated to $k$ as above. We want to show that this defines a relative tangent cycle $\CC(\widehat\lambda, \slits; \R)$. 

The invariance of $\alpha_c(k) $ under homotopy of $k$ respecting $\widehat\lambda$ is immediate. 

We need to check the quasi-additivity property. Let the
arc $k \subset \widehat U$ be tightly transverse to $\widehat \lambda$,  let  $k_1$ and $k_2$ be obtained by splitting $k$ at a point $x\in k-\widehat\lambda$ contained in a component $d$ of $k-\widehat\lambda$ that is disjoint from $\partial k$, and let $s$ be the spike separated from the rest of $\widehat U - \widehat \lambda$ by the component $d$. Let $k_s$ be the component of $\delv \widehat U$ that faces the slit $s$. Orient $k_s$ by the boundary orientation of $\partial\widehat U$. 

Then, from the definition of the relative homology classes $[k]$, $[k_1]$, and $[k_2]\in H_1(\widehat U, \delh \widehat U; \R) $,
$$
[k] = [k_1] + [k_2] + \epsilon(s) [k_s] \in H_1(\widehat U, \delh \widehat U; \R)
$$
where $\epsilon(s)=\pm1$ is the sign of the slit $s$. Taking intersection numbers with $c \in H_1(\widehat U, \delv \widehat U; \R) $, it follows that 
$$
\alpha_c(k) = \alpha_c(k_1) + \alpha_c(k_2) + \epsilon(s) [k_s]\cdot c.
$$

This proves that $\alpha_c$ is a tangent cycle for $\widehat\lambda$ relative to its slits, with boundary $\partial \alpha_c $ defined by the property that $\partial\alpha_c(s) = - [k_s]\cdot c$ for every slit $s$. 

We define $\psi \colon H_1(\widehat U, \delv \widehat U; \R) \to \CC(\widehat\lambda, \slits; \R)$ by the property that $\psi(c) = \alpha_c$ for every $c \in H_1(\widehat U, \delv \widehat U; \R) $.

\smallskip
\noindent\textsc{Step 3.} For every $c \in H_1(\widehat U, \delv \widehat U; \R) $ and every slit $s$ of $\widehat\lambda$, $\partial \psi(c) (s)\in \R$ is the multiplicity associated to the component $k_s$  of $\delv \widehat U$  facing  $s$ by $\partial c \in H_0(\delv \widehat U; \R)$. 

This is just a rephrasing of the property that $\partial\alpha_c(s) = - [k_s]\cdot c$. 

\smallskip
\noindent\textsc{Step 4.} The maps $\phi \colon \CC(\widehat\lambda, \slits; \R) \to H_1(\widehat U, \delv \widehat U; \R)$ and $\psi \colon H_1(\widehat U, \delv \widehat U; \R) \to \CC(\widehat\lambda, \slits; \R)$  are inverse of each other. 

Pick a generic tie $k_e$ in each edge $e$ of $\widehat U$. Then, by construction, the image $c=\phi(\alpha)$ of $\alpha \in  \CC(\widehat\lambda, \slits; \R) $ is defined by the property that $\alpha(k_e) = [k_e]\cdot c$ for every edge $e$. Conversely, for every $c \in H_1(\widehat U, \delv \widehat U; \R) $, $\alpha = \psi(c)$ is characterized by the fact that $\alpha(k) = [k]\cdot c$ for every arc $k$ tightly transverse to $\widehat\lambda$.

In particular, $[k_e]\cdot \phi\bigl( \psi(c) \bigr)  = [k_e]\cdot c$ for every edge $e$, and it follows that $\phi\bigl( \psi(c) \bigr) = c$ by Poincar\'e duality since the $k_e$ generate $H_1(\widehat U, \delh \widehat U; \R) $. This proves that $\phi \circ \psi$ is equal to the identity.

Conversely, for a relative tangent cycle $\alpha \in  \CC(\widehat\lambda, \slits; \R)$, the same argument shows that $\psi \bigl( \phi(\alpha) \bigr) (k_e) = \alpha(k_e)$ for every edge $e$ of $\widehat U$. For a slit $s$, let $k_s$ be the component of $\delv \widehat U$ that faces $s$, let $e$ be the edge of $\widehat U$ that contains $k_s$, and let $e_1$ and $e_2$ be the other two edges that touch $k_s$. Then, by definition of the quasi-additivity, 
\begin{align*}
\epsilon(s) \partial \alpha(s) &= \alpha (k_{e_1}) + \alpha  (k_{e_2}) - \alpha (k_e)  \\
&= \psi \bigl( \phi(\alpha) \bigr) (k_{e_1}) + \psi \bigl( \phi(\alpha) \bigr)  (k_{e_2}) - \psi \bigl( \phi(\alpha) \bigr) (k_e) \\
&= \epsilon(s) \partial \psi \bigl( \phi(\alpha) \bigr) (s).
\end{align*}
This proves that $\psi \bigl( \phi(\alpha) \bigr) - \alpha $ has boundary 0, and is therefore a tangent cycle by Lemma~\ref{lem:RelativeTgtCyclesWithBdry0}. Since $\psi \bigl( \phi(\alpha) \bigr)(k_e) - \alpha(k_e) $ for every edge $e$ of $\widehat U$, it follows from Proposition~\ref{prop:HomologyTangentCycles} that $\psi \bigl( \phi(\alpha) \bigr) - \alpha =0$.

This proves that $\psi\circ \phi$ is the identity, and completes the proof of Proposition~\ref{prop:RelativeHomologyTangentCycles}. 
\end{proof}




\subsection{Twisted relative tangent cycles}
\label{subsect:TwistedRelTangentCycles}

So far, we have considered relative tangent cycles valued in $\R$. 
In  our analysis of Hitchin characters, we will encounter relative tangent cycles that are valued in $\R^{n-1}$
and behave in a very specific manner with respect to the involution $\tau \colon \widehat U\to \widehat U$ that exchanges the two sheets of the cover $\widehat U \to U$. 

More precisely, an \emph{$\R^{n-1}$--valued  tangent cycle for $\widehat \lambda$ relative to its slits} associates a vector $\alpha(k)\in \R^{n-1}$ to each arc $k$ tightly transverse to $\widehat\lambda$, in such a way that $\alpha$ is invariant under homotopy respecting $\widehat\lambda$ and is quasi-additive with respect to a boundary function $\partial \alpha \colon \{ \text{slits of } \widehat\lambda \} \to \R^{n-1}$. 

A \emph{twisted tangent cycle for $\lambda$ relative to its slits and valued in $\widehat \R^{n-1}$} is an $\R^{n-1}$--valued  relative tangent cycle $\alpha$ for $\widehat\lambda$ such that, for every tightly transverse arc $k$,
$$
\alpha \bigl( \tau(k) \bigr) = \overline {\alpha(k)}
$$
where $x \mapsto \bar x$ is the involution of $\R^{n-1}$ that reverses the order of the coordinates, namely that associates $\bar  x = (x_{n-1}, x_{n-2}, \dots, x_1)$ to $x=(x_1, x_2, \dots, x_{n-1})\in \R^{n-1}$. Let 
$$
\CC(\lambda, \slits; \widehat\R^{n-1}) = \bigl\{ \alpha \in \CC(\widehat\lambda, \slits;\R^{n-1}); \alpha\bigl( \tau(k) \bigr) = \overline{\alpha(k)} \bigr\}
$$ 
denote the space of these twisted relative tangent cycles. 

The terminology and notation is justified by the fact that these twisted relative tangent cycles can be interpreted as tangent cycles for the geodesic lamination $\lambda$, relative to the slits of $\lambda$, and valued in the twisted coefficient bundle $\widehat \R^{n-1} = (\widehat U \times \R^{n-1})/\Z_2$ where $\Z_2$ acts by $\tau$ on $\widehat U$ and by $x\mapsto \bar x$ on $\R^{n-1}$. 

We can similarly define the space of twisted tangent cyles
\begin{align*}
\CC(\lambda; \widehat\R^{n-1}) &=\bigl \{ \alpha \in \CC(\widehat\lambda;\R^{n-1}); \alpha\bigl( \tau(k) \bigr) = \overline{\alpha(k)}\bigr \}\\
&= \bigl \{ \alpha \in \CC(\lambda, \slits ;\widehat \R^{n-1}); \partial \alpha=0 \bigr \}
\end{align*}
where the second equality comes from Lemma~\ref{lem:RelativeTgtCyclesWithBdry0}.

\begin{prop}
\label{prop:ComputeTwistedTgentCycles}
The  vector spaces $\CC(\lambda; \widehat\R^{n-1})$ and $\CC(\lambda, \slits; \widehat\R^{n-1})$ have  dimensions
\begin{align*}
\dim \CC(\lambda; \widehat\R^{n-1}) &= 6(g-1)(n-1) + {\textstyle\lfloor \frac{n-1}2 \rfloor}\\
\dim \CC(\lambda, \slits; \widehat\R^{n-1}) &= 18(g-1)(n-1) 
\end{align*}
where $\lfloor x \rfloor$ denotes the largest integer that is less than or equal to $x$. 
\end{prop}

\begin{proof} We use  a version of Propositions~\ref{prop:HomologyTangentCycles} and \ref{prop:RelativeHomologyTangentCycles} that gives a homological interpretation of twisted tangent cycles.  It uses a different coefficient bundle  $\widetilde \R^{n-1} = (\widehat U \times \R^{n-1})/\Z_2$, where $\Z_2$ still acts by the covering involution $\tau$ on $\widehat U$ but now acts  on $\R^{n-1}$ by $x\mapsto -\bar x$.

Indeed, because $\tau$ reverses the orientation of $\widehat\lambda$, the map $\phi \colon \CC(\widehat \lambda, \slits; \R) \to H_1(\widehat U, \delv \widehat U; \R)$ of the proof of Proposition~\ref{prop:RelativeHomologyTangentCycles} conjugates the action of $\tau$ on $\CC(\widehat \lambda, \slits; \R)$ to $-\tau_*$, where $\tau_* \colon H_1(\widehat U, \delv \widehat U; \R) \to H_1(\widehat U, \delv \widehat U; \R)$ is the homomorphism induced by $\tau$. Therefore, the tensor product $\phi\otimes\Id_{\R^{n-1}}$  sends $\CC(\lambda, \slits; \widehat \R^{n-1})$ to $\{ c\in H_1(\widehat U, \delv \widehat U; \R^{n-1}); \tau_*(c) = -\overline c\}$, which is naturally identified to $H_1(U,\delv U; \widetilde \R^{n-1})$. This provides a natural isomorphism $\CC( \lambda, \slits; \widehat \R^{n-1}) \cong H_1(U,\delv U; \widetilde \R^{n-1})$, which also induces an  isomorphism $\CC( \lambda; \widehat \R^{n-1}) \cong H_1(U; \widetilde \R^{n-1})$.

Considering Euler characteristics, 
$$
\chi( U) (n-1) = \dim H_0( U; \widetilde \R^{n-1}) - \dim H_1( U; \widetilde \R^{n-1}) .
$$
Since $\widehat U$ is connected, 
$$
H_0( U; \widetilde \R^{n-1}) = \{ c\in H_0(\widehat U;  \R^{n-1}); \tau_*(c) =-\overline c\} \cong  \{x\in \R^{n-1}; x=-\overline x\}
$$
has dimension $\lfloor \frac{n-1}2 \rfloor$. Also, because the complement $S-U$ consists of $4(g-1)$ hexagons, $\chi(U) = \chi(S)-4(g-1) = -6(g-1)$. It follows that 
\begin{align*}
\dim \CC( \lambda; \widehat \R^{n-1}) &=\dim H_1( U; \widetilde \R^{n-1})
 =  - \chi( U) (n-1) + \dim H_0( U; \widetilde \R^{n-1}) \\
&=  6(g-1)(n-1) + {\textstyle\lfloor \frac{n-1}2 \rfloor}. 
\end{align*}

For $\CC( \lambda, \slits; \widehat \R^{n-1}) \cong H_1(U,\delv U; \widetilde \R^{n-1})$, consider the exact sequence
$$
0\to H_1(U ; \widetilde \R^{n-1}) \to H_1(U,\delv U; \widetilde \R^{n-1}) \to H_0(\delv U; \widetilde \R^{n-1}) \to H_0( U; \widetilde \R^{n-1})\to 0. 
$$
We already observed that $\dim H_0( U; \widetilde \R^{n-1}) = \lfloor\frac{n-1}2\rfloor$. Since $\tau$ respects no component of $\delv \widehat U$, the twisted homology space $H_0(\delv U; \widetilde \R^{n-1}) $ is isomorphic to $ H_0(\delv U;  \R^{n-1}) $ and therefore has dimension $12(g-1)(n-1)$ as $\delv U$ has $12(g-1)$ components. 
It follows from the exact sequence above that
\begin{align*}
\dim \CC( \lambda, \slits; \widehat \R^{n-1})
&= \dim H_1(U, \delv U ; \widetilde \R^{n-1}) \\
&=  \dim H_1(U ; \widetilde \R^{n-1}) +\dim H_0(\delv U; \widetilde \R^{n-1}) - \dim H_0( U; \widetilde \R^{n-1})\\
&= 18(g-1)(n-1). \qedhere
\end{align*}
\end{proof}

\subsection{Relative tangent cycles from another viewpoint}
\label{subsect:RelTgtCyclesDifferentView}

We give a different description of relative tangent cycles. Compared to the original definition, this presentation does not lend itself as well to the homological interpretation and computations of the previous sections. However, it will be better adapted to the geometric constructions that form the core of this  article. It also bypasses the need to consider the orientation cover $\widehat\lambda$. 

In the universal cover $\widetilde S$ of $S$, let $\widetilde U$ be the preimage of a train track neighborhood $U$ of $\lambda$. 

A relative tangent cycle $\alpha \in \CC(\widehat \lambda, \slits; \R)$ associates a number $\alpha(T,T')\in \R$ to each ordered pair of  distinct components $T$ and $T'$  of $\widetilde S - \widetilde\lambda$ as follows. Choose an oriented arc $\widetilde k\subset \widetilde S$ that is tightly transverse to $\widetilde \lambda$ and joints $T$ to $T'$; in this preliminary stage, one can for instance  take for $\widetilde k$ any geodesic arc going from $T$ to $T'$, since every component of $\widetilde S - \widetilde \lambda$ is a triangle. Using Proposition~\ref{prop:TrainTrackMaxGeodLam}, one can modify  $\widetilde k$ by a homotopy respecting $\widetilde \lambda$  so that it is contained in $\widetilde U$, and is tightly transverse to $\widetilde\lambda$ in $\widetilde U$. Project $\widetilde k$  to an arc $k\subset U$, which is tightly transverse to $\lambda$.

 The tightly transverse arc $k$ admits two lifts to the 2--fold cover $\widehat U$ of $U$, each oriented so that the canonical orientation of the leaves of the orientation cover $\widehat\lambda$ points to the left of these arcs at each intersection point. Let $\widehat k \subset \widehat U$ be the lift whose orientation projects to the same orientation of $k$ as that of $\widetilde k$. By construction, $\widehat k$ is tightly transverse to $\widehat \lambda$ in $\widehat U$, and we can consider the number $\alpha(\widehat k) \in \R$ defined by $\alpha \in \CC(\widehat \lambda, \slits; \R)$. 
 
 In this construction, the arc $\widetilde k$ is uniquely determined by $T$ and $T'$ up to homotopy respecting $\widetilde\lambda$ in $\widetilde U$, which determines $\widehat k$ up to homotopy respecting $\widehat \lambda$ in $\widehat U$. It follows that $\alpha(\widehat k)$ depends only on $T$ and $T'$, and we can define $\alpha(T, T') =  \alpha(\widehat k) \in \R$. 
 
 The quasi-additivity property of $\alpha \in \CC(\widehat \lambda, \slits; \R)$ has a relatively simple translation in this context. Each slit $s$ of $\lambda$, namely each spike of $S- \lambda$, lifts to two slits of $\widehat\lambda$: a positive spike $s^+$ of $\widehat U - \widehat\lambda$ where the leaves of $\widehat \lambda$ adjacent to $s^+$ are oriented towards the end of this spike by the canonical orientation of $\widehat\lambda$; and a negative spike $s^-$ where the adjacent leaves are oriented away from the end of $s^-$. Define two functions $\partial^+\alpha$, $\partial^-\alpha \colon \{ \text{slits of } \lambda\} \to \R$ by the property that $\partial^+\alpha (s) = \partial\alpha(s^+)$ and $\partial^-\alpha (s) = \partial\alpha(s^-)$ for every slit $s$ of $\lambda$, where $\partial\alpha \colon \{ \text{slits of } \widehat\lambda\} \to \R$ is the boundary of $\alpha \in \CC(\widehat \lambda, \slits; \R)$. 
 
 If $T$, $T'$, $T''$ are three components of $\widetilde S - \widetilde \lambda$ such that $T''$ separates $T$ from $T'$ in $\widetilde S$, let $\widetilde s''$ be the spike of $T''$ delimited by the two sides of $T''$ that separate $T$ from $T'$, and let $s''$ be the projection of $\widetilde s''$ to $S$. The quasi-additivity of $\alpha \in \CC(\widehat \lambda, \slits; \R)$ then translates to the property that
 $$
 \alpha(T,T') = \alpha(T,T'') + \alpha (T'', T') - \partial^+\alpha(s'')
 $$
 if the spike $\widetilde s''$ of $T''$ points to the left as seen from $T$, and
 $$
 \alpha(T,T') = \alpha(T,T'') + \alpha (T'', T') + \partial^-\alpha(s'')
 $$
 if $\widetilde s''$ points to the right as seen from $T$.   
 
 The following statement is then automatic. 
 
 \begin{prop}
 \label{prop:RelTgtCycleDiffrentViewpoint}
The above construction provides a one-to-one correspondence between relative tangent cycles $\alpha \in \CC(\widehat \lambda, \slits; \R)$ and maps $\alpha$ associating a number $\alpha(T,T')\in \R$ to each ordered pair of  distinct components $T$ and $T'$  of $\widetilde S - \widetilde\lambda$ for which there exist two functions $\partial^\pm\alpha \colon \{ \text{slits of } \lambda\} \to \R$ with:
\begin{enumerate}
\item $\alpha$ is $\pi_1(S)$--invariant, in the sense that $\alpha(\gamma T, \gamma T') = \alpha(T,T')$ for every $\gamma \in \pi_1(S)$ and every pair of  distinct components $T$ and $T'$  of $\widetilde S - \widetilde\lambda$;
\item if $T''$ separates $T$ from $T'$ in $\widetilde S$, if $\widetilde s''$ is the spike of $T''$ delimited by the two sides of $T''$ that separate $T$ from $T'$, and if $s''$ is the slit of $\lambda$ defined by the projection of $\widetilde s''$ to $S$, then
 $$
 \alpha(T,T') = \alpha(T,T'') + \alpha (T'', T') - \partial^+\alpha(s'')
 $$
 if $\widetilde s''$ points to the left as seen from $T$, and
 $$
 \alpha(T,T') = \alpha(T,T'') + \alpha (T'', T') + \partial^-\alpha(s'')
 $$
 if $\widetilde s''$ points to the right as seen from $T$. 
\end{enumerate}

In addition, the boundary $\partial\alpha \colon \{ \text{slits of } \widehat \lambda\} \to \R$ is related to the functions $\partial^\pm\alpha \colon \{ \text{slits of } \lambda\} \to \R$  by the property that $\partial \alpha(s^\pm) = \partial^\pm \alpha(s)$ for every slit $s$ of $\lambda$ lifting to a positive slit $s^+$ and a negative slit $s^-$ of the orientation cover $\widetilde \lambda$. 
\qed
\end{prop}
 
 Proposition~\ref{prop:RelTgtCycleDiffrentViewpoint} has an immediate factor-by-factor extension to relative tangent cycles valued in $\R^{n-1}$. By restriction to the space of twisted relative tangent cycles $\CC( \lambda, \slits; \widehat\R^{n-1}) \subset \CC(\widehat \lambda, \slits; \R^{n-1})$, this automatically gives the following statement. Recall that $x \mapsto \overline x$ denotes the involution of $\R^{n-1}$ that sends $x=(x_1, x_2, \dots, x_{n-1})$ to $\overline x=(x_{n-1}, x_{n-2}, \dots, x_1)$. 
 
 \begin{prop}
  \label{prop:TwistedRelTgtCycleDiffrentViewpoint}
Proposition~{\upshape\ref{prop:RelTgtCycleDiffrentViewpoint}}  provides a one-to-one correspondence between twisted relative tangent cycles $\alpha \in \CC( \lambda, \slits; \widehat\R^{n-1})$ and maps $\alpha$ associating a vector $\alpha(T,T')\in \R^{n-1}$ to each ordered pair of  components $T$ and $T'$  of $\widetilde S - \widetilde\lambda$ such that there exists a function $\partial^+\alpha \colon \{ \text{slits of } \lambda\} \to \R^{n-1}$ with:
\begin{enumerate}
\item $\alpha$ is $\pi_1(S)$--equivariant, in the sense that $\alpha(\gamma T, \gamma T') = \alpha(T,T')$ for every $\gamma \in \pi_1(S)$ and every pair of  distinct components $T$ and $T'$  of $\widetilde S - \widetilde\lambda$;
\item if $T''$ separates $T$ from $T'$ in $\widetilde S$, if $\widetilde s''$ is the spike of $T''$ delimited by the two sides of $T''$ that separate $T$ from $T'$, and if $s''$ is the slit of $\lambda$ defined by the projection of $\widetilde s''$ to $S$, then
 $$
 \alpha(T,T') = \alpha(T,T'') + \alpha (T'', T') - \partial^+\alpha(s'')
 $$
 if $\widetilde s''$ points to the left as seen from $T$, and
 $$
 \alpha(T,T') = \alpha(T,T'') + \alpha (T'', T')  - \overline{ \partial^+\alpha(s'')}
 $$
 if $\widetilde s''$ points to the right as seen from $T$;
 
 \item
 $\alpha(T',T)=\overline{\alpha(T,T')}$ for every pair of  distinct components $T$ and $T'$  of $\widetilde S - \widetilde\lambda$.
 
\end{enumerate}

 In addition, the boundary $\partial\alpha \colon \{ \text{slits of } \widehat \lambda\} \to \R$ is related to the function $\partial^+\alpha \colon \{ \text{slits of } \lambda\} \to \R$  by the property that $\partial \alpha(s^+) = \partial^+ \alpha(s)$ and $\partial \alpha(s^-) =  - \overline{\partial^+\alpha(s)}$  for every slit $s$ of $\lambda$ lifting to a positive slit $s^+$ and a negative slit $s^-$ of the orientation cover $\widetilde \lambda$. 
 \qed
\end{prop}

Note that the function $\partial^-\alpha \colon \{ \text{slits of } \lambda\} \to \R^{n-1}$ that one would have expected in this case is equal to $\partial^-\alpha =- \overline{ \partial^+\alpha}$ by the third condition of Proposition~\ref{prop:TwistedRelTgtCycleDiffrentViewpoint}. In particular,   $\partial \alpha (s^\mp) = - \overline{\partial\alpha(s^\pm)}$ for every $\alpha \in \CC( \lambda, \slits; \widehat\R^{n-1})$ when  $s^+$ and $s^-$  are the positive and negative slits of the orientation cover $\widehat \lambda$ that lift the same slit $s$ of $\lambda$. 

\section{The shearing tangent cycle of a Hitchin character}
\label{bigsect:ShearingCycle}

We will now associate a twisted relative tangent cycle $\sigma^\rho \in \CC( \lambda, \slits; \widehat \R^{n-1}) \cong H_1(U,\delv U; \widetilde \R^{n-1})$  to each Hitchin character $\rho \in \Hit(S)$. The key ingredient of this construction is the slithering map introduced in the next section. 

\subsection{Slithering}
\label{subsect:Slithering}

The slithering construction is a higher dimensional analogue of the horocyclic foliation defined, in the case \cite{Thu1, Bon96} where $n=2$ , by a hyperbolic metric and a maximal geodesic lamination $\lambda$ on the surface $S$. 

Consider a Hitchin homomorphism $\rho \colon \pi_1(S) \to \PSL$, and its associated flag map $\mathcal F_\rho \colon \partial_\infty \widetilde S \to \Flag$ as in \S\ref{subsect:FlagCurve}.

In the universal cover $\widetilde S$, let $g$ be a leaf of the preimage $\widetilde \lambda \subset \widetilde S$ of the maximal geodesic lamination $\lambda\subset S$. Choose an arbitrary orientation for $g$, and let $x_+$ and $x_-$  be  its positive and negative endpoints, respectively.  By Theorem~\ref{thm:FlagCurvePositive}, the flag pair $\bigl(\mathcal F_\rho(x_+), \mathcal F_\rho(x_-)  \bigr)$ is generic. It therefore defines a decomposition of $\R^n$ as the direct sum of the lines $\widetilde L_a(g) = \mathcal F_\rho(x_+)^{(a)} \cap \mathcal F_\rho(x_-)^{(n-a+1)}$, as in \S\ref{subsect:FlagCurve}.

Note that reversing the orientation of $g$ exchanges $x_+$ and $x_-$, and therefore replaces $\widetilde L_a(g)$ by $\widetilde L_{n-a+1}(g)$. 

Now consider two leaves $g$ and $g'\subset \widetilde{\lambda}$. We say that $g$ and $g'$ are \emph{oriented in parallel} if exactly one of the orientations of $g$ and $g'$ coincides with the boundary orientation determined by the component of $S-g\cup g'$ that separates $g$ from $g'$.

\begin{prop}
\label{prop:Slithering}
There exists a unique family of linear isomorphisms $\Sigma_{gg'} \colon \R^n \to \R^n$, indexed by all pairs of leaves $g$, $g' \subset \widetilde \lambda$, such that:

\begin{enumerate}

\item $\Sigma_{gg} = \Id_{\R^n}$, $\Sigma_{g'g}=\big(\Sigma_{gg'} \big )^{-1} $, and  $\Sigma_{gg''} = \Sigma_{gg'} \circ \Sigma_{g'g''}$ when $g'$ separates $g$ from $g''$;

\item $\Sigma_{gg'}$  depends locally H\"older continuously on $g$ and $g'$; namely, the map $(g,g') \mapsto \Sigma_{gg'}$ is H\"older continuous on (the square of) any compact subset of the space of leaves of $\widetilde \lambda$;

\item if $g$ and $g'$ have an endpoint $x\in \partial_\infty \widetilde S$ in common and are oriented towards $x$, and if $E= \F_\rho(x)\in \Flag$, then $\Sigma_{gg'}$  sends each line $\widetilde L_a(g') $ to $ \widetilde L_a(g)$ and its  restriction $\widetilde L_a(g') \to \widetilde L_a(g)$ of $\Sigma_{gg'}$ is the composition of the two natural isomorphisms $\widetilde L_a(g') \cong E^{(a)}/E^{(a-1)} \cong \widetilde L_a(g)$.
\end{enumerate}

In addition, the maps $\Sigma_{gg'}$ satisfy
\begin{enumerate}\setcounter{enumi}{3}

\item if $g$ and $g'$ are oriented in parallel, $\Sigma_{gg'}$ sends each line $ \widetilde L_a(g')$ to the line $ \widetilde L_a(g)$;

\item $\Sigma_{gg'}\colon \R^n \to \R^n$ has determinant $+1$.
\end{enumerate}
\end{prop}

By definition, $\Sigma_{gg'} \colon \R^n \to \R^n $ is the \emph{slithering map} from the line decomposition $\R^n = \bigoplus_{a=1}^n \widetilde L_a(g')$ to the line decomposition $\R^n = \bigoplus_{a=1}^n \widetilde L_a(g)$. We will construct $\Sigma_{gg'}$ by sweeping through all the leaves of $\widetilde \lambda$ that separate $g$ from $g'$, and by composition of a (usually infinite)  sequence of pivot moves as in Condition~(3) of Proposition~\ref{prop:Slithering}. The terminology of ``slithering'' is motivated by the fact that, in general,  any small section of this sweep involves both pivot moves    to the left and pivot moves to the right\footnote{In particular,  this is unrelated to Thurston's notion \cite{Thu2, Cal} of ``slithering'' for foliations of 3--dimensional manifolds, beyond the analogy with the movements of a snake.}. 

Note that, although the line decomposition $\R^n = \bigoplus_{a=1}^n \widetilde L_a(g)$ depends on an orientation for the leaf $g$, the slithering map $\Sigma_{gg'} \colon \R^n \to \R^n $ is independent of a choice of orientation for $g$ or $g'$.

\begin{proof}
[Proof of Proposition~\ref{prop:Slithering}]
We will split the construction of the slithering map of Proposition~\ref{prop:Slithering} into several steps, including a few lemmas. 

Let $T$ be a component of $\widetilde S - \widetilde\lambda$ that separates $g$ from $g'$. It is a triangle since the geodesic lamination $\lambda$ is maximal, and two of its three sides separate $g$ from $g'$; among these two sides, let $g_T$ be the one that is closest to $g$, and $g_T'$ the one closest to $g'$. Define $\Sigma_T=\Sigma_{g_{T}g_T'}$ by Condition~(3) of  Proposition~\ref{prop:Slithering}. Namely, if $E_T=\F_\rho(x_T)\in \Flag$ is the image under the flag map $\F_\rho \colon \partial_\infty \widetilde S \to \Flag$ of the common endpoint $x_T\in \partial_\infty \widetilde S$ of $g_T$ and $g_T'$, the map $\Sigma_T=\Sigma_{g_Tg_T'}$ sends $\R^n = \bigoplus_{a=1}^n \widetilde L_a(g_T')$ to $\R^n = \bigoplus_{a=1}^n \widetilde L_a(g_T)$ by the property that its restriction $\widetilde L_a(g_T') \to \widetilde L_a(g_T)$ coincides with the composition of the natural isomorphisms $\widetilde L_a(g_T') \cong E_T^{(a)}/E_T^{(a-1)} \cong \widetilde L_a(g_T)$. Note that $\Sigma_T$ has determinant 1, namely belongs to $\SL$.

We will now define
$$
\Sigma_{gg'} = \overrightarrow{\prod_{T}}\, \Sigma_T
$$
as the composition of the maps $\Sigma_T=\Sigma_{g_Tg_T'} \colon \R^n \to \R^n$ as $T$ ranges over all components of $\widetilde S - \widetilde \lambda$ separating $g$ from $g'$. Of course, there usually are  infinitely many maps in this composition, and we also must be careful with the order in which we compose these maps; the arrow over the product symbol is here to remind us that this is an ordered product, if the components $T$ are ordered from $g$ to $g'$. To make sense of this composition, let $\mathcal T_{gg'}$ be the set of components of $\widetilde S - \widetilde\lambda$ that separate $g$ from $g'$. Let $ \mathcal{T}=\{T_1, T_2, \dots, T_m\}$ be a finite subset of $\mathcal T_{gg'}$, where the indexing is chosen so that each ideal triangle $T_j$ separates $g$ from $T_{j+1}$. We can then consider the finite composition
$$
\Sigma_{\mathcal{T}}= \Sigma_{T_1}\circ \Sigma_{T_2}\circ \cdots \circ \Sigma_{T_{m-1}}\circ \Sigma_{T_m} \in \SL.
$$
We will then show that $\Sigma_{\mathcal T}$ converges to some linear map $\Sigma_{gg'}\in\SL$ as  the finite subset $\mathcal{T}=\{T_1, T_2, \dots, T_m\}$ tends to the whole set $\mathcal T_{gg'}$ of those components of $\widetilde S - \widetilde\lambda$ which separate $g$ from $g'$.

The proof of convergence relies on the following estimate. Choose an arc $k\subset  \widetilde S$ that is tightly transverse to the geodesic lamination $\widetilde\lambda$, and crosses both $g$ and $g'$; for instance, we can choose $k$ to be a geodesic arc. 

In particular, for every component $T$ of $\widetilde S- \widetilde \lambda$ that separates $g$ from $g'$, $k \cap T$ consists of a single arc. 

Endow the space $\mathrm{End}(\R^n)$ of linear maps $\R^n \to \R^n$ with any of the classical norms $\left \Vert \ \right\Vert$ such that $\left\Vert\Id_{\R^n}\right\Vert=1$ and $\left\Vert \phi \circ \psi \right\Vert \leq \left\Vert \phi\right\Vert \left\Vert \psi\right \Vert$. Our estimates will also depend on the choice of a negatively curved metric $m$ on $S$ for which the leaves of $\lambda$ are geodesic. 

\begin{lem}
\label{lem:EstimateElemSlithering}
There exists constants $A$ and $\nu>0$ such that
$$
\left\Vert \Sigma_{T} - \Id_{\R^n} \right\Vert \leq A\, \ell(k\cap T)^\nu
$$
for every component $T$ of $\widetilde S- \widetilde \lambda$ that separates $g$ from $g'$,
where $\ell(\ )$ denotes the arc length for the auxilliary metric $m$. 
\end{lem}

\begin{proof}
Let $x_T$, $y_T$, $y_T'\in \partial_\infty \widetilde S$ denote the three vertices of the triangle $T$, in such a way that $x_T$ and $y_T$ are the endpoints of the side $g_T$ that is closest to $g$, and $x_T$, $y_T'$ are the endpoints of the side $g_T'$ closest to $g'$. Then $\Sigma_T=\Sigma_{g_T g_T'}$ depends only on the two generic flag pairs $\bigl( \F_\rho(x_T),  \F_\rho(y_T) \bigr)$ and $\bigl( \F_\rho(x_T),  \F_\rho(y_T')  \bigr)$. In fact, $\Sigma_T$ depends differentiably on these two flag pairs,  and these pairs stay in a compact subset  of the space of generic flag pairs (depending on $k$ and on the continuity of the flag curve $\F_\rho$). Therefore, 
$$
\left\Vert \Sigma_{T} - \Id_{\R^n} \right\Vert = O \Bigl (d \bigl( \F_\rho(y_T),  \F_\rho(y_T')  \bigr) \Bigr)$$
where $d( \ )$ is an arbitrary riemannian metric on $\Flag$. 

Since the flag curve $\F_\rho$ is H\"older continuous (Proposition~\ref{prop:FlagCurve}), 
$$
d \bigl( \F_\rho(y_T),  \F_\rho(y_T')  \bigr) = O  \bigl( d(y_T, y_T')^\nu  \bigr)
$$
for some H\"older exponent $\nu$. The required estimate then follows from an easy geometric argument showing that
$$d(y_T, y_T') = O \bigl ( \ell(k\cap T)  \bigr) ,$$
where the constant hidden in the symbol $O(\ )$ depends on a lower bound for the angle between the arc $k$ and the leaves of $\widehat\lambda$ that it crosses. 
\end{proof}

Note that the constant $A$ depends on the arc $k$. The H\"older exponent $\nu$ depends only on the flag curve $\F_\rho$. 

The second ingredient is a now classical property of geodesic laminations.
\begin{lem}
\label{lem:SumPowersLengthsGapsConverges}
As  $T$ ranges over all components of $\widetilde S- \widetilde \lambda$  separating $g$ from $g'$, the sum
$$
\sum_{T\in \mathcal T_{gg'}}  \ell(k\cap T)^\nu 
$$
is convergent for every $\nu>0$. 

More precisely, there is a function $r\colon \mathcal T_{gg'} \to \N$ and  constants $B$, $C$, $B'$, $C'>0$ such that
\begin{enumerate}
\item $B\mathrm e^{-Cr(T)} \leq \ell(k\cap T) \leq B'\mathrm e^{-C'r(T)}$ for every $T\in \mathcal T_{gg'}$;
\item for every $m\in \N$, the number of triangles $T\in \mathcal T_{gg'}$ with $r(T)=m$ is uniformly bounded, independently of $m$. 
\end{enumerate}
\end{lem}
\begin{proof}
See for instance Lemmas~4 and 5 of \cite{Bon96}, and compare \S \ref{subsect:PositiveIntersectionRevisited}. 
\end{proof}

We are now ready to show the convergence of the infinite product $ \overrightarrow{\prod}{}^{\phantom{I}}_{T}\, \Sigma_T $. 

Recall that $\mathcal T_{gg'}$ denotes the set of components of $\widetilde S - \widetilde\lambda$ that separate $g$ from $g'$ and that, for every finite subset $\mathcal{T}=\{T_1, T_2, \dots, T_m\}$ of $\mathcal T_{gg'}$ where the $T_i$ are ordered from $g$ to $g'$, 
$$
\Sigma_{\mathcal{T}}= \Sigma_{T_1}\circ \Sigma_{T_2}\circ \cdots \circ \Sigma_{T_{m-1}}\circ \Sigma_{T_m}.
$$

\begin{lem}
\label{lem:SlitheringBounded}
As $\mathcal T$ ranges over all finite subsets of $\mathcal T_{gg'}$, the matrices $\Sigma_{\mathcal{T}}$ remain uniformly bounded.
\end{lem}

\begin{proof}
If  $\mathcal{T}=\{T_1, T_2, \dots, T_m\}$, Lemma~\ref{lem:EstimateElemSlithering} shows that $\left\Vert \Sigma_{T_i} \right\Vert \leq 1 + A\, \ell(k\cap T)^\nu$ for some constants $A$, $\nu>0$. Then,
$$
\left\Vert \Sigma_{\mathcal{T}} \right \Vert 
\leq \prod_{i=1}^m \bigl( 1 + A\, \ell(k\cap T_i)^\nu \bigr)
\leq \prod_{T \in \mathcal T_{gg'}} \bigl( 1 + A\, \ell(k\cap T)^\nu \bigr)
<\infty
$$
where the finiteness of the second product follows from Lemma~\ref{lem:SumPowersLengthsGapsConverges}. 
\end{proof}

\begin{lem}
\label{lem:SlitheringExists}
As the finite subset $\mathcal T$ tends to $\mathcal T_{gg'}$, the limit
$$
\Sigma_{gg'} = \kern 5pt\overrightarrow{\prod_{\kern -10pt T\in \mathcal T_{gg'}\kern -10pt}} \kern 5pt\Sigma_T
=\lim_{\mathcal T \to \mathcal T_{gg'}} \Sigma_{\mathcal T} 
$$
exists in $\SL$.  
\end{lem}
\begin{proof} Let $\mathcal T = \{T_1, T_2, \dots, T_m\}$ be a finite subset of $\mathcal T_{gg'}$, where the $T_i$ are ordered from $g$ to $g'$. If  $\mathcal T' = \mathcal T \cup \{T\}  $ has one more element $T\in \mathcal T_{gg'}$ and if $T$  separates $T_i$ from $T_{i+1}$, set $\mathcal T_1 = \{T_1, T_2, \dots, T_i\}$ and $\mathcal T_2 = \{T_{i+1}, T_2, \dots, T_m\}$; then
$$
\left\Vert \Sigma_{\mathcal T'} - \Sigma_{\mathcal T}  \right\Vert
=
\left\Vert \Sigma_{\mathcal T_1}\circ (\Sigma_T-\Id_{\R^n}) \circ \Sigma_{\mathcal T_2}  \right\Vert
=O\bigl( \ell (k\cap T)^\nu \bigr)
$$
by Lemmas~\ref{lem:EstimateElemSlithering} and \ref{lem:SlitheringBounded}. Lemma~\ref{lem:SumPowersLengthsGapsConverges} then shows that, as $\mathcal T$ ranges over all finite subsets of $ \mathcal T_{gg'}$, the family of  maps $\Sigma_{\mathcal T}\in \SL$ satisfies the Cauchy Property. The limit therefore exists. 
\end{proof}

Having defined the slithering map $\Sigma_{gg'}\colon \R^n \to \R^n$, we now show that it satisfies the properties of Proposition~\ref{prop:Slithering}. We begin with Condition~(1). 

\begin{lem}
\label{lem:ComposingSlithering}
For any two leaves $g$, $g'$ of $\widetilde\lambda$, $\Sigma_{gg}=\Id_{\R^n}$ and $\Sigma_{g'g}=\Sigma_{gg'}^{-1}$. In addition, $\Sigma_{gg''} = \Sigma_{gg'} \circ \Sigma_{g'g''}$ when one of the three leaves $g$, $g'$, $g''$ separates the other two.
\end{lem}

\begin{proof}
The first two properties are immediate from definitions. When $g'$ separates $g$ from $g''$, $\mathcal T_{gg''}$ is the disjoint union of $\mathcal T_{gg'}$ and $\mathcal T_{g'g''}$ and the property that $\Sigma_{gg''} = \Sigma_{gg'} \circ \Sigma_{g'g''}$ is again an immediate consequence of the construction. The other two cases follow from this one by an algebraic manipulation. 
\end{proof}

We now turn to Condition~(2). 

\begin{lem}
\label{lem:SlitheringHolderContinuous}
The slithering map $\Sigma_{gg'}$ provided by Lemma~{\upshape\ref{lem:SlitheringExists}}  depends  H\"older continuously on the leaves $g$ and $g'\subset \widetilde{\lambda}$ meeting the tightly transverse arc $k$. 
\end{lem}

\begin{proof}
If the leaf $h$ is close to $g$, and if the leaf $h'$ is close to the leaf $g'$, we can apply Lemma~\ref{lem:ComposingSlithering} to decompose $\Sigma_{hh'}$ as 
$$
\Sigma_{hh'} = \Sigma_{hg} \circ \Sigma_{gg'} \circ \Sigma_{g'h'}. 
$$

The argument used in the proof of Lemma~\ref{lem:SlitheringExists} shows that, for some $\nu>0$, 
$$
\left\Vert \Sigma_{hg}-\Id_{\R^n} \right\Vert = O \Biggl( \sum_{T\in \mathcal T_{hg}} \ell(k\cap T)^\nu \Biggr). 
$$
By  Lemma~\ref{lem:SumPowersLengthsGapsConverges}, the above series is dominated by a geometric series and, using the precise estimate provided by the second half of that statement, 
$$
 \sum_{T\in \mathcal T_{hg}} \ell(k\cap T)^\nu = O \Bigl( \max_{T\in \mathcal T_{hg}} \ell(k\cap T)^{\nu'}  \Bigr)=  O \bigl( \ell(k_{hg})^{\nu'}  \bigr) = O \bigl( d(h, g)^{\nu'}  \bigr) 
$$
for $\nu' = \nu \frac{C'}C$ with the constants $C$, $C'>0$ of Lemma~\ref{lem:SumPowersLengthsGapsConverges}, and where $k_{hg}$ is the subarc of $k$ that joins the two points $k\cap g$ and $k \cap h$. Therefore,
$$
\left\Vert \Sigma_{hg}-\Id_{\R^n} \right\Vert =  O \bigl( d(g,h)^{\nu'}  \bigr) .
$$

Similarly,
$$
\left\Vert \Sigma_{g'h'}-\Id_{\R^n} \right\Vert =  O \bigl( d(g',h')^{\nu'}  \bigr) .
$$

Combining these two estimates with the bound provided by Lemma~\ref{lem:SlitheringBounded},
\begin{align*}
\left\Vert \Sigma_{hh'}- \Sigma_{gg'}\right\Vert
&\leq  \left \Vert \Sigma_{hh'}- \Sigma_{gh'}\right\Vert + \left\Vert \Sigma_{gh'}- \Sigma_{gg'}\right\Vert \\
&\leq   \left\Vert \Sigma_{hg}- \Id_{\R^n} \right\Vert   \left\Vert\Sigma_{gh'}\right\Vert +     \left\Vert\Sigma_{gg'}\right\Vert \left \Vert \Sigma_{g'h'}- \Id_{\R^n} \right\Vert \\
& =  O \bigl( d(g,h)^{\nu'} + d(g',h')^{\nu'}  \bigr) ,
\end{align*}
which proves that the map $(g,g') \mapsto \Sigma_{gg'}\in \SL$ is H\"older continuous over the square of the space of leaves of $\widetilde\lambda$ that cross the arc $k$. 
\end{proof}

Lemma~\ref{lem:SlitheringHolderContinuous}  proves the local H\"older continuity Condition (2) of Proposition~\ref{prop:Slithering}. 

If the leaves $g$ and $g'$ share a common endpoint $x\in \partial _\infty \widetilde S$, then all leaves of $\widetilde\lambda$ that separate $g$ from $g'$ also have $x$ as an endpoint. In particular, $\Sigma_{gg'}$ is defined as an infinite product of elementary slitherings $\Sigma_T = \Sigma_{g_Tg_T'}$ that respect the flag $E = \F_\rho(x)$ and act as the identity on each line $E^{(a)}/E^{(a-1)}$. It follows that $\Sigma_{gg'}$ satisfies the same property, which proves Condition~(3) of  Proposition~\ref{prop:Slithering}.

\begin{lem}
\label{lem:SlitheringLineDecompositions}
Suppose that the leaves $g$ and $g'\subset \widetilde{\lambda}$ are oriented in parallel. Then the slithering map $\Sigma_{gg'}$ provided by Lemma~{\upshape\ref{lem:SlitheringExists}}  sends  the line decomposition $\R^n = \bigoplus_{a=1}^n \widetilde L_a(g')$ to the line decomposition $\R^n = \bigoplus_{a=1}^n \widetilde L_a(g)$. 
\end{lem}

\begin{proof}
The strategy is to approximate by a finite lamination the part of $\widetilde\lambda$ that separates $g$ from $g'$. The slithering map associated to this finite lamination will send the line decomposition $\R^n = \bigoplus_{a=1}^n \widetilde L_a(g')$ to the line decomposition $\R^n = \bigoplus_{a=1}^n \widetilde L_a(g)$, and approximate the slithering map $\Sigma_{gg'}$. Passing to the limit in the approximation process will  conclude the proof. 

Let  $\mathcal{T}= \{T_1, T_2 \dots, T_m\}$ be a finite subset of $ \mathcal{T}_{gg'}$, where the $T_i$ are ordered from $g$ to $g'$. We  insert two triangles $U_i$ and $U_i'$ between $T_i$ and $T_{i+1}$ as follows. Recall that $g_{T_i}$ and $g_{T_i}'$ are the two sides of $T_i$ separating $g$ from $g'$,  with $g_{T_i}$ closest to $g$. Let $h_i$ be the geodesic of $\widetilde S$ that joins the left-hand side (as seen from $g$) endpoint of $g_{T_i}'$ to the right-hand side endpoint of $g_{T_{i+1}}^{\phantom{i}}$. The two geodesics $g_{T_i}'$ and $h_i$ are two sides of a unique ideal triangle $U_i \subset \widetilde S$, possibly reduced to a single geodesic when $g_{T_i}'= h_i$. We can similarly consider the ideal triangle $U_i'$, possibly reduced to a single geodesic, with sides $h_i$ and $g_{T_{i+1}}^{\phantom{i}}$. See Figure~\ref{fig:Diagonals Added}. The same construction  with the conventions that $g_{T_0}'=g$ and $g_{T_{m+1}}^{\phantom{i}}=g'$ also defines triangles $U_0$,  $U_0'$, $U_m$, $U_m'$. 

\begin{figure}[htbp]

\SetLabels
( .72*.33 ) $T_i $ \\
( .2* .78) $ T_{i+1}$ \\
( .85* .54) $U_i$ \\
(.15* .52) $ U_i'$ \\
(.73 *.2 ) $ g_{T_i}$ \\
( .75*.5 ) $g_{T_i}' $ \\
( .3*.66 ) $g_{T_{i+1}} $ \\
( .3* .93) $g_{T_{i+1}}' $ \\
( .55*.55 ) $h_i$ \\
\endSetLabels
\centerline{\AffixLabels{\includegraphics{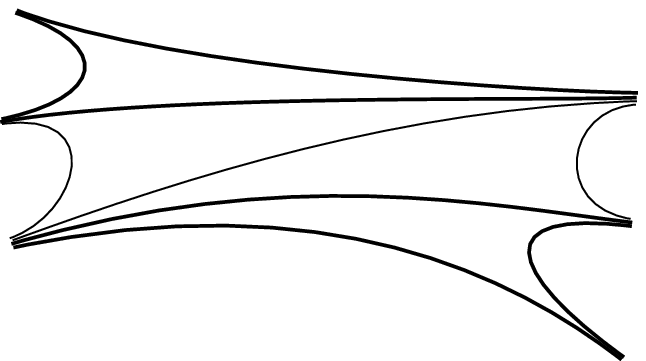}}}

\caption{}
\label{fig:Diagonals Added}
\end{figure}

As before, the triangles $U_i$ and $U_i'$ define an elementary slithering map $\Sigma_{U_i}$ sending  the line decomposition $\R^n = \bigoplus_{a=1}^n \widetilde L_a(h_i)$ to the line decomposition $\R^n = \bigoplus_{a=1}^n \widetilde L_a(g_{T_i}')$, and an elementary slithering map $\Sigma_{U_i'}$ sending  the line decomposition $\R^n = \bigoplus_{a=1}^n \widetilde L_a(g_{T_{i+1}})$ to the line decomposition $\R^n = \bigoplus_{a=1}^n \widetilde L_a(h_i)$. These slithering maps are equal to the identity when the corresponding triangles are reduced to geodesics. 

Now consider
\begin{align*}
\widehat\Sigma_{\mathcal T} =  (\Sigma_{U_0} \circ \Sigma_{U_0'}) \circ \Sigma_{T_1}& \circ (\Sigma_{U_1} \circ \Sigma_{U_1'}) \circ \Sigma_{T_2} \circ (\Sigma_{U_2} \circ \Sigma_{U_2'}) \circ \Sigma_{T_3} \circ \dots\\
&\dots  \circ \Sigma_{T_{m-1}} \circ (\Sigma_{U_{m-1}} \circ \Sigma_{U_{m-1}'}) \circ \Sigma_{T_m} \circ ( \Sigma_{U_m} \circ \Sigma_{U_m'}).
\end{align*}

By construction, $\widehat\Sigma_{\mathcal T} $ sends the line decomposition $\R^n = \bigoplus_{a=1}^n \widetilde L_a(g_{T_{m+1}}^{\phantom{i}}) = \bigoplus_{a=1}^n \widetilde L_a(g')$ to the line decomposition $\R^n = \bigoplus_{a=1}^n \widetilde L_a(g_{T_{0}}') = \bigoplus_{a=1}^n \widetilde L_a(g)$. 

To compare $\widehat\Sigma_{\mathcal T} $ and $\Sigma_{\mathcal T} $, choose an arc $k$ tightly transverse to $\widetilde\lambda$ and meeting both $g$ and $g'$. Then, Lemma~\ref{lem:EstimateElemSlithering} provides constants $A$, $\nu>0$ such that  $\left\Vert \Sigma_{U_i} - \Id_{\R^n} \right\Vert \leq A\, \ell(k\cap U_i)^\nu$ and $\left\Vert \Sigma_{U_i'} - \Id_{\R^n} \right\Vert \leq A\, \ell(k\cap U_i')^\nu$. 

We can assume that $\nu\leq 1$ without loss of generality. Then, with this condition, 
$$
\ell(k\cap U_i)^\nu \leq \ell\bigl (k\cap (U_i\cup U_i') \bigr)^\nu \leq \sum_{\kern 5ptT\in \mathcal T_{g_{T_i}'\kern -1pt g^{\phantom{i}}_{T_{i+1}}}\kern -15pt} \ell(k\cap T)^\nu
$$
where the sum is  over all components $T$ of $\widetilde S - \widetilde \lambda$ that separate $T_i$ from $T_{i+1}$. A similar estimate holds for $\ell(k\cap U_i')^\nu$. It follows that
$$
\left\Vert \Sigma_{U_i}\circ \Sigma_{U_i'} - \Id_{\R^n} \right\Vert = O \Biggl(
\sum_{\kern 5ptT\in \mathcal T_{g_{T_i}'\kern -1pt g^{\phantom{i}}_{T_{i+1}}}\kern -15pt} \ell(k\cap T)^\nu
\Biggr ).
$$

The arguments used in the proof of Lemmas~\ref{lem:SlitheringBounded} and \ref{lem:SlitheringExists} can then be applied to show that
$$
\left\Vert \widehat\Sigma_{\mathcal T} - \Sigma_{\mathcal T}  \right\Vert = O\Biggl(\,
\sum_{T\in \mathcal T_{gg'}-\mathcal T } \ell(k\cap T)^\nu
\Biggr ).
$$

Lemma~\ref{lem:SumPowersLengthsGapsConverges} then shows that $ \widehat\Sigma_{\mathcal T} $ and $\Sigma_{\mathcal T}  $ have the same limit as the finite subset $\mathcal T$ tends to $\mathcal T_{gg'}$. Therefore, $ \widehat\Sigma_{\mathcal T} $ also converges to the slithering map $\Sigma_{gg'}$ as $\mathcal T$ tends to $\mathcal T_{gg'}$.

We already observed that each $\widehat\Sigma_{\mathcal T} $ sends the line decomposition $\R^n = \bigoplus_{a=1}^n \widetilde L_a(g')$ to the line decomposition $\R^n =  \bigoplus_{a=1}^n \widetilde L_a(g)$. Passing to the limit, we conclude that $\Sigma_{gg'}$ has the same property. 
\end{proof}

Lemma~\ref{lem:SlitheringLineDecompositions} proves Condition~(4) of  Proposition~\ref{prop:Slithering}.

We already observed in Lemma~\ref{lem:SlitheringExists} that $\Sigma_{gg'}$ has determinant 1, which is Condition~(5). 

The only property of Proposition~\ref{prop:Slithering} remaining to prove is the uniqueness of the slithering map.

\begin{lem}
\label{lem:SlitheringUnique}
If a family of linear isomorphisms $\Sigma_{gg'}' \colon \R^n \to \R^n$, indexed by all pairs of leaves $g$, $g' \subset \widetilde \lambda$, satisfies Conditions~{\upshape(1--3)} of Proposition~{\upshape \ref{prop:Slithering}}, then $\Sigma_{gg'}' $ is equal to the map $\Sigma_{gg'}$ constructed above for every $g$, $g'$. 
\end{lem}
In particular, Conditions~(4--5) are consequences of Conditions~(1--3). 
\begin{proof} As usual, let $k$ be a tightly transverse  arc  that crosses both  $g$ and $g'$. Let $\mathcal T = \{ T_1, T_2, \dots, T_m \}$ be a finite subset of the set $\mathcal T_{gg'}$ of components of $\widetilde S - \widetilde \lambda$ that separate $g$ from $g'$, indexed in such a way that the $T_i$ occur in this order as one goes from $g$ to $g'$. Let $g_{T_i}^{\phantom{i}}$ and $g_{T_i}'$ be the sides of $T_i$ that are closest to $g$ and $g'$, respectively. 

By Condition~(1),
\begin{align*}
\Sigma_{gg'}' = \Sigma_{gg_{T_1}^{\phantom{i}}}' \circ \Sigma_{g_{T_1}^{\phantom{i}}g_{T_1}'}' &\circ \Sigma_{g_{T_1}'g_{T_2}^{\phantom{i}}}' \circ \Sigma_{g_{T_2}^{\phantom{i}}g_{T_2}'}' \circ   \dots \\
\dots &\circ \Sigma_{g_{T_{m-1}}^{\phantom{i}}g_{T_{m-1}}'}' \circ \Sigma_{g_{T_{m-1}}'g_{T_m}^{\phantom{i}}}' \circ \Sigma_{g_{T_m}^{\phantom{i}}g_{T_m}'}' \circ \Sigma_{g_{T_m}'g'}' .
\end{align*}

 Condition~(3)  implies that $\Sigma_{g_{T_i}^{\phantom{i}}g_{T_i}'}' = \Sigma_{g_{T_i}^{\phantom{i}}g_{T_i}'} = \Sigma_{T_i}$, so that 
\begin{align*}
\Sigma_{gg'}' = \Sigma_{gg_{T_1}^{\phantom{i}}}' \circ \Sigma_{T_1}  &\circ \Sigma_{g_{T_1}'g_{T_2}^{\phantom{i}}}' \circ\Sigma_{T_2}\circ   \dots \\
\dots &\circ \Sigma_{T_{m-1}}  \circ \Sigma_{g_{T_{m-1}}'g_{T_m}^{\phantom{i}}}' \circ \Sigma_{T_m} \circ \Sigma_{g_{T_m}'g'}' .
\end{align*}

By Condition~(2), the map $(h,h') \mapsto \Sigma_{hh'}'$ is H\"older continuous over the space of leaves of $\widetilde \lambda$ that meet the arc $k$. As a consequence, there exists a constant $\nu>0$ such that for every $i$
$$
\left\Vert \Sigma_{g_{T_i}'g_{T_{i+1}}^{\phantom{i}}}'   - \Id_{\R^n}\right\Vert 
=O\bigl( d(g_{T_i}'  ,   g_{T_{i+1}}^{\phantom{i}}) ^\nu \bigr).
$$
Because the leaves  $g_{T_i}' $  and $  g_{T_{i+1}}^{\phantom{i}}$ are disjoint, a classical estimate in negative curvature geometry (see for instance  \cite[\S 5.2.6]{CEG}) shows that $d(g_{T_i}'  ,   g_{T_{i+1}}^{\phantom{i}}) $ is bounded by a constant times the length of the subarc  $k_{g_{T_i}'   g_{T_{i+1}}^{\phantom{i}}} \subset k$ delimited by the points  $k \cap g_{T_i}'  ,$ and $k \cap  g_{T_{i+1}}^{\phantom{i}}$. The geodesic lamination $\widetilde\lambda$ has measure 0 (\cite[\S 8.5]{Thu0}\cite{BirSer}). Therefore,
$$
\ell(k_{g_{T_i}'   g_{T_{i+1}}^{\phantom{i}}}  ) = \sum_{\kern 5pt T\in \mathcal T_{g_{T_i}'   g_{T_{i+1}}^{\phantom{i}}} \kern -5pt} \ell (k\cap T).
$$

Assuming $\nu\leq 1$ without loss of generality, we can combine all these estimates and conclude that 
$$
\left\Vert \Sigma_{g_{T_i}'g_{T_{i+1}}^{\phantom{i}}}'   - \Id_{\R^n}\right\Vert 
=O\Biggl( \sum_{\kern 5pt T\in \mathcal T_{g_{T_i}'   g_{T_{i+1}}^{\phantom{i}}} \kern -5pt} \ell (k\cap T)^\nu \Biggr).
$$
This also holds for $i=0$ and $m$, with the convention that $g_{T_0}' =g$ and $g_{T_{m+1}}^{\phantom{i}} =g'$. 

From this estimate, we can then use the arguments of the proofs of Lemma~\ref{lem:SlitheringBounded} and \ref{lem:SlitheringExists} to show that
$$
\left\Vert \Sigma_{gg'}'   - \Sigma_{\mathcal T} \right\Vert 
=O\Biggl( \sum_{ T\in \mathcal T_{gg'} - \mathcal T } \ell (k\cap T)^\nu \Biggr).
$$
By Lemma~\ref{lem:SumPowersLengthsGapsConverges}, this proves that
$$
\Sigma_{gg'}'
=\lim_{\mathcal T \to \mathcal T_{gg'}} \Sigma_{\mathcal T} 
=\Sigma_{gg'},
$$
which concludes the proof of Lemma~\ref{lem:SlitheringUnique}.
\end{proof}

This uniqueness property completes the proof of Proposition~\ref{prop:Slithering}.
\end{proof}

\begin{rem}
\label{rem:SlitheringNotUniqueIfNotHolder}
In Proposition~\ref{prop:Slithering} (and in Lemma~\ref{lem:SlitheringUnique}), the uniqueness property would be false without the hypothesis that the slithering map $\Sigma_{gg'}$ depends locally H\"older continuously (and not just continuously) on the leaves $g$, $g'$. To understand why, let $\alpha_1$, $\alpha_2$, \dots, $\alpha_{n-1}$ be transverse measures for $\lambda$ such that $\alpha_{n-a}=\alpha_a$ for every $a$ (so that in practice we have $\lfloor \frac n2 \rfloor$ such $\alpha_a$); assume in addition that the  $\alpha_a$ have no atom (which is automatic if $\lambda$ has no closed leaf). 
For two leaves $g$, $g'$ of $\widetilde\lambda$, the atom-free hypothesis guarantees that the $\alpha_a$--mass $\alpha_a(g,g')$ of the set of leaves of $\widetilde \lambda$ separating $g$ from $g'$ depends continuously on $g$ and $g'$. Define $\beta_1(g,g')$, $\beta_2(g,g')$, \dots, $\beta_n(g,g')$ by the property that $\alpha_a(g,g') = \beta_{a+1}(g,g') - \beta_a(g,g')$ and $\sum_{a=1}^n \beta(g,g') =0$.  If $g$ and $g'$ are oriented in parallel in such a way that $g'$ is to the left of $g$, let  $\Sigma_{gg'}' \colon \R^n \to \R^n$ be obtained by postcomposing the slithering map $\Sigma_{gg'}$ with the linear map that respects each line $\widetilde L_a(g)$ and acts by $\E^{\beta_a(g,g')}$ on $\widetilde L_a(g)$. This new family of maps $\Sigma_{gg'}'$ satisfies Conditions~(1) and (3--5) of  Proposition~\ref{prop:Slithering}, the maps $\Sigma_{gg'}'$ depend continuously (but not locally H\"older continuously) on $g$ and $g'$, and they are of course different from the original family of slithering maps $\Sigma_{gg'}$ if at least one of the $\alpha_a$ is non-zero. 

This construction automatically generalizes to the situation where the $\alpha_a$ are topological differential forms in the sense of \cite{Kenyon}, in which case it completely describes how the uniqueness can fail if we remove the H\"older condition from Proposition~\ref{prop:Slithering}.
\end{rem}

\subsection{The shearing cycle}
\label{subsect:ShearingCycle}

We now use the slithering map to associate to the Hitchin homomorphism $\rho\colon\pi_1(S) \to \PSL$ a certain twisted   tangent cycle $\sigma^\rho \in \CC(\lambda, \slits; \widehat\R^{n-1})$ relative to the slits of $\lambda$. This relative tangent cycle is the \emph{shearing  cycle} of the Hitchin homomorphism $\rho$. 

We will use the point of view of \S \ref{subsect:RelTgtCyclesDifferentView}. Let $T$ and $T'$ be two components of $\widetilde S - \widetilde\lambda$. 

Let $g$ be the side of $T$ that is closest to $T'$, and let $g'$ be the side of $T'$ closest to $T$. We orient these two leaves of $\widetilde \lambda$ to the left as seen from $T$. In particular, $g$ and $g'$ are oriented in parallel, and the slithering map $\Sigma_{gg'} \colon \R^n \to \R^n$ of Proposition~\ref{prop:Slithering} sends each line $\widetilde L_a(g')$ to the line $\widetilde L_a(g)$.  

\begin{figure}[htbp]

\SetLabels
( .5* .22) $ T$ \\
( -.02* .21) $ x$ \\
( 1.03* .28) $y $ \\
( .37* -.04 ) $z $ \\
( .5* .4) $ g$ \\
( .6* .78) $ T'$ \\
( -.03* .68) $ x'$ \\
(1.05 * .62) $y '$ \\
( .72* 1.02) $z '$ \\
( .5*.6 ) $ g'$ \\
\endSetLabels
\centerline{\AffixLabels{\includegraphics{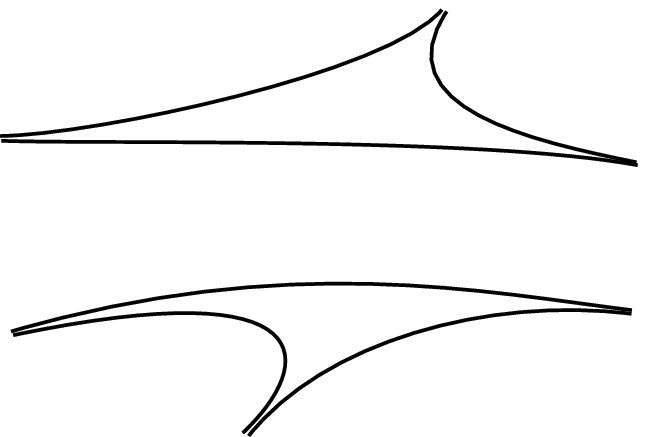}}}

\caption{}
\label{fig:ShearingCycles}
\end{figure}

Let $x$ and $y\in \partial_\infty \widetilde S$  be  the positive and  negative endpoints of  $g$, and let $z$ be the third vertex of the ideal triangle $T$. Similarly, let  $x'$ and $y'\in \partial_\infty \widetilde S$  be  the positive and  negative endpoints of  $g'$, and let $z'$ be the third vertex of  $T'$. See Figure~\ref{fig:ShearingCycles}. The flag curve $\F_\rho \colon \partial_\infty \widetilde S \to \Flag$ of Proposition~\ref{prop:FlagCurve} now associates six flags $\F_\rho(x)$, $\F_\rho(y)$, $\F_\rho(z)$, $\F_\rho(x')$, $\F_\rho(y')$ and $\F_\rho(z')\in \Flag$ to these vertices. By our definitions, the slithering map $\Sigma_{gg'}$ sends $\F_\rho(x')$ to $\F_\rho(x)$ and $\F_\rho(y')$ to $\F_\rho(y)$. 

We want to consider the double ratio
$
D_a \Bigl(
\F_\rho(x), \F_\rho(y), \F_\rho(z), \Sigma_{gg'}\bigl( \F_\rho(z') \bigr)
\Bigr)
$, as in \S\ref{subsect:DoubleRatios}. 

\begin{lem}
\label{lem:ShearAlongArcDefined}
The double ratio $
D_a \Bigl(
\F_\rho(x), \F_\rho(y), \F_\rho(z), \Sigma_{gg'}\bigl( \F_\rho(z') \bigr)
\Bigr)
$ is finite and positive. 
\end{lem}

\begin{proof}
When $T$ and $T'$ are adjacent so that  $g=g'$, then $\Sigma_{gg'}=\Id_{\R^n}$ and the statement is   an immediate consequence of  the positivity property of Theorem~\ref{thm:FlagCurvePositive}. In the general case, however, the appearance of the slithering map $\Sigma_{gg'}$ requires a more elaborate argument. 

The key ingredient is a deeper consequence of the positivity property, which
 is that the line bundles $L_a \to T^1S$ of Theorem~\ref{thm:AnosovProperty} carry a canonical  joint orientation. This does not mean that each individual bundle $L_a$ has a preferred orientation, but that the collection of  all $L_a$ carry orientations that are uniquely determined up to simultaneous reversal of all orientations; in other words, all line bundles $L_a \otimes L_{a+1} \to T^1 S$ admit canonical orientations. Actually,  we will see that the line bundles $L_a$  admit two equally canonical but opposite joint orientations: the left-hand-side and right-hand-side joint orientations. 

To define these joint orientations, focus attention on  a point $\widetilde u \in T^1 \widetilde S$. As in \S \ref{subsect:FlagCurve}, consider the line decomposition $\R^n = \bigoplus_{a=1}^n \widetilde L_a(\widetilde u)$ defined by the fibers over $\widetilde u$ of the line bundles $\widetilde L_a \to T^1 \widetilde S$ lifting the bundles $L_a \to T^1S$. Then, if  $p$ and $q\in \partial_\infty \widetilde S$  are  the positive and  negative endpoints of the orbit $g$ of $\widetilde u$ under the geodesic flow, $\widetilde L_a(\widetilde u) = \mathcal F_\rho(p)^{(a)} \cap \mathcal F_\rho(q)^{(n-a+1)}$ by definition of the flag curve $\F_\rho$ in Proposition~\ref{prop:FlagCurve}. Consider another point $r\in \partial_\infty \widetilde S$ that is different from $p$ and $q$, and that sits to the left of $p$ as seen from $q$.
 By Theorem~\ref{thm:FlagCurvePositive}, the flag triple $\bigl( \F_\rho(p), \F_\rho(q), \F_\rho(r) \bigr)$ is generic. As a consequence, if $v$ is a nontrivial vector in the line $\F_\rho(r)^{(1)}$, the projection of $v\in \R^n = \bigoplus_{b=1}^n \widetilde L_b(\widetilde u)$ to the line $\widetilde L_a(\widetilde u)$ parallel to all $\widetilde  L_b(\widetilde u)$ with $b\neq a$ is nontrivial, and therefore specifies an orientation for $\widetilde L_a(\widetilde u)$. Replacing $v$ by any other non-trivial vector $v' \in \F_\rho(r)^{(1)}$ determines the same orientation on $\widetilde L_a(\widetilde u)$  if the ratio $\frac v{v'}$ in the line $\F_\rho(r)^{(1)}$  is positive, or reverses all these orientations if $\frac v{v'}<0$. Therefore the joint orientation of the lines $\widetilde L_a(\widetilde u)$ is independent of the choice of $v \in \F_\rho(r)^{(1)}$. 

To show that the joint orientation of the lines $\widetilde L_a(\widetilde u)$ is independent of the choice of  the point $r\in \partial_\infty \widetilde S$, consider another point $r'\in \partial_\infty \widetilde S$ different from $p$ and $q$, and now located on the right of $p$ as seen from $q$. This point $r'$ similarly defines a joint orientation for the lines $\widetilde L_a(\widetilde u)$, and we will see that this joint orientation is exactly the opposite of that defined by $r$. To prove this, pick nontrivial vectors $v \in \F_\rho(r)^{(1)}$ and $v' \in \F_\rho(r')^{(1)}$. Let $v_a$ and $v_a'$ denote the respective projections of $v$ and $v'$ to the line $\widetilde L_a(\widetilde u)$ parallel to all $\widetilde  L_b(\widetilde u)$ with $b\neq a$. If, in addition, $r$ and $r'$ are in different components of $ \partial_\infty \widetilde S - \{p,q\}$, the positivity condition of Theorem~\ref{thm:FlagCurvePositive} and the definition of the double ratio can be combined to show that
$$
0< D_a \bigl(
\F_\rho(p), \F_\rho(q), \F_\rho(r),  \F_\rho(r') 
\bigr)
=
-\frac{v_{a+1}}{v_{a+1}'}\frac{v_{a}'}{v_{a}}
$$
where the ratios $\frac{v_b'}{v_b}\in \R-\{0\}$ are computed in the lines $\widetilde L_b(\widetilde u)$. As a consequence,  $v$ and $v'$ induce opposite orientations on the lines  $\widetilde L_a(\widetilde u) \otimes \widetilde L_{a+1}(\widetilde u)$. In other words, the joint orientation of the lines $\widetilde L_a(\widetilde u)$ defined by the point $r'\in \partial_\infty \widetilde S$ is the opposite of that defined by $r$. It immediately follows that the joint orientation defined by $r$ is independent of the choice of $r$ in the left-hand-side component of $ \partial_\infty \widetilde S- \{ p,q\}$ (as seen from~$q$). 

We will refer to the joint orientation defined by $r$ as the \emph{left-hand-side joint orientation} of the lines  $\widetilde L_a(\widetilde u)$, whereas the  \emph{right-hand-side joint orientation} will be the one defined by $r'$. These two joint orientations are opposite of each other.

Let $h$ and $h'$ be two oriented geodesics of $\widetilde S$ that share the same positive endpoint $p\in \partial_\infty \widetilde S$, and let $\Sigma_{hh'} \colon \R^n \to \R^n$ be the elementary slithering map, sending each line $\widetilde L_a(h')$ to $\widetilde L_a(h)$, defined as in Proposition~\ref{prop:Slithering}(4). The definition of  $\Sigma_{hh'} $ through the isomorphisms $\widetilde L_a(h') \cong \mathcal F^\rho(p)^{(a)} /  \mathcal F^\rho(p)^{(a-1)} \cong \widetilde L_a(h) $ makes it clear that  $\Sigma_{hh'}$ sends the left-hand-side joint orientation of the family of  lines $\widetilde L_a(h')$ to the left-hand-side  joint orientation of the $\widetilde L_a(h)$. 

We now return to the leaves $g$, $g'$ of $\widetilde\lambda$. As in the proof of Lemma~\ref{lem:SlitheringLineDecompositions} and with the notation used there, approximate the part of $\widetilde \lambda$ that separates $g$ and $g'$ by a finite lamination, and the slithering map $\Sigma_{gg'}$ by a finite composition
\begin{align*}
\widehat\Sigma_{\mathcal T} =  (\Sigma_{U_0} \circ \Sigma_{U_0'}) \circ \Sigma_{T_1}& \circ (\Sigma_{U_1} \circ \Sigma_{U_1'}) \circ \Sigma_{T_2} \circ (\Sigma_{U_2} \circ \Sigma_{U_2'}) \circ \Sigma_{T_3} \circ \dots\\
&\dots  \circ \Sigma_{T_{m-1}} \circ (\Sigma_{U_{m-1}} \circ \Sigma_{U_{m-1}'}) \circ \Sigma_{T_m} \circ ( \Sigma_{U_m} \circ \Sigma_{U_m'}).
\end{align*}
of elementary slitherings where, for any to consecutive terms, the corresponding triangles $T_i$ and $U_i$, or $U_i$ and $U_i'$, or $U_i'$ and $T_{i+1}$, share a side $g_{T_i}^{\phantom i}$, $h_i$ or $g_{T_{i+1}}'$, respectively. By our earlier observation, each of these elementary slitherings respects joint orientations of the appropriate families of lines. It follows that $\widehat\Sigma_{\mathcal T} $ sends the joint orientation of the lines  $\widetilde L_a(g')$ to the joint orientation of the $\widetilde L_a(g)$. Passing to the limit as the approximation $\widehat\Sigma_{\mathcal T}$ tends to $\Sigma_{gg'}$, we concluce that the slithering map $\Sigma_{gg'}$ sends the  left-hand-side  joint orientation of the lines  $\widetilde L_a(g')$ to the  left-hand-side joint orientation of the $\widetilde L_a(g)$. 

We are now ready to determine the sign of  the double ratio $
D_a \Bigl(
\F_\rho(x), \F_\rho(y), \F_\rho(z), \Sigma_{gg'}\bigl( \F_\rho(z') \bigr)
\Bigr)
$. Pick nontrivial vectors $v$ and $v'$ in the lines $\F_\rho(z)^{(1)}$ and  $\F_\rho(z')^{(1)}$, respectively. The left-hand-side joint orientation of the family of lines $\widetilde L_a(g)$ is defined by the projections $v_a$ of $v$ to $\widetilde L_a(g)$ parallel to the other lines $\widetilde L_b(g)$ with $b\neq a$. Similarly, the right-hand-side joint orientation of the  lines $\widetilde L_a(g')$ is defined by the projections $v_a'$ of $v'$ to $\widetilde L_a(g')$ parallel to the  lines $\widetilde L_b(g')$ with $b\neq a$. Since we just proved that the slithering map $\Sigma_{gg'}$ respects joint orientations, and since the left- and right-hand-side orientations are opposite of each other, the joint orientation of the $\widetilde L_a(g)$ by the vectors $v_a$ is opposite to that  defined by the   vectors $\Sigma_{gg'}(v_a') $. In other words, all ratios $\frac{\Sigma_{gg'}(v_{a}')}{v_{a}} \frac{v_{a+1}}{\Sigma_{gg'}(v_{a+1}')}
$ are negative. 
By definition of the double ratio,
$$
D_a \Bigl(
\F_\rho(x), \F_\rho(y), \F_\rho(z), \Sigma_{gg'}\bigl( \F_\rho(z') \bigr)
\Bigr)
=
-\frac{v_{a+1}}{\Sigma_{gg'}(v_{a+1}')}
\frac{\Sigma_{gg'}(v_{a}')}{v_{a}}
>0
$$
which concludes the proof of Lemma~\ref{lem:ShearAlongArcDefined}. 
\end{proof}

Lemma~\ref{lem:ShearAlongArcDefined} enables us to define the \emph{$a$--th shear parameter} of the Hitchin homomorphism $\rho$ between the components $T$ and $T'$ of $\widetilde S - \widetilde\lambda$ as
$$
\sigma_a^\rho (T,T') = \log D_a \Bigl(
\F_\rho(x), \F_\rho(y), \F_\rho(z), \Sigma_{gg'}\bigl( \F_\rho(z') \bigr)
\Bigr) \in \R.
$$
These shear parameters are then combined in the \emph{shear vector}
 $$\sigma^\rho(T,T')=\big(\sigma_1^\rho(T,T'), \sigma_2(T,T'), \dots ,\sigma_{n-1}^\rho(T,T')\big )\in \R^{n-1}.$$

We now show that the family of  shear vectors $\sigma^\rho(T,T')$ define a  relative tangent cycle $\sigma^\rho \in \CC(\lambda, \slits; \widehat \R^{n-1})$ for $\lambda$ valued in the twisted coefficient bundle $\widehat \R^n$, as in Proposition~\ref{prop:TwistedRelTgtCycleDiffrentViewpoint}. We begin with the easier part, namely Condition~(3) of that statement.

\begin{lem}
\label{lem:ReversingOrientationShear}
For any two components $T$ and $T'$ of $\widetilde S - \widetilde\lambda$, 
$$
\sigma_a^\rho(T', T) = \sigma_{n-a}^\rho (T,T').
$$
\end{lem}

\begin{proof}
Using the notation of Figure~\ref{fig:ShearingCycles}, 
\begin{align*}
\sigma_a^\rho(T', T) 
&= \log D_a \Bigl(
\F_\rho(y'), \F_\rho(x'), \F_\rho(z'), \Sigma_{g'g}\bigl( \F_\rho(z) \bigr)
\Bigr)\\
&= \log D_{n-a} \Bigl(        
\F_\rho(x'), \F_\rho(y'), \Sigma_{g'g}\bigl( \F_\rho(z)\bigr),  \F_\rho(z')  
\Bigr)\\
&= \log D_{n-a} \Bigl(
\Sigma_{g'g}\bigl( \F_\rho(x)\bigr), \Sigma_{g'g}\bigl( \F_\rho(y)\bigr), \Sigma_{g'g}\bigl( \F_\rho(z)\bigr),  \F_\rho(z')  
\Bigr)\\
&= \log D_{n-a} \Bigl(
 \F_\rho(x),  \F_\rho(y),  \F_\rho(z) , \Sigma_{gg'}\bigl( \F_\rho(z')\bigr) 
\Bigr) = \sigma_{n-a}^\rho(T,T'),
\end{align*}
where the second equality is a consequence of  the elementary properties of  double ratios stated in Lemma~\ref{lem:RelationsDoubleRatios}, the third equality comes from the fact that $\Sigma_{g'g}$ sends each line $\widetilde L_b(g)$ to $\widetilde L_b(g')$, and the fourth equality follows from the invariance of double ratios under the action of $\Sigma_{gg'}=\Sigma_{g'g}^{-1}\in \SL$. 
\end{proof}

Let $s$ be a slit of $\lambda$ or, equivalently, a spike of the complement $S-\lambda$. Lift $s$ to a spike of $\widetilde S - \widetilde \lambda$, namely to a vertex $x\in \partial_\infty \widetilde S$ of a triangle component $T$ of $\widetilde S - \widetilde \lambda$. Let $y$ and $z$ be the other two vertices of $T$, indexed so that $x$, $y$ and $z$ occur in this order counterclockwise around $T$. The flag curve $\F_\rho$ then determines a positive triple of flags $\F_\rho(x)$, $\F_\rho(y)$ and $\F_\rho(z)\in \Flag$. Considering their quadruple ratios as in \S \ref{subsect:QuadRatios}, define
$$
\theta^\rho_a(s) = \log Q_a \bigl (\F_\rho(x), \F_\rho(y), \F_\rho(z) \bigr) ,
$$
which is clearly independent of the lift of the slit $s$ to the universal cover $\widetilde S$.

Lemma~\ref{lem:ExpressQuadrupleRatioTripleRatios} expresses $\theta_a^\rho(s)$ in terms of the triangle invariants $\tau^\rho_{abc}(s)$ of $\rho$. 
\begin{lem}
\label{lem:ExpressThetaTriangleInvariants}
\pushQED{\qed}
\begin{equation*}
\theta^\rho_a(s) = \sum_{b+c=n-a} \tau^\rho_{abc}(s).
\qedhere
\end{equation*}
\end{lem}

Recall that by definition a slit $\widehat s$ of the orientation cover $\widehat\lambda$ is positive if the canonical orientation of $\widehat\lambda$ orients the two leaves that are adjacent to $\widehat s$ towards $\widehat s$, and that $\widehat s$ is negative when these two leaves are oriented away from~$\widehat s$. 

\begin{lem}
\label{lem:ShearingCycleIsCycle}
The rule $(T,T') \mapsto \sigma_a^\rho(T,T')$ defines a relative tangent cycle  $\sigma_a^\rho \in \CC(\widehat \lambda, \slits; \R)$. The boundary $\partial \sigma_a^\rho \colon \{ \text{slits of } \widehat \lambda \} \to \R$ is defined by the property that, for every slit $\widehat s$ of $\widehat\lambda$ projecting to a slit $s$ of $\lambda$, 
$$
\partial \sigma_a^\rho (\widehat s) = 
\begin{cases}
\theta^\rho_a(s) &\text{ if } \widehat s \text{ is a positive slit of } \widehat \lambda,\\
-\theta^\rho_{n-a}(s) &\text{ if } \widehat s \text{ is negative. }
\end{cases}
$$. 
\end{lem}

\begin{proof} Using the framework of Proposition~\ref{prop:RelTgtCycleDiffrentViewpoint}, let $T$, $T'$, $T''$ be three components of $\widetilde S - \widetilde \lambda$ such that $T''$ separates $T$ from $T'$ in $\widetilde S$. Let $\widetilde s''$ be the spike of $T''$ delimited by the two sides of $T''$ that separate $T$ from $T'$. 

We first consider the case where $\widetilde s''$ points to the left as seen from $T$. 

Let $g$ be the side of $T$ that is closest to $T'$ and $T''$, and let $g'$ be the side of $T'$ that is closest to $T$ and $T''$. Let $f$ be the side of $T''$ that faces $T$, and let $f'$ be the side of $T''$ that faces $T'$. Orient these leaves of $\widetilde \lambda$ to the left as seen from $T$. Let $E$, $F$, $E'$, $F'$, $E''$, $H$, $H' \in \Flag$ be the flags respectively associated by the flag  curve $\F_\rho \colon \partial_\infty \widetilde S \to \Flag$ to the positive endpoint of $g$, the negative endpoint of $g$, the positive endpoint of $g'$, the negative endpoint of $g'$, the positive endpoint $\widetilde s''$ of $f$ and $f'$, the negative endpoint of $f$, and the negative endpoint of $f'$. Similarly, let $G$, $G'\in \Flag$ be respectively associated to the vertex of $T$ that is not contained in $g$, and to the vertex of $T'$ that is not contained in $g'$. See Figure~\ref{fig:ShearingCycle2}, where the vertices of $T$, $T'$, $T''$ are labelled by the flags associated to them by the flag curve $\F_\rho$. 

\begin{figure}[htbp]

\SetLabels  
( -.02*.17 ) $ E$ \\
(1.02 * .21) $ F$ \\
( .35* -.05) $G $ \\
( -.04* .74) $E' $ \\
( 1.04* .69) $F' $ \\
( .71* 1.02) $ G'$ \\
( -.05* .48) $E'' $ \\
( 1.03* .41) $ H$ \\
( 1.05* .61) $ H'$ \\
( .55* .315) $ g$ \\
(.45 *.69 ) $g' $ \\
(.7 * .42) $ f$ \\
( .7*.58 ) $f' $ \\
(.5 *.2 ) $ T$ \\
( .58* .81) $T' $ \\
( .86* .51) $T''$ \\
\endSetLabels
\centerline{\AffixLabels{\includegraphics{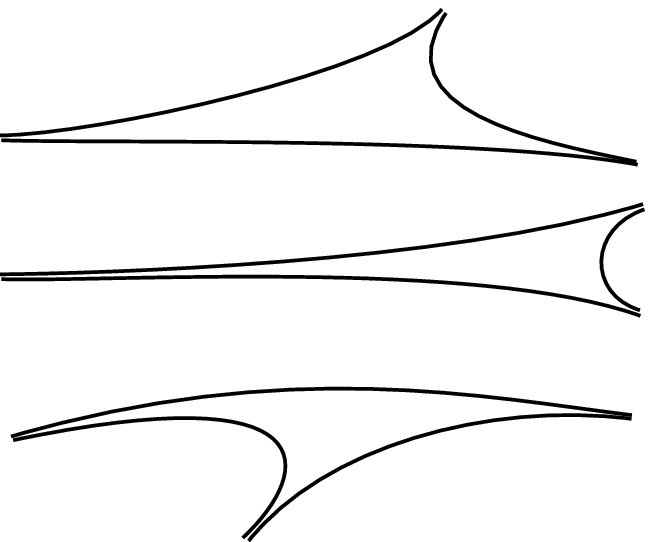}}}
\caption{}
\label{fig:ShearingCycle2}
\end{figure}

Then,
\begin{align*}
\sigma_a^\rho (T,T')&= \log D_a \bigl( E, F, G, \Sigma_{gg'}(G') \bigr)\\
&= \log D_a \bigl( E'', H, \Sigma_{fg}(G), \Sigma_{fg'}(G') \bigr)
\end{align*}
by using the fact that the slithering map $ \Sigma_{fg}$ sends $E$ to $E''$ and $F$ to $H$. 
Similarly,
\begin{align*}
\sigma_a^\rho (T, T'')&= \log D_a \bigl( E, F, G, \Sigma_{gf}(H') \bigr)\\
&= \log D_a \bigl( E'', H, \Sigma_{fg}(G), H' \bigr)
\end{align*}
and
\begin{align*}
\sigma_a^\rho (T'', T')&= \log D_a \bigl( E'', H', H, \Sigma_{f'g'}(G') \bigr)\\
&= \log D_a \bigl( E'', H, \Sigma_{ff'}(H), \Sigma_{fg'}(G') \bigr).
\end{align*}

Using the elementary properties of double ratios stated in  Lemma~\ref{lem:RelationsDoubleRatios}, it follows that
$$
\sigma_a^\rho (T, T') =\sigma_a^\rho (T, T'') + \sigma_a^\rho (T'', T')
+ \log D_a\bigl (E'', H, H'  ,  \Sigma_{ff'}(H)\bigr).
$$

By definition of the double product, 
$$D_a\bigl (E'', H, H'  ,  \Sigma_{ff'}(H)\bigr)= 
- \frac{e''{}^{(a)}\wedge h^{(n-a-1)} \wedge h'{}^{(1)}}
 {e''{}^{(a)}\wedge h^{(n-a-1)} \wedge  \Sigma_{ff'}\bigl(h^{(1)}\bigr)}
\frac
 {e''{}^{(a-1)}\wedge h^{(n-a)} \wedge  \Sigma_{ff'}\bigl(h^{(1)}\bigr)}
 {e''{}^{(a-1)}\wedge h^{(n-a)} \wedge h'{}^{(1)}}
 $$
 for arbitrary non-zero $e''{}^{(b)} \in \Lambda^b (E''{}^{(b)} )$, $h^{(b)} \in \Lambda^b (H^{(b)} )$, $h'{}^{(b)} \in \Lambda^b (H'{}^{(b)} )$. 
 
 The elementary slithering map $\Sigma_{f'f} = \Sigma_{ff'}^{-1}$ sends $H$ to $H'$. By Condition~(4) of Proposition~\ref{prop:Slithering}, it  acts trivially on each $\Lambda^b (E''{}^{(b)} )$ and on $\Lambda^n(\R^n)$. If we choose $h'{}^{(b)} = \Sigma_{f'f} (h^{(b)} )$, we consequently have that 
$$
e''{}^{(b)}\wedge h^{(n-b-1)} \wedge  \Sigma_{ff'}\bigl(h^{(1)}\bigr)
= e''{}^{(b)}\wedge h'{}^{(n-b-1)} \wedge  h^{(1)}
$$
for every $b$. 
Similarly, $ e''{}^{(b)}  \wedge h^{(n-b)} = e''{}^{(b)}  \wedge h'{}^{(n-b)} $ for every $b$

Combining these properties and rearranging terms provides
\begin{align*}
D_a\bigl (E'', H, H'  ,  \Sigma_{ff'}(H)\bigr)&= 
- \frac{e''{}^{(a)}\wedge h^{(n-a-1)} \wedge h'{}^{(1)}}
 {e''{}^{(a)}\wedge h'{}^{(n-a-1)} \wedge  h^{(1)}}
\frac
 {e''{}^{(a-1)}\wedge h'{}^{(n-a)} \wedge  h^{(1)}}
 {e''{}^{(a-1)}\wedge h^{(n-a)} \wedge h'{}^{(1)}}\\
  &= 
\frac{e''{}^{(a)}\wedge h^{(n-a-1)} \wedge h'{}^{(1)}}
 {e''{}^{(a-1)}\wedge h^{(n-a)} \wedge h'{}^{(1)}}
\frac
 {e''{}^{(a-1)} \wedge  h^{(1)} \wedge h'{}^{(n-a)} }
 {e''{}^{(a)} \wedge  h^{(1)} \wedge h'{}^{(n-a-1)} }
\\
 &\qquad\qquad\qquad
 \qquad\qquad\qquad\qquad\qquad
 \frac{e''{}^{(a+1)}  \wedge h'{}^{(n-a-1)} }
 {e''{}^{(a+1)} \wedge h^{(n-a-1)} }
  \frac{e''{}^{(a)}  \wedge h^{(n-a)} }
 {e''{}^{(a)} \wedge h'{}^{(n-a)} }
 \\
 &= Q_{a}(E'', H, H')^{-1}.
\end{align*}

This proves that 
\begin{align*}
\sigma_a^\rho (T,T') &= \sigma_a^\rho (T,T'')+  \sigma_a^\rho (T'', T') - \log Q_a (E'',H, H')\\
&= \sigma_a^\rho (T,T'')+  \sigma_a^\rho (T'', T') - \theta^\rho_a(s'')
\end{align*}
where $s''$ is the slit of $\lambda$ that is the projection of the slit $\widetilde s''$ of $\widetilde\lambda$. 

This computation holds when $\widetilde s''$ points to the left as seen from $T$. 
When $\widetilde s''$ points to the right, a very similar computation or an application of Lemma~\ref{lem:ReversingOrientationShear} shows that in this case
$$
\sigma_a^\rho (T,T') = \sigma_a^\rho (T,T'')+  \sigma_a^\rho (T'', T') - \theta^\rho_{n-a}(s'').
$$

Considering these two cases, Proposition~\ref{prop:RelTgtCycleDiffrentViewpoint} then shows that the rule $(T,T') \mapsto \sigma_a^\rho(T,T')$  defines a relative tangent cycle  $\sigma_a^\rho \in \CC(\widehat \lambda, \slits; \R)$, whose boundary $\partial \sigma_a^\rho \colon \{ \text{slits of } \widehat \lambda \} \to \R$ is the one described in the statement of Lemma~\ref{lem:ShearingCycleIsCycle}. This concludes the proof of that lemma.
\end{proof}

Through Proposition~\ref{prop:TwistedRelTgtCycleDiffrentViewpoint}, the combination of Lemmas~\ref{lem:ReversingOrientationShear} and \ref{lem:ShearingCycleIsCycle} shows that the relative tangent cycles $\sigma_a^\rho \in \CC(\widehat \lambda, \slits ; \R)$ can be combined to define a  relative tangent cycle $\sigma^\rho \in \CC(\lambda, \slits; \widehat \R^{n-1})$ valued in the twisted coefficient bundle $\widehat\R^{n-1}$ introduced in \S \ref{subsect:TwistedRelTangentCycles}. This twisted relative tangent cycle is the \emph{shearing cycle} of the Hitchin character $\rho \in \Hit(S)$ with respect to the maximal geodesic lamination $\lambda$.

 \section{Hitchin characters are determined by their invariants}
 \label{bigsect:Uniqueness}
  The goal of this section is to show that, if two Hitchin homomorphisms $\rho$, $\rho' \colon \pi_1(S) \to \PSL$ have the same triangle invariants and the same shearing cycle, then they represent the same character in  the Hitchin component $\Hit(S)$.

\subsection{Revisiting the slithering map}
\label{subsect:RevisitSlithering}

We want to give a different description of the slithering map $\Sigma_{gg'}$ of \S \ref{subsect:Slithering}. This new formulation is based on the following simple algebraic trick. 

\begin{lem}
\label{lem:ReorderProduct}
Let $A_1$, $A_2$, \dots, $A_m$ be elements of a group. Then,  
$$
A_1A_2 \dots A_{m-1}A_m = \widehat A_m \widehat A_{m-1} \dots \widehat A_2 \widehat A_1
$$
where $\widehat A_i = (A_1A_2 \dots A_{i-1}) A_i (A_1A_2\dots A_{i-1})^{-1}$. 
\end{lem}
\begin{proof} Observe that 
$A_1A_2 \dots A_{m-1}A_m = \widehat A_m A_1A_2 \dots A_{m-1}$, and proceed by induction.
\end{proof}

We return to the construction of the slithering map $\Sigma_{gg'}$ in     \S \ref{subsect:Slithering}. Let $g$ and $g'$ be two leaves of the preimage $\widetilde\lambda \subset \widetilde S$ of the geodesic lamination $\lambda$, and let $\mathcal T_{gg'}$ be the set of components of $\widetilde S - \widetilde \lambda$ that separate $g$ from $g'$, where these components are ordered from $g$ to $g'$. For such a component $T\in \mathcal T_{gg'}$, we consider the elementary slithering $\Sigma_T = \Sigma_{g_T^{\phantom i}g_T'}$ defined by Condition~(4) of Proposition~\ref{prop:Slithering}, where $g_T$ and $g_T'$ are the two sides of $T$ that are respectively closest to $g$ and $g'$.  

We now consider the infinite product of the maps
$$
\widehat \Sigma_T = \Sigma_{gg_T}\circ  \Sigma_T \circ  \Sigma_{gg_T}^{-1}.
$$
More precisely, let $\mathcal T = \{T_1, T_2, \dots, T_m\}$ be a finite subset of $\mathcal T_{gg'}$, where each $T_i$ separates $T_{i+1}$ from $g$. We then consider the limit
$$
\overleftarrow{\prod_{\kern -10pt T\in \mathcal T_{gg'}\kern -10pt}} \kern 5pt\widehat\Sigma_T
=\lim_{\mathcal T \to \mathcal T_{gg'}} \widehat\Sigma_{T_m} \circ \widehat\Sigma_{T_{m-1}} \circ \dots \circ \widehat\Sigma_{T_2} \circ \widehat\Sigma_{T_1} .
$$
The reverse arrow on top of the product sign is here to remind us that the composition of  the $\widehat\Sigma_{T}$ is taken in the order opposite to the ordering of the elements of $\mathcal T_{gg'}$ from $g$ to $g'$, 

\begin{prop}
\label{prop:RevisitSlithering}
$$
 \Sigma_{gg'} = \overleftarrow{\prod_{\kern -10pt T\in \mathcal T_{gg'}\kern -10pt}} \kern 5pt\widehat\Sigma_T .
$$
\end{prop}

\begin{proof}
First of all, the fact that the infinite product converges is proved by the estimates of \S \ref{subsect:Slithering}, using the fact that the $\Sigma_{gg_T}$ are uniformly bounded (Lemma~\ref{lem:SlitheringBounded}) and the estimates on $\Sigma_T-\Id_{\R^n}$ given by Lemmas~\ref{lem:EstimateElemSlithering} and \ref{lem:SumPowersLengthsGapsConverges}. 

As usual, let $\mathcal T = \{T_1, T_2, \dots, T_m\}$ be a finite subset of $\mathcal T_{gg'}$, where each $T_i$ separates $T_{i+1}$ from $g$. By Lemma~\ref{lem:ReorderProduct}, 
$$
\overrightarrow{\prod_{\kern -10pt T\in \mathcal T\kern -10pt}} \kern 3pt\Sigma_T
=
 \Sigma_{T_1} \circ \Sigma_{T_{2}} \circ \dots \circ \Sigma_{T_{m-1}} \circ \Sigma_{T_m} 
 =
 \widehat\Sigma_{T_m}^{\mathcal T} \circ \widehat\Sigma_{T_{m-1}}^{\mathcal T} \circ \dots \circ \widehat\Sigma_{T_2}^{\mathcal T} \circ \widehat\Sigma_{T_1}^{\mathcal T} 
 =\overleftarrow{\prod_{\kern -10pt T\in \mathcal T\kern -10pt}} \kern 3pt\widehat\Sigma_T^{\mathcal T}
$$
where 
$$
 \widehat\Sigma_{T_i}^{\mathcal T} 
 = \bigl( \Sigma_{T_1}  \circ \Sigma_{T_2}  \circ \dots \circ \Sigma_{T_{i-1}}  \bigr)
 \circ \Sigma_{T_i} 
 \circ
  \bigl( \Sigma_{T_1}  \circ \Sigma_{T_2}  \circ \dots \circ \Sigma_{T_{i-1}}  \bigr)^{-1}. 
$$

For a fixed $T$, the map $ \widehat\Sigma_{T}^{\mathcal T} $ tends to $\widehat \Sigma_T = \Sigma_{gg_T}\circ  \Sigma_T \circ  \Sigma_{gg_T}^{-1}$ as the finite family $\mathcal T$ tends to the set $\mathcal T_{gg'}$ of all components of $\widetilde S - \widetilde \lambda$ separating $g$ from $g'$,  by definition of the slithering map. By uniformity in the estimates guaranteeing the convergence of the infinite products, it follows that 
\begin{equation*}
   \Sigma_{gg'} 
   =
   \lim_{\mathcal T \to \mathcal T_{gg'}} \kern 2pt \overrightarrow{\prod_{\kern -10pt T\in \mathcal T\kern -10pt}} \kern 3pt\Sigma_T
   =  \lim_{\mathcal T \to \mathcal T_{gg'}} \kern 2pt \overleftarrow{\prod_{\kern -10pt T\in \mathcal T\kern -10pt}} \kern 3pt\widehat\Sigma_T^{\mathcal T}
   =  \lim_{\mathcal T \to \mathcal T_{gg'}} \kern 2pt \overleftarrow{\prod_{\kern -10pt T\in \mathcal T\kern -10pt}} \kern 3pt\widehat\Sigma_T
   = \overleftarrow{\prod_{\kern -10pt T\in \mathcal T_{gg'}\kern -10pt}} \kern 5pt\widehat\Sigma_T . \qedhere
\end{equation*}
\end{proof}

\subsection{Reconstructing a Hitchin homomorphism from its invariants}
\label{subsect:ParamInjective}

We now show how to reconstruct, up to conjugation by an element of $\PSL$,  a Hitchin homomorphism $\rho \colon \pi_1(S) \to \PSL$ from its triangle invariants and its shearing cycle. 

For this, we first normalize $\rho$ to avoid having to worry about conjugations.  Fix a component $T_0$ of $\widetilde S-\widetilde \lambda$, with vertices $x_0$, $y_0$, $z_0 \in  \partial_\infty \widetilde S$. Also, choose a positive flag triple $(E_0, F_0, G_0)$. 

\begin{lem}
\label{lem:NormalizeHitchinRep}
After conjugating the Hitchin homomorphism $\rho$ by an element of $\PSL$, we can arrange that the flag $\F_\rho(x_0)$ is equal to $E_0$, the flag $\F_\rho(y_0)$ is equal to $F_0$, and the line $\F_\rho(z_0)^{(1)}$ is equal to the line $G_0^{(1)}$. 
\end{lem}

\begin{proof} By elementary linear algebra, there exists a unique element $\phi \in \PGL$ sending the flag $\F_\rho(x_0)$ to $E_0$, the flag $\F_\rho(y_0)$  to $F_0$, and the line $\F_\rho(z_0)^{(1)}$  to the line $G_0^{(1)}$. Because the set of positive flag triples is connected (see for instance Proposition~\ref{prop:TripRatiosDetermineFlagTriples}), $\phi$ is in the connected component of $\PGL$ that contains the identity, namely $\phi$ is an element of $\PSL$. 

Conjugating $\rho$ by  $\phi \in \PSL$ replaces the flag curve $\F_\rho \colon \partial_\infty \widetilde S \to \Flag$ by its composition with the action of $\phi$ on $\Flag$, which completes the proof. 
\end{proof}

The following lemma will help in the exposition, by decreasing the number of cases to consider. Let $g_0$ be the side of $T_0$ joining $x_0$ and $y_0$, and let $h_0$ be the side joining $x_0$ and $z_0$. 

\begin{lem}
\label{lem:NiceGeneratorsForPi1}
The fundamental group $\pi_1(S)$ is generated by finitely many elements $\gamma \in \pi_1(S)$ whose axes cross both $g_0$ and $h_0$, and send $T_0$ to a triangle $\gamma T_0$ contained in the component of $\widetilde S - T_0$ that is adjacent to $g_0$.

\end{lem}

\begin{proof} The axes of $\pi_1(S)$ are dense in the space of geodesics of $\widetilde S$. Therefore, there exists an element $\gamma_0 \in \pi_1(S)$ whose axis crosses both $g_0$ and $h_0$, and whose attracting fixed point in $\partial_\infty \widetilde S$ is contained in the closure of the component $U$ of $\widetilde S - T_0$ delimited by $g_0$. In particular, $\gamma_0 T_0$ is contained in $U$. 

Let $\gamma_1$, $\gamma_2$, \dots, $\gamma_k$ be a set of generators for $\pi_1(S)$. The Pingpong Lemma shows that, for $m_i$, $n_i>0$ large enough, the attracting and repulsing fixed points of $\gamma_i'= \gamma_0^{m_i} \gamma_i \gamma_0^{n_i}$ are very close to the attracting and repulsing fixed points of $\gamma_0$. In particular, the axis of $\gamma_i'$ crosses both $g_0$ and $h_0$, and $\gamma_i'T_0$ is contained in $U$. 

Then the family of elements $\gamma_0$, $\gamma_1'$, $\gamma_2'$, \dots, $\gamma_k'$ generates $\pi_1(S)$ and has the required properties. 
\end{proof}

\begin{figure}[htbp]

\SetLabels
(.48 * .2) $ T_0$ \\
(-.02 * .22) $x_0 $ \\
(1.02 *.18 ) $y_0 $ \\
( .64* -.03) $z_0 $ \\
( .5 * .32) $ g_0$ \\
( .4* .13) $h_0 $ \\
( .58*.82 ) $\gamma T_0$ \\
( -.05* .75) $\gamma x_0 $ \\
( .71* 1.02) $\gamma y_0 $ \\
( 1.05* .7) $\gamma z_0 $ \\
( .4*.88 ) $ \gamma g_0$ \\
(.5* .69) $\gamma h_0 $ \\
( .15* .47) $T $ \\
( .4* .42) $ g_T$ \\
( -.03 * .35) $x_T $ \\
( 1.04* .5) $ y_T$ \\
( -.03* .58) $z_T $ \\
(.42 * .54) $ g_T'$ \\
( * ) $ $ \\
( * ) $ $ \\
( * ) $ $ \\
\endSetLabels
\centerline{\AffixLabels{\includegraphics{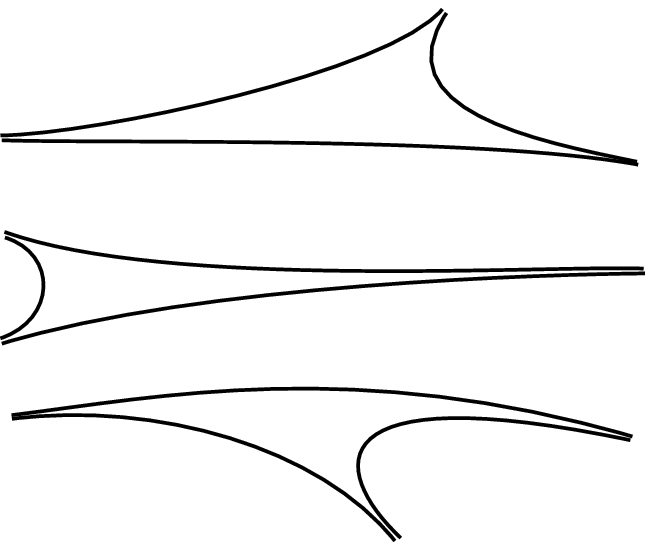}}}

\caption{}
\label{fig:Reconstruction}
\end{figure}

For $t=(t_1, t_2, \dots, t_{n-1}) \in \R^{n-1}$, let $u_1$, $u_2$, \dots, $u_n$ be uniquely determined by the properties that  $t_a =  u_a - u_{a+1}$ and $\sum_{a=1}^n u_a=0$. Namely,  $u_a=  \frac1n \sum_{b=1}^{n-1} (n-b) t_b - \sum_{b=1}^{a-1} t_b $. Then, let  $\Theta^t_{E_0F_0}\colon \R^n \to \R^n$ be the element of $\SL$  that acts by multiplication of $\mathrm e^{u_a}$ on each line $E_0^{(a)} \cap F_0^{(n-a+1)}$. 

For every generic flag triple $(E,F,G)$, elementary linear algebra provides a unique projective map $\phi \in \PGL$ that sends $E$ to $E_0$, $F$ to $F_0$, and such that 
$$
D_a\bigl(E_0, F_0, G_0, \phi(G) \bigr) =1
$$
for every $a\in \{1,2,\dots, n-1\}$. We then define
$$
\mathcal G_{(E_0, F_0, G_0)}(E,F,G) = \phi(G)\in \Flag. 
$$

In particular, we can apply this to the flag triple $\bigl( \F_\rho(x_0), \F_\rho(z_0), \F_\rho(y_0) \bigr)$ associated to the vertices of the base triangle $T_0$.  (Note the unusual vertex ordering.) This defines a projective map $\phi_0 \in \PGL$ sending the flag $E_0 =  \F_\rho(x_0)$ to itself, $\F_\rho(z_0)$ to $F_0=\F_\rho(y_0)$, and $F_0$ to $\mathcal G_{(E_0, F_0, G_0)}\bigl( \F_\rho(x_0), \F_\rho(z_0), \F_\rho(y_0) \bigr)$.

\begin{lem}
\label{lem:GroupActionAndSlithering}
Let $\rho$ be normalized as in Lemma~{\upshape \ref{lem:NormalizeHitchinRep}}, and let $\gamma\in \pi_1(S)$ be as in Lemma~{\upshape \ref{lem:NiceGeneratorsForPi1}}. Then,
$$
\rho(\gamma) = \Sigma_{g_0(\gamma h_0)}^{-1} \circ  \Theta^{\sigma^\rho(T_0, \gamma T_0)}_{E_0F_0}
\circ \phi_0 \in \PGL
$$
where $\Theta^t_{E_0F_0}$ and $\phi_0$ are defined as above, and where $\sigma^\rho(T_0, \gamma T_0)\in \R^{n-1}$ is the shear vector of $\rho$ between $T_0$ and $\gamma T_0$. 
\end{lem}

\begin{proof}
By definition of the shear parameter
\begin{align*}
\sigma^\rho_a(T_0, \gamma T_0)
&= \log D_a \bigl( \F_\rho(x_0), \F_\rho(y_0), \F_\rho(z_0), \Sigma_{g_0(\gamma h_0)} \circ \F_\rho(\gamma y_0) \bigr)\\
&= \log D_a \bigl( E_0, F_0, G_0, \Sigma_{g_0(\gamma h_0)} \circ \F_\rho(\gamma y_0) \bigr),
\end{align*}
where the second equality comes form the fact that the flag  $\F_\rho(x_0)$ is equal to $E_0$, the flag $\F_\rho(y_0)$ is equal to $F_0$, and the line $\F_\rho(z_0)^{(1)}$ is equal to the line $G_0^{(1)}$. (Recall that the double ratio $D_a(E,F,G,G')$ does not really depend on the whole flags $G$ and $G'$, only on the lines $G^{(1)}$ and $G'{}^{(1)}$.) Since
$$
D_a\Bigl(E_0, F_0, G_0, \mathcal G_{(E_0, F_0, G_0)}\bigl( \F_\rho(x_0), \F_\rho(z_0), \F_\rho(y_0) \bigr) \Bigr) =1,
$$
it follows from Lemma~\ref{lem:ComputeDoubleRatios} that $\Sigma_{g_0(\gamma h_0)} \circ \F_\rho(\gamma y_0) $ and $\mathcal G_{(E_0, F_0, G_0)}\bigl( \F_\rho(x_0), \F_\rho(z_0), \F_\rho(y_0) \bigr)=\phi_0(F_0)$ differ only by the action of $ \Theta^{\sigma^\rho(T_0, \gamma T_0)}_{E_0F_0}$. More precisely,  
\begin{align*}
\Sigma_{g_0(\gamma h_0)} \circ \F_\rho(\gamma y_0) &= 
 \Theta^{\sigma^\rho(T_0, \gamma T_0)}_{E_0F_0}
\circ
\mathcal G_{(E_0, F_0, G_0)}\bigl( \F_\rho(x_0), \F_\rho(z_0), \F_\rho(y_0) \bigr)\\
&=  \Theta^{\sigma^\rho(T_0, \gamma T_0)}_{E_0F_0}
\circ \phi_0(F_0) = 
 \Theta^{\sigma^\rho(T_0, \gamma T_0)}_{E_0F_0}
\circ \phi_0 \circ \F_\rho(y_0).
\end{align*}

The geodesic $g_0$ has endpoints $x_0$ and $y_0$, and the geodesic $\gamma h_0$ has endpoints $\gamma x_0$ and $\gamma z_0$. Therefore,
$$
\Sigma_{g_0(\gamma h_0)} \circ \F_\rho(\gamma x_0)
= \F_\rho( x_0) 
=  \Theta^{\sigma^\rho(T_0, \gamma T_0)}_{E_0F_0}
 \circ  \F_\rho( x_0) 
=  \Theta^{\sigma^\rho(T_0, \gamma T_0)}_{E_0F_0}
\circ \phi_0 \circ  \F_\rho( x_0) 
$$
since the flag $E_0= \F_\rho( x_0) $ is fixed by $ \Theta^{\sigma^\rho(T_0, \gamma T_0)}_{E_0F_0}$ and by $ \phi_0$. 

Similarly, because $ \Theta^{\sigma^\rho(T_0, \gamma T_0)}_{E_0F_0}$ fixes the flag $F_0 = \F_\rho(y_0)$ and because $\phi_0$ sends $\F_\rho(z_0)$ to $F_0$,
$$
\Sigma_{g_0(\gamma h_0)} \circ \F_\rho(\gamma z_0)
= \F_\rho( y_0) 
=  \Theta^{\sigma^\rho(T_0, \gamma T_0)}_{E_0F_0}
 \circ  \F_\rho( y_0) 
=  \Theta^{\sigma^\rho(T_0, \gamma T_0)}_{E_0F_0}
\circ \phi_0 \circ  \F_\rho( z_0) .
$$

Remembering that the flag curve is $\rho$--equivariant, so that $\F_\rho(\gamma x) = \rho(\gamma) \circ \F_\rho(x)$ for every $x\in \partial_\infty \widetilde S$, we concludes that the projective maps $\Sigma_{g_0(\gamma h_0)} \circ \rho(\gamma)$ and $ \Theta^{\sigma^\rho(T_0, \gamma T_0)}_{E_0F_0}
\circ \phi_0$ coincide on each flag of the generic flag triple $\bigl( \F_\rho(x_0), \F_\rho(z_0), \F_\rho(y_0) \bigr)$. This proves that 
$$\Sigma_{g_0(\gamma h_0)} \circ \rho(\gamma)= \Theta^{\sigma^\rho(T_0, \gamma T_0)}_{E_0F_0}
\circ \phi_0$$
as projective maps. The result then follows. 
\end{proof}

In the formula of Lemma~\ref{lem:GroupActionAndSlithering}, the term $ \Theta^{\sigma^\rho(T_0, \gamma T_0)}_{E_0F_0}$ depends only on the shearing cycle $\sigma^\rho$, while $\phi_0$ is completely determined by the triangle invariants $\tau_{abc}^\rho (s)$ of the base triangle $T_0$. We now turn our attention to the remaining term, the slithering map $\Sigma_{g_0(\gamma h_0)}$.

By Proposition~\ref{prop:RevisitSlithering},
$$
 \Sigma_{g_0(\gamma h_0)}  =\kern 10pt  \overleftarrow{\prod_{\kern -13pt T\in \mathcal T_{g_0(\gamma h_0)}\kern -13pt}} \kern 7pt\widehat\Sigma_T .
$$
with the notation of that statement. 

Consider the contribution $\widehat\Sigma_T = \Sigma_{g_0g_T} \circ \Sigma_T \circ \Sigma_{g_0g_T}^{-1}$ of a triangle $T\in \mathcal T_{g_0(\gamma h_0)}$, separating $g_0$ from $\gamma h_0$. Index the vertices of $T$ as $x_T$, $y_T$ and $z_T$, in such a way that the side $g_T = y_T x_T$ is the one that is closest to $g_0=y_0x_0$, and is oriented in parallel with $g_0$. There are two cases to consider, according to whether the side $g_T'$ of $T$ that is closest to $\gamma T_0$ is equal to $z_Tx_T$ or to $y_Tz_T$. 

Consider the case where $T$ points to the right, namely where $g_T'$ is equal to $y_Tz_T$, as in Figure~\ref{fig:Reconstruction}. Then, the elementary slithering $\Sigma_T= \Sigma_{g_T^{\phantom i}g_T'}$ is the unique linear map that fixes the flag $\F_\rho(y_T)$, acts by the identity on each line $\F_\rho(y_T)^{(a+1)}/\F_\rho(y_T)^{(a)}$, and sends the flag $\F_\rho(z_T)$ to $\F_\rho(x_T)$. It follows that 
$\widehat\Sigma_T = \Sigma_{g_0g_T} \circ \Sigma_T \circ \Sigma_{g_0g_T}^{-1}$ is the unique linear map that fixes the flag $\Sigma_{g_0g_T}\circ \F_\rho(y_T) = F_0$, acts as the identity on each line $F_0^{(a+1)}/F_0^{(a)}$, and sends the flag $\Sigma_{g_0g_T}\circ\F_\rho(z_T)$ to $\Sigma_{g_0g_T}\circ\F_\rho(x_T)=E_0$. 

We now express  $\Sigma_{g_0g_T}\circ\F_\rho(z_T)$ in terms of  the flag $G_T'= \mathcal G_{(E_0, F_0, G_0)} \bigl ( \F_\rho(x_T), \F_\rho(y_T), \F_\rho(z_T) \bigr)$, as defined above Lemma~\ref{lem:GroupActionAndSlithering}. By definition, $G_T'$  is the unique flag such that there is a projective map sending $\F_\rho(x_T)$ to $E_0$, $\F_\rho(y_T)$ to $F_0$ and $\F_\rho(z_T)$ to $G'_T$, and such that
$
D_a \bigl(
E_0, F_0, G_0, G_T'  \bigr)
=1
$
for every $a\in \{1,2,\dots, n-1\}$. Also, 
\begin{align*}
\sigma^\rho (T_0, T) &= \log 
D_a \bigl( \F_\rho(x_0), \F_\rho(y_0), \F_\rho(z_0), \Sigma_{g_0g_T} \circ \F_\rho(z_T) \bigr)\\
&= \log D_a \bigl( E_0, F_0, G_0, \Sigma_{g_0g_T} \circ \F_\rho(z_T) \bigr).
\end{align*}
As in the proof of Lemma~\ref{lem:GroupActionAndSlithering}, we conclude that 
$
\Sigma_{g_0g_T} \circ \F_\rho(z_T) = \Theta_{E_0F_0}^{\sigma^\rho(T_0, T)} (G_T')
$.

Therefore, $\widehat \Sigma_T$ is the unique projective map that sends the flag $F_0$ to itself, acts as the identity on each line $F_0^{(a+1)}/F_0^{(a)}$, and sends $ \Theta_{E_0F_0}^{\sigma^\rho(T_0, T)} (G_T')$ to $E_0$. Because $ \Theta_{E_0F_0}^{\sigma^\rho(T_0, T)}$ fixes the flags $E_0$ and $F_0$, we conclude that
$$
\widehat \Sigma_T =  \Theta_{E_0F_0}^{\sigma^\rho(T_0, T)} \circ \widehat\Sigma_T' \circ  \Theta_{E_0F_0}^{-\sigma^\rho(T_0, T)} 
$$
where $\widehat\Sigma_T'$ is the projective map that fixes $F_0$, acts as the identity on each  $F_0^{(a+1)}/F_0^{(a)}$, and sends $G_T'= \mathcal G_{(E_0, F_0, G_0)} \bigl ( \F_\rho(x_T), \F_\rho(y_T), \F_\rho(z_T) \bigr)$ to $E_0$. A key observation here is that  $\widehat\Sigma_T'$ depends only on the orbit of the flag triple  $\bigl ( \F_\rho(x_T), \F_\rho(y_T), \F_\rho(z_T) \bigr)$ under the action of $\PGL$. In particular, $\widehat\Sigma_T'$ is completely determined by the triangle invariants of the Hitchin homomorphism $\rho$ (and by our normalization conventions). 

A similar property holds them $T$ points to the left, namely when $g_t'$ is equal to the geodesic $z_Tx_T$. More precisely,
$$
\widehat \Sigma_T =  \Theta_{E_0F_0}^{\sigma^\rho(T_0, T)} \circ \widehat\Sigma_T' \circ  \Theta_{E_0F_0}^{-\sigma^\rho(T_0, T)} 
$$
where $\Sigma_T'$ fixes $E_0$, acts as the identity on each  $E_0^{(a+1)}/E_0^{(a)}$, and sends $ \mathcal G_{(E_0, F_0, G_0)} \bigl ( \F_\rho(x_T), \F_\rho(y_T), \F_\rho(z_T) \bigr)$ to $F_0$. In particular, $\widehat\Sigma_T'$ is completely determined by the triangle invariants of $\rho$ in this case as well. 

Combining these observations with Lemma~\ref{lem:GroupActionAndSlithering} gives:

\begin{lem}
\label{lem:GroupActionAndInvariants}
Let $\rho$ be normalized as in Lemma~{\upshape \ref{lem:NormalizeHitchinRep}}, and let $\gamma\in \pi_1(S)$ be as in Lemma~{\upshape \ref{lem:NiceGeneratorsForPi1}}. Then,
$$
\rho(\gamma) = \biggl(\kern 15pt
\overleftarrow{\prod_{\kern -13pt T\in \mathcal T_{g_0(\gamma h_0)}\kern -13pt}} \kern 7pt
\Bigl(
\Theta_{E_0F_0}^{\sigma^\rho(T_0, T)} \circ \widehat\Sigma_T' \circ  \Theta_{E_0F_0}^{-\sigma^\rho(T_0, T)}
\Bigr)
\biggr)^{-1}
\circ  \Theta^{\sigma^\rho(T_0, \gamma T_0)}_{E_0F_0}
\circ \phi_0
$$
in $\PGL$, with the definitions introduced above. In particular, the maps $\widehat\Sigma_T'$ and $\phi_0$ depend only on the triangle invariants $\tau_{abc}^\rho(s)$ of $\rho$, while the terms $\Theta_{E_0F_0}^{\pm \sigma^\rho(T_0, T)} $ are determined by its shearing cycle $\sigma^\rho \in \mathcal C(\lambda, \slits; \widehat\R^{n-1})$. \qed
\end{lem}

\begin{cor}
\label{cor:HitchinDeterminedByInvariants}
If two Hitchin homomorphisms $\rho$, $\rho'\colon \pi_1(S) \to \PSL$  that have the same triangle invariants $\sigma_{abc}^{\rho}(s)= \sigma_{abc}^{\rho'}(s) $ and the same shearing cycles $\sigma^\rho = \sigma^{\rho'} \in \mathcal C(\lambda, \slits; \widehat \R^n)$ are conjugate by an element of  $\PSL$, and therefore represent the same character in  $\Hit(S)$. 
\end{cor}

\begin{proof}
Conjugate  $\rho$ and $\rho'$ by elements of $\PSL$ to normalize them as in Lemma~\ref{lem:NormalizeHitchinRep}. Then, for every element $\gamma \in \pi_1(S)$ satisfying the conditions of Lemma~\ref{lem:NiceGeneratorsForPi1}, the formula of Lemma~\ref{lem:GroupActionAndInvariants} shows that $\rho(\gamma) = \rho'(\gamma)$. Since these $\gamma$ generate $\pi_1(S)$, this proves that $\rho=\rho'$. 
\end{proof}

\section{Length functions}
\label{bigsect:LengthFunctions}

Our next goal is to determine which triangle invariants and shearing cycles can be realized as invariants of Hitchin characters. The length functions considered in this section provide one of the constraints that need to be satisfied by these invariants. 

\subsection{Length functions associated to Hitchin characters}
\label{subsect:LengthFunctions}
Let $\rho\colon \pi_1(S) \to \PSL$ be a Hitchin homomorphism. Labourie proves in \cite{Lab1} that for every non-trivial $\gamma \in \pi_1(S)$, the matrix $\rho(\gamma)\in \PSL$ is diagonalizable and its eigenvalues can be indexed as $\mu_1\bigl(\rho(\gamma) \bigr)$,  $\mu_2 \bigl(\rho(\gamma) \bigr)$, \dots,  $\mu_n \bigl(\rho(\gamma) \bigr)$ in such a way that 
$$
\frac{ \mu_a\bigl(\rho(\gamma) \bigr) }{ \mu_{a+1}\bigl(\rho(\gamma) \bigr)} >1
$$
for every $i=1$, $2$, \dots, $n-1$. (Note that eigenvalues of an element of $\PSL$ are only defined up to sign, but that the quotient between two such eigenvalues makes intrinsic sense.) This property is in fact an easy consequence of Theorem~\ref{thm:AnosovProperty}. 

Eigenvalues are independent under conjugation. This consequently defines $n-1$ functions 
$$
\ell_a^\rho \colon \{\text{non-trivial conjugacy classes of } \pi_1(S) \} \to \R
$$
by the property that $ \ell_a^\rho(\gamma) = \log \frac{ \mu_a\bigl(\rho(\gamma) \bigr) }{ \mu_{a+1}\bigl(\rho(\gamma) \bigr)} > 0 $. The same conjugation invariance shows that the length function $\ell_a^\rho$ depends only on the Hitchin character $\rho \in \Hit(S)$, not on the Hitchin homomorphism $\rho\colon \pi_1(S) \to \PSL$ that represents it. 

The set of conjugacy classes of the fundamental group $\pi_1(S)$ is discrete, but these length functions have a natural extension to a continuous space. Indeed, endowing the surface $S$ with an arbitrary negatively curved riemannian metric, a conjugacy class of $\pi_1(S)$ uniquely determines an oriented  closed geodesic of $S$, and therefore a closed orbit of the geodesic flow of the unit tangent bundle $T^1S$. This closed leaf is endowed with an integer multiplicity $m>0$ if the conjugacy class is not primitive and is an $m$--power of a primitive class. Considering the Dirac transverse measure defined by this closed orbit and this multiplicity, this provides an analytic interpretation of a conjugacy class of $\pi_1(S)$ as a transverse measure for the \emph{geodesic foliation} $\mathcal F_S$ of $T^1S$, whose leaves are the orbits of the geodesic flow. 

This defines a completion of the set of conjugacy classes of $\pi_1(S)$ by the space $\CC(S)$ of all (positive Radon) transverse measures for the geodesic foliation $\mathcal F_S$ \cite{Bon86, Bon88, Bon91}, analogous to Thurston's completion \cite{Thu0, FLP, PenH} of the set of isotopy classes of simple closed curves in $S$ by the space $\ML(S)$ of measured laminations on $S$. 

For differentiability properties, it is useful to consider more general transverse structures for the geodesic foliation, namely \emph{transverse H\"older distributions} in the sense of \cite{Bon97a, Bon97b}. This embeds the set of conjugacy classes of $\pi_1(S)$ in the topological vector space $\CH(S)$ of all transverse H\"older distributions for the geodesic foliation $\mathcal F_S$. In other words, we now have embeddings
$$
\{\text{non-trivial conjugacy classes of } \pi_1(S) \} \subset \CC(S) \subset \CH(S).
$$

The elements of $\CC(S)$ and $\CH(S)$ are respectively called \emph{measure geodesic currents} and \emph{H\"older geodesic currents} for the surface $S$. See the references mentioned above for a proof that these constructions depend only on the topology of the surface $S$, and in particular are independent of the choice of a negatively curved riemannian metric on $S$.  

\begin{thm}
[{\cite{Dre1}}]
\label{thm:LengthFunctions}
For each Hitchin character $\rho \in \Hit(S)$ and for each $a=1$, $2$, \dots, $n-1$, the length function
$$
\ell_a^\rho \colon \{\text{non-trivial conjugacy classes of } \pi_1(S) \} \to \R
$$
extends to a continuous linear map
$
\ell_a^\rho \colon \CH(S) \to \R
$.\qed
\end{thm}

\begin{rem}
\label{rem:Dreyer1DifferentConventions}
The reader should beware that the above functions $\ell_a^\rho$ are slightly different from those introduced in \cite{Dre1}. Namely, our functions $\ell_a^\rho$ would be called $\ell_a^\rho - \ell_{a+1}^\rho$  in \cite{Dre1}. Although mathematically equivalent to those of \cite{Dre1}, our conventions tend to be better adapted to the framework of the current article, as can for instance be apparent in   Proposition~\ref{prop:MeasureHasPositiveLengths} and Theorem~\ref{thm:ShearingAndLength} below. 
\end{rem}

\begin{rem}
\label{rem:LengthHolderGeodCurrentsUnique?}
By linearity and continuity, the extension $\ell_a^\rho \colon \CH(S) \to \R$ is uniquely determined on the closure of the set of all linear combinations of conjugacy classes of $\pi_1(S)$. We do not know if this closure is equal to all of $\CH(S)$ (this seems unlikely), but it does contain all the H\"older geodesic currents that will occur in this article. 
\end{rem}

The following statement will be particularly important in our characterization of which relative tangent cycles can occur as shearing cycles of Hitchin characters.

\begin{prop}
\label{prop:MeasureHasPositiveLengths}
Let $\alpha\in \CC(S)$ be a non-zero  measure geodesic current. Then, 
$$
\ell_a^\rho (\alpha)>0
$$
for every Hitchin character $\rho \in \Hit(S)$ and every $a=1$, $2$, \dots, $n-1$. 
\end{prop}

\begin{proof}
This is a simple consequence of the Anosov property of Theorem~\ref{thm:AnosovProperty}.

For this, we need to remind the reader of the construction of the length functions $
\ell_a^\rho \colon \CH(S) \to \R
$ in \cite{Dre1}, taking Remark~\ref{rem:Dreyer1DifferentConventions} into account.  As in \S \ref{subsect:FlagCurve}, consider the geodesic flow $(g_t)_{t\in \R}$ on the unit tangent bundle $T^1S$ (for an arbitrary metric of negative curvature) and its flat lift to a flow $(G_t)_{t\in \R}$ on the vector bundle $T^1S \times _{\rho'} \R^n$, twisted by a homomorphism $\rho' \colon \pi_1(S) \to \SL$ lifting $\rho$. In addition, choose a riemannian metric $\Vert \ \Vert$ on the vector bundle $T^1S \times _{\rho'} \R^n \to T^1S$. 

The vector bundle $T^1S \times _{\rho'} \R^n \to T^1S$ splits as a direct sum of line bundles $L_a \to T^1S$ as in \S \ref{subsect:FlagCurve}. For $a=1$, $2$, \dots, $n$, this data provides a function $f_a \colon T^1S \to \R$ defined by the property that for $x\in T^1S$
$$
f_a(x) = - \left( {\textstyle \frac{d}{dt}} \log \left\Vert G_t\bigl(v_a(x)\bigr)\right \Vert_{g_t(x)}   \right)_{t=0}
$$
where $v_a(x)$ is an arbitrary non-zero vector in the fiber $L_a(x)$ of the line bundle $L_a \to T^1S$. For a measure geodesic current $\alpha \in \CC(S)$, the length $\ell_a^\rho(\alpha)$ is then defined as the integral
$$
\ell_a^\rho(\alpha)
= \int_{T^1S} (f_a - f_{a+1}) \, \alpha \kern -2pt \times \kern -3pt dt
$$
of the function $(f_a - f_{a+1})$ with respect to the measure $\alpha \kern -2pt \times \kern -3pt dt$ on $T^1S$ that, locally, is the product of the transverse measure $\alpha$ for the geodesic flow $(g_t)_{t\in \R}$ with the measure $dt$ along the  orbits of this geodesic flow.  (Remember that what is called $\ell_a^\rho(\alpha)$ in this article was called $\ell_a^\rho(\alpha)-\ell_{a+1}^\rho(\alpha)$ in \cite{Dre1}). 

The measure $\alpha \kern -2pt \times \kern -2pt dt$ is invariant under the geodesic flow. Therefore, for every $t_0>0$, 
\begin{align*}
\int_{T^1S} f_a  \, \alpha \kern -2pt\times \kern -2pt dt &= \int_{T^1S} f_a\circ g_u  \, \alpha  \kern -2pt\times \kern -2pt dt
 = \frac 1{t_0} \int_{T^1S}  \int_0^{t_0} f_a\circ g_u  \, du \  \alpha\kern -2pt \times \kern -2pt dt \\
 &= \frac 1{t_0} \int_{T^1S}  \int_0^{t_0} -{\textstyle \frac{d}{du}} \log \left\Vert G_u \bigl(v_a (x)\bigr)\right \Vert_{g_u(x)}   \, du \  \alpha\kern -2pt \times \kern -2pt dt(x) \\
 &= \frac 1{t_0} \int_{T^1S}  
 \log \frac{\left\Vert v_a(x) \right\Vert_{x}}  
 {\left\Vert G_{t_0} \bigl(v_a(x) \bigr)\right \Vert_{g_{t_0}\kern -1pt(x)} \kern -18pt} 
 \kern 10pt  \alpha\kern -2pt \times \kern -2pt dt(x)
\end{align*}
so that 
$$
\ell_a^\rho(\alpha) =  \frac 1{t_0} \int_{T^1S}  
 \log \frac{\left\Vert v_a(x) \right\Vert_{x}}  
 {\left\Vert G_{t_0} \bigl(v_a(x) \bigr)\right \Vert_{g_{t_0}\kern -1pt(x)} \kern -18pt} 
 \kern 10pt  
 \frac
  {\left\Vert G_{t_0} \bigl(v_{a+1}(x) \bigr)\right \Vert_{g_{t_0}\kern -1pt(x)} \kern -18pt} 
  {\left\Vert v_{a+1}(x) \right\Vert_{x}}  
  \kern 20 pt
 \alpha\kern -2pt \times \kern -3pt dt(x).
$$
Theorem~\ref{thm:AnosovProperty} provides constants $A$, $B>0$ such that 
$$
 \log \frac{\left\Vert v_a(x) \right\Vert_{x}}  
 {\left\Vert G_{t_0} \bigl(v_a(x) \bigr)\right \Vert_{g_{t_0}\kern -1pt(x)} \kern -18pt} 
 \kern 10pt  
 \frac
  {\left\Vert G_{t_0} \bigl(v_{a+1}(x) \bigr)\right \Vert_{g_{t_0}\kern -1pt(x)} \kern -18pt} 
  {\left\Vert v_{a+1}(x) \right\Vert_{x}}  
  \kern 20 pt
  \geq \log A + Bt_0
$$ 
for every $t_0>0$. In particular, this integrant is strictly positive for $t_0$ large enough, and it follows that the integral $\ell_a^\rho(\alpha)$ is strictly positive. 
\end{proof}

\subsection{Shearing  cycles and length functions}
\label{subsect:ShearingAndLength}

We now consider a special type of H\"older geodesic current. 

We saw in \S \ref{subsect:TangentCycles} that a positive tangent cycle $\mu \in \CC(\widehat \lambda; \R)$ determines a transverse measure for $\widehat \lambda$. A general tangent cycle $\alpha \in \CC(\widehat \lambda; \R)$  determines a transverse H\"older distribution, which lifts to a H\"older geodesic current  $\alpha \in \CH(S)$ \cite{Bon97a, Bon97b}. This provides an embedding $\CC(\widehat \lambda ; \R) \subset \CH(S)$, and the length functions  $\ell_a^\rho \colon \CH(S) \to \R$ of the previous section restrict  to  linear functions $\ell_a^\rho \colon \CC(\widehat\lambda; \R) \to \R$. 

\begin{thm}
\label{thm:ShearingAndLength}
Let $\rho \in \Hit(S)$ be a Hitchin character with shearing cyle $\sigma^\rho \in \CC(\lambda,\slits; \widehat\R^{n-1}) \linebreak \subset \CC(\widehat\lambda, \slits; \R^{n-1})$. Then, for every $a=1$, $2$, \dots, $n-1$, the $a$--th component $\sigma_a^\rho \in \CC(\widehat \lambda, \slits; \R) \cong H_1(\widehat U, \delv \widehat U; \R)$ of $\sigma^\rho$ is related to the length function $\ell_a^\rho \colon \CC(\widehat\lambda; \R) \to \R$ by the property that
$$
\ell_a^\rho(\alpha) = [\alpha] \cdot [\sigma_a^\rho]
$$
for every tangent cycle $\alpha \in \CC(\widehat\lambda; \R)  \cong H_1(\widehat U; \R)$, where $\cdot $ denotes the algebraic intersection number of relative homology classes in the train track neighborhood $\widehat U$ of the orientation cover $\widehat \lambda$. 
\end{thm}

\begin{proof} We will split the proof into several lemmas. 

We first give a different computation of the shearing cycle that uses  the functions $f_a$ that we encountered in the proof of Proposition~\ref{prop:MeasureHasPositiveLengths}. Actually, we will consider the differential 1--form $\omega_a= f_a\, dt$ defined along the orbits of the geodesic flow. By restriction, this form projects to a differential form along the leaves of the orientation cover $\widehat\lambda$, that we will also denote by $\omega_a$. 

We now extend this $\omega_a$ to a closed 1--form on the neighborhood $\widehat U$ of $\widehat\lambda$, in a weaker sense  because of the low regularity of the line bundle $L_a$. Remember that a differential form is closed if and only if it is locally exact. This leads us to define a \emph{topological closed $1$--form} on $\widehat U$ as the data, at each point of $ \widehat U$,  of a germ of continuous function well-defined up to an additive constant; in addition we require these function germs to be locally compatible in the sense that, when $y$ is sufficiently close to $x \in \widehat U$, the germ associated to $y$ is the restriction of the germ associated to $x$. Such a topological closed 1--form is \emph{H\"older continuous} if it is defined by a family of germs of H\"older continuous functions.

In our case, the 1--form $\omega_a$ was locally defined on each leaf $g$ of $\lambda$ as $\omega_a = dF_a$ for an explicit smooth function $F_a(t) = -\log \left\Vert G_t(\bigl(v_a(x) \bigr) \right\Vert_{g_t(x)}$  defined on that leaf and, locally, uniquely determined up to an additive constant. The construction of this function $F_a$ involves the line bundle $L_a$ and the choice of a riemannian metric on the bundle $T^1S \times_{\rho'}\R^n$. In particular, because the line bundle $L_a$ is H\"older continuous by Proposition~\ref{prop:FlagCurve}, this function $F_a$ can be chosen to be locally H\"older continuous on $\widehat\lambda$. Since a H\"older continuous function defined on a closed subset of a metric space always extends to a Holder continuous function over the larger space, this enables us to extend $\omega_a$ to a H\"older continuous topological closed $1$--form $\omega_a$ on $\widehat U$.

The definition of topological closed 1--form is specially designed so that the integral $\int_k \omega_a$ makes sense for every continuous arc $k$ in $\widehat U$. In particular, $\omega_a$ determines a cohomology class $[\omega_a] \in H^1(\widehat U; \R)$. 

\begin{lem}
\label{lem:LengthIsEvaluationOmega}
For every tangent cycle $\alpha \in \CC(\widehat\lambda;\R)$, the length $\ell_a^\rho(\alpha)$ is equal to the evaluation
$$
\ell_a^\rho(\alpha) = \bigl \langle  [\omega_a] - [\omega_{a+1}], [\alpha]   \bigr\rangle
$$
of the cohomology class $[\omega_a] - [\omega_{a+1}] \in H^1(\widehat U; \R)$ over the homology class $[\alpha] \in  H_1(\widehat U; \R)$ determined by $\alpha$ as in Proposition~{\upshape\ref{prop:HomologyTangentCycles}}. 
\end{lem}
\begin{proof} The tangent cycle $\alpha\in \CC(\widehat\lambda; \R)$  defines a transverse H\"older distribution for the geodesic lamination $\widehat\lambda$; see \cite{Bon97a, Bon97b}. As in \cite{RueSul}, we can then interpret the data of the geodesic lamination $\widehat\lambda$ endowed with this  transverse H\"older distribution  as a closed de Rham current in $\widehat U$. The homology class of $H_1(\widehat U; \R)$ defined by this de Rham current is exactly the class $[\alpha]$ introduced in Proposition~\ref{prop:HomologyTangentCycles}. 

By definition, the length $\ell_a^\rho(\alpha) $ is obtained by locally integrating the differential form $\omega_a - \omega_{a+1}$ over the leaves of $\widehat\lambda$, and then integrating the corresponding function of the leaves of $\widehat\lambda$ with respect to the transverse H\"older distribution defined by $\alpha$. See \cite{Dre1} for precise details, using a suitable partition of unity for $T^1S$. This construction is identical to the expression of \cite{RueSul} for the evaluation of $[\omega_a - \omega_{a+1}] \in H^1(\widehat U; \R)$ over the homology class $[\alpha] \in H_1(\widehat U; \R)$ represented by the de Rham current $\alpha \in \CH(\widehat\lambda)$. 
\end{proof}

To relate the shearing cycles $\sigma_a^\rho$ to the forms $\omega_a$, consider an arc $k$  in $\widehat U$ that is tightly transverse to $\widehat\lambda$. As usual, orient $k$ to the right of the leaves of $\widehat\lambda$, and lift $k$ to an oriented arc $\widetilde k$ in the universal cover $\widetilde S$. Consistently with the canonical orientation of the leaves of $\widehat\lambda$, we orient the leaves of $\widetilde\lambda$ that meet $\widetilde k$ to the left of $\widetilde k$. 

We first consider a component $d$ of $\widetilde k-\widetilde\lambda$ that does not contain any of the two endpoints of $k$. In particular, the positive and negative endpoints $x_d^+$ and $x_d^-$ of $d$ belong to $\widetilde \lambda$. 

The tangent  of the oriented leaf of $\widetilde\lambda$ passing through $x_d^\pm$ determines an element $u_d^\pm\in T^1\widetilde S$ of the unit tangent bundle of $\widetilde S$. If $g_d^\pm$ denotes the leaf of $\widetilde\lambda$ passing through $x_d^\pm$ and if we use the same letter to denote the projection $d\subset k\subset \widehat U$ of the arc $d\subset \widetilde k\subset \widetilde S$, we now connect the integral $\int_d \omega_a$ to  the elementary slithering map $\Sigma_{g_d^+ g_d^-} \colon \R^n \to \R^n$. The riemannian metric on the vector bundle $T^1 S \widetilde\times_{\rho'} \R^n$ used in the definition of the forms $\omega_a=f_a\,dt$ along $\widehat\lambda$ defines, for each $u\in T^1 \widetilde S$, a norm $\Vert \ \Vert_u$ on $\R^n$. 

\begin{figure}[htbp]

\SetLabels
( .37* .34) $ \widetilde k$ \\
\tiny
( .395* .485) $d $ \\
( .455*.535 ) $ x_d^+$ \\
( .44*.44 ) $x_d^- $ \\
( .25* .57) $u_d^+ $ \\
(.25 * .42) $ u_d^-$ \\
( .76*.52 ) $ y_d^+$ \\
( .76*.43 ) $y_d^- $ \\
( .6*.52 ) $w_d^+ $ \\
( .6* .43) $ w_d^-$ \\
( .51*.76 ) $ d^+$ \\
(.5 * .68) $x_{d^+}^- $ \\
(.3 * .68) $ u_{d^+}^-$ \\
( .43* .18) $d^- $ \\
( .44*.28 ) $ x_{d^-}^+$ \\
(.25 * .29) $u_{d^-}^+ $ \\
( -.0*.77 ) $ x$ \\
( 1* .81) $ y$ \\
( .63*1 ) $ z$ \\
( * ) $ $ \\
( * ) $ $ \\
( * ) $ $ \\
\endSetLabels
\centerline{\AffixLabels{\includegraphics{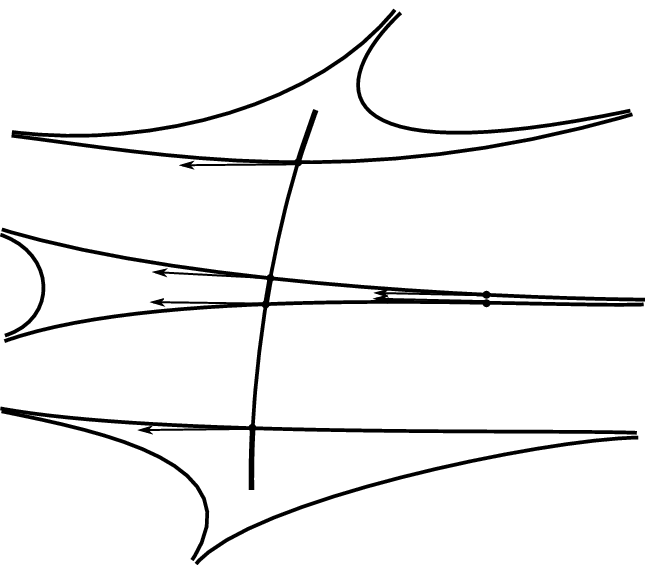}}}

\caption{}
\label{fig:LengthIntersection}
\end{figure}

\begin{lem}
\label{lem:IntegralOmegaOverGap}
Let $k$ be an arc in $\widehat U$ that is tightly transverse to $\widehat\lambda$, and let $d$ be a component of $k-\widehat\lambda$ that contains none of the two endpoints of $k$. Then,
$$
\int_d \omega_a =  \log\frac{\Vert v_a (u_d^-) \Vert_{u_{d}^-}}{\bigl\Vert \Sigma_{g_d^+ g_d^-}\bigl(v_a (u_{d}^-) \bigr)\bigr\Vert_{u_{d}^+}}
$$
for any non-zero vector $v_a (u_d^-) $ in the line $ \widetilde L_a (u_d^-)$. 
\end{lem}

\begin{proof}
The two leaves $g_d^+$ and $g_d^-$ are asymptotic. We can therefore find points $y_d^+\in g_d^+$ and $y_d^- \in g_d^-$ which are arbitrarily close to each other.   Let $w_d^\pm $ be the element of the  unit tangent bundle  $T^1\widetilde S$ determined by the tangent of the oriented geodesic $g_d^\pm$ at the point $y_d^\pm$. See Figure~\ref{fig:LengthIntersection}.  We can then deform $d$ to an arc consisting of the arc from $x_d^-$ to $y_d^-$ in the leaf $g_d^-$, followed by a short arc from $y_d^-$ to $y_d^+$, and completed by the arc from $y_d^+$ to $x_d^+$ in $g_d^+$. Then, by definition of the form $\omega_a$,
\begin{align*}
\int_d \omega_a &= \int_{x_d^-}^{y_d^-} \omega_a +  \int_{y_d^-}^{y_d^+} \omega_a +  \int_{y_d^+}^{x_d^+} \omega_a \\
&= \log \frac{\Vert v_a (u_d^-) \Vert_{u_{d}^-}} {\Vert v_a(u_d^-) \Vert_{w_{d}^-}}
+  \int_{y_d^-}^{y_d^+} \omega_a +
\log \frac{\Vert v_a (u_d^+) \Vert_{w_{d}^+}} {\Vert v_a (u_d^+) \Vert_{u_{d}^+}}
\end{align*}
for arbitrary non-zero vectors $v_a (u_d^+) \in \widetilde L_a (u_d^+)$ and $v_a (u_d^-) \in \widetilde L_a (u_d^-)$. In particular,  we can choose $v_a (u_d^+) = \Sigma_{g_d^+ g_d^-}\bigl(v_a (u_{d}^-) \bigr)$, in which case 
$$
\int_d \omega_a
= \log\frac{\Vert v_a (u_d^-) \Vert_{u_{d}^-}}{\bigl\Vert \Sigma_{g_d^+ g_d^-}\bigl(v_a(u_{d}^-) \bigr)\bigr\Vert_{u_{d}^+}} 
- \log\frac{\Vert v_a (u_d^-) \Vert_{w_{d}^-}}{\bigl\Vert \Sigma_{g_d^+ g_d^-}\bigl(v_a(u_{d}^-) \bigr)\bigr\Vert_{w_{d}^+}}  
+  \int_{y_d^-}^{y_d^+} \omega_a .
$$
Now, we  let the points $y_d^+$ and $y_d^-$ tend to the common endpoint of $g_d^+$ and $g_d^-$ in such a way that the distance from  $y_d^+$ to $y_d^-$ tends to 0. Looking at the projections to $S$, the integral $\int_{y_d^-}^{y_d^+} \omega_a $ tends to 0, while the quotient $\frac{\Vert v_a (u_d^-) \Vert_{w_{d}^-}}{\bigl\Vert \Sigma_{g_d^+ g_d^-}\bigl(v_a(u_{d}^-) \bigr)\bigr\Vert_{w_{d}^+}} $ tends to 1 (compare Lemma~\ref{lem:EstimateElemSlithering} and use the $\rho$--equivariance of the riemannian metric $\Vert\ \Vert$). It follows that 
$$
\int_d \omega_a =  \log\frac{\Vert v_a(u_d^-) \Vert_{u_{d}^-}}{\bigl\Vert \Sigma_{g_d^+ g_d^-}\bigl(v_a(u_{d}^-) \bigr)\bigr\Vert_{u_{d}^+}}
$$
for any non-zero vector $v_a(u_d^-) \in \widetilde L_a (u_d^-)$. 
\end{proof}

We will now choose preferred vectors $v_a(u_d^-) \in \widetilde L_a (u_d^-)$.

Let  $d^+$ and $d^-$ be the components of $\widetilde k-\widetilde\lambda$ that contain the positive and negative endpoints of $\widetilde k$, respectively. In particular, their endpoints $x_{d^+}^-$ and $x_{d^-}^+$ are the points of $\widetilde k \cap \widetilde \lambda$  that  are closest to the positive and negative endpoints in $\widetilde k$, respectively.  As usual, let  $u_{d^\pm}^\mp\in T^1 \widetilde S$ be defined by the vector tangent to the (oriented) leaf $g_{d^\pm}^\mp$ of $\widetilde\lambda$ passing through $x_{d^\pm}^\mp$. See Figure~\ref{fig:LengthIntersection}.
 
 The flag map $\mathcal F_\rho \colon \partial_\infty \widetilde S \to \Flag$  associates several lines of $\R^n$ to the vector $u_{d^+}^-\in T^1 \widetilde S$. This includes the  $n$ lines  $\widetilde L_a(u_{d^+}^-) =  \mathcal F_\rho(x)^{(a)} \cap  \mathcal F_\rho(y)^{(n-a+1)} $ of \S \ref{subsect:FlagCurve}, defined by the flags  $ \mathcal F_\rho (x)$ and  $ \mathcal F_\rho (y)$ respectively associated to the positive endpoint $x$ and the negative endpoint $y$ of the leaf $g_{d^+}^-$. We can also consider the line  $\mathcal F_\rho (z)^{(1)}$ of the flag $ \mathcal F_\rho (z)$ associated to the third vertex $z$ of the triangle component of $\widetilde S - \widetilde \lambda$ that contains  $d^+$. Pick a non-trivial vector $v(u_{d^+}^-) $ in this line $\mathcal F_\rho (z)^{(1)}$, and let $v_a(u_{d^+}^-)\in \widetilde L_a(u_{d^+}^-)$ be the projection of $v(u_{d^+}^-)$ parallel to the $\widetilde L_b(u_{d^+}^-)$ with $b\neq a$. 

In particular, considering the riemannian metric $\Vert\ \Vert$,  the quantity $\frac{\Vert v_a (u_{d^+}^-) \Vert_{u_{d^+}^-}}{\Vert v_{a+1} (u_{d^+}^-) \Vert_{u_{d^+}^-}}$ is independent of the choice of the vector $ v(u_{d^+}^-)  \in \mathcal F_\rho (z)^{(1)}$. Note that this ratio is finite and positive by genericity of the flag triple $\bigl( \mathcal F_\rho(x), \mathcal F_\rho(y), \mathcal F_\rho(z) \bigr)$. 

We can introduce similar definitions at the point $x_{d^-}^+$ of $\widetilde k \cap \widetilde \lambda$  that is closest to the negative endpoint of $\widetilde k$. Considering the triangle component of $\widetilde S- \widetilde\lambda$ that contains the negative endpoint of $\widetilde k$,  this leads to a well-defined positive ratio $\frac{\Vert v_a (u_{d^-}^+) \Vert_{u_{d^-}^+}}{\Vert v_{a+1} (u_{d^-}^+) \Vert_{u_{d^-}^+}}$.

\begin{lem}
\label{lem:IntegrateOmegaTransverseArc}
Let $k$ be an arc in $\widehat U$ that is tightly transverse to $\widehat\lambda$. Then, for the above definitions, 
$$
\sigma_a^\rho(k) =  \int_{k-d^+ \cup d^-}( \omega_a 
-   \omega_{a+1} )
+ \log\frac{\Vert v_a (u_{d^+}^-) \Vert_{u_{d^+}^-}}{\Vert v_{a+1} (u_{d^+}^-) \Vert_{u_{d^+}^-}}
- \log \frac{\Vert v_a (u_{d^-}^+) \Vert_{u_{d^-}^+}}{\Vert v_{a+1} (u_{d^-}^+) \Vert_{u_{d^-}^+}}.
$$
\end{lem}
Note that the notation is  ambiguous in the special case where $u_{d^+}^-=u_{d^-}^+$, which occurs when the arc $k$ crosses $\widehat\lambda$ in only one point. We will leave to the reader the easy task of lifting the ambiguity in this case. 

\begin{proof}
By a well-known result of Birman-Series \cite{BirSer}, the intersection $\widetilde k \cap \widetilde \lambda$ has Hausdorff dimension 0. Since the topological closed 1--form $\omega$ is H\"older continuous, if follows that
$$
\int_{k-d^+ \cup d^-} \omega_a = \sum_d \int_{d} \omega_a 
$$
where the sum is over all components $d$ of $\widetilde k-\widetilde\lambda$ that are different from $d^+$ and $d^-$. (The critical property is  that the image of a set of Hausdorff dimension 0 under a H\"older continuous function has Hausdorff dimension 0, and in particular has Lebesgue measure 0 in $\R$.)

We now apply Lemma~\ref{lem:IntegralOmegaOverGap} while choosing $v_a (u_d^-) = \Sigma_{g_{d}^-g_{d^-}^+}\bigl(v_a(u_{d^-}^+) \bigr) \in L_a(u_d^-)$, where $v_a(u_{d^-}^+) \in \widetilde L_a(u_{d^-}^+)$ is determined as above by the vertices of the triangle component of $\widetilde S - \widetilde \lambda$ that contains $d^+$. Then, 
$$
 \int_{k-d^+ \cup d^-} \omega_a 
= \sum_d \log\frac
{\bigl\Vert \Sigma_{g_{d}^-g_{d^-}^+}\bigl(v_a(u_{d^-}^+) \bigr) \bigr\Vert_{u_{d}^-}}
{\bigl\Vert\Sigma_{g_{d}^+g_{d^-}^+}\bigl(v_a(u_{d^-}^+) \bigr) \bigr\Vert_{u_{d}^+}}
$$
by observing that
$$\Sigma_{g_d^+ g_d^-}\bigl(v_a(u_{d}^-) \bigr) = \Sigma_{g_d^+ g_d^-} \circ \Sigma_{g_{d}^-g_{d^-}^+}\bigl(v_a(u_{d^-}^+) \bigr) = \Sigma_{g_d^+ g_{d^-}^+}\bigl(v_a(u_{d^-}^+) \bigr).
$$

If $g_x$ denotes the oriented leaf of $\widetilde\lambda$ passing through $x\in \widetilde k \cap \widetilde \lambda$ and if $u_x\in T^1 \widetilde S$ is the unit vector tangent to $g_x$ at $x$, the map $x \mapsto \bigl\Vert \Sigma_{g_x g_{d^-}^+}\bigl(v_a(u_{d^-}^+) \bigr) \bigr\Vert_{u_x}$ is H\"older continuous, because $g_x$ depends Lipshitz continuously on $x$ by \cite[\S 5.2.6]{CEG}), and because the slithering map $\Sigma_{gg_{d^-}^+}$ is a H\"older continuous function of the leaf $g$ by Proposition~\ref{prop:Slithering}. Using again the fact that $\widetilde k \cap \widetilde \lambda$ has Hausdorff dimension 0, it follows that  
$$
 \int_{k-d^+ \cup d^-} \omega_a 
=\sum_d \log\frac
{\bigl\Vert \Sigma_{g_{d}^-g_{d^-}^+}\bigl(v_a(u_{d^-}^+) \bigr) \bigr\Vert_{u_{d}^-}}
{\bigl\Vert\Sigma_{g_{d}^+g_{d^-}^+}\bigl(v_a(u_{d^-}^+) \bigr) \bigr\Vert_{u_{d}^+}}
=  \log \frac
{\Vert v_a (u_{d^-}^+) \Vert_{u_{d^-}^+}}
{\bigl\Vert \Sigma_{g_{d^+}^-g_{d^-}^+}\bigl(v_a(u_{d^-}^+) \bigr)\bigr\Vert_{u_{d^+}^-}}.
$$

By construction, the slithering map $\Sigma_{g_{d^+}^-g_{d^-}^+}$ sends $\widetilde L_a(u_{d^-}^+)$ to $\widetilde L_a(u_{d^+}^-)$. In particular, there exists a non-zero number $\mu_a$ such that $\Sigma_{g_{d^+}^-g_{d^-}^+} \bigl(  v_a(u_{d^-}^+) \bigr) = \mu_a v_a(u_{d^+}^-)$. Then,
\begin{align*}
 \int_{k-d^+ \cup d^-}( \omega_a - \omega_{a+1} )
&=
 \log \frac
{\Vert v_a (u_{d^-}^+) \Vert_{u_{d^-}^+}}
{\Vert v_{a+1} (u_{d^-}^+) \Vert_{u_{d^-}^+}}
- 
 \log \frac
{\Vert v_a (u_{d^+}^-) \Vert_{u_{d^+}^-}}
{\Vert v_{a+1} (u_{d^+}^-) \Vert_{u_{d^+}^-}}
-\log\left|\frac{\mu_a}{\mu_{a+1}} \right| \\
&=
 \log \frac
{\Vert v_a (u_{d^-}^+) \Vert_{u_{d^-}^+}}
{\Vert v_{a+1} (u_{d^-}^+) \Vert_{u_{d^-}^+}}
- 
 \log \frac
{\Vert v_a (u_{d^+}^-) \Vert_{u_{d^+}^-}}
{\Vert v_{a+1} (u_{d^+}^-) \Vert_{u_{d^+}^-}}
+ \sigma_a^\rho(k)
\end{align*}
by definition of the shear parameter $\sigma_a^\rho(k)$ (use Lemma~\ref{lem:ComputeDoubleRatios}). 
\end{proof} 

\begin{lem}
\label{lem:EvaluateOmegaIsIntersection}
For every homology class $[\alpha] \in H_1(\widehat U; \R)$, 
$$
 \bigl \langle  [\omega_a] - [\omega_{a+1}], [\alpha]   \bigr\rangle = [\alpha] \cdot [\sigma_a^\rho].
$$
\end{lem}
\begin{proof}
We already observed, in the proof of Proposition~\ref{prop:RelativeHomologyTangentCycles}, that $H_1(\widehat U, \delh \widehat U; \R)$ admits a basis where each element is represented by a generic tie of $\widehat U$. We can therefore  write  the image $[\alpha]\in H_1(\widehat U, \delh \widehat U; \R)$ as a linear combination $[\alpha] = \sum_i \mu_i [k_i]$ of classes represented by generic  ties $k_i$, with coefficients $\mu_i\in \R$. 

Recall that the ties of $\widehat U$ are oriented to the right for the canonical orientation of the leaves of $\widehat \lambda$. In particular, the components of  the horizontal boundary $\delh \widehat U$ are of two types: those components where the orientation of the ties point outside of $\widehat U$, and those where it points inside. Also, because of this orientation convention, 
$$
 [\alpha] \cdot [\sigma_a^\rho]  =  \sum_i \mu_i \,[k_i]  \cdot [\sigma_a^\rho]  = \sum_i \mu_i \, \sigma_a^\rho(k_i)
$$
by definition of the homology class $[\sigma_a^\rho] \in H_1(\widehat U, \delv \widehat U; \R)$ associated to the relative tangent cycle $\sigma_a^\rho \in \CC(\widehat \lambda, \slits; \R)$ by Proposition~\ref{prop:RelativeHomologyTangentCycles}. 

We now modify each arc $k_i$ by a homotopy respecting $\widehat \lambda$ and $\delh \widehat U$ to obtain an arc $k_i'$ such that the following holds: for every component $C$ of the horizontal boundary $\delh \widehat U$, there is an arc $k_C\subset \widehat U$ such that, for every  arc $k_i'$ with an endpoint in $C$, the component of $k_i'-\widehat\lambda$ containing this endpoint is equal to $k_C$. The only case where this regrouping of arcs near the horizontal boundary requires some care is when the original tie $k_i$ meets $\widehat\lambda$ in one point; in this special situation, one  needs  to first choose the relevant arcs $k_C$ so that $k_i'=k_i$, and then modify the other $k_j$ accordingly. 

Now, by Lemma~\ref{lem:IntegrateOmegaTransverseArc},
\begin{align*}
\sum_i \mu_i\, \sigma_a^\rho(k_i') &=   \sum_i \mu_i  \int_{k_i'}( \omega_a -    \omega_{a+1} ) \\
&\qquad -   \sum_i \mu_i  \int_{d_i^+ }( \omega_a 
-   \omega_{a+1} ) -  \sum_i \mu_i   \int_{ d_i^-}( \omega_a 
-   \omega_{a+1} )\\
&\qquad+  \sum_i \mu_i  \log\frac{\Vert v_a (u_{d_i^+}^-) \Vert_{u_{d_i^+}^-}}{\Vert v_{a+1} (u_{d_i^+}^-) \Vert_{u_{d_i^+}^-}}
-  \sum_i \mu_i \log \frac{\Vert v_a (u_{d_i^-}^+) \Vert_{u_{d_i^-}^+}}{\Vert v_{a+1} (u_{d_i^-}^+) \Vert_{u_{d_i^-}^+}}
\end{align*}
where $d_i^+$ and $d_i^-$ are the components of $k_i' - \widehat\lambda$ containing the positive and negative components of $k_i'$, respectively. In particular, each $d_i^\pm$ is equal to one of the arcs $k_C$ associated to the components $C$ of the horizontal boundary $\delh \widehat U$.

The key observation is now that $[\alpha] =\sum_i \mu_i [k_i] \in H_1(\widehat U, \delh \widehat U; \R)$ comes from an element of $H_1(\widehat U; \R)$, and in particular has boundary 0. This implies that, for each component $C$ of $\delh \widehat U$ where the ties point outwards, the sum of the $\mu_i$ such that $k_i$ has an endpoint in $C$ is equal to 0; equivalently, the $\mu_i$ such that $d_i^+=k_C$ add up to 0. Similarly, for each component $C$ of $\delh \widehat U$ where the ties point inwards, the sum of the coefficients $\mu_i$ such that $d_i^- = k_C$ is equal to 0. 

This implies that most terms cancel out in the above sum, and that
$$
\sum_i \mu_i \sigma_a^\rho(k_i') =   \sum_i \mu_i  \int_{k_i'}( \omega_a -    \omega_{a+1} )=   \bigl \langle  [\omega_a] - [\omega_{a+1}], [\alpha]   \bigr\rangle.
$$
For the second equality note that, because  the $\mu_i$ for which the positive (resp. negative) endpoint of $k_i'$ is in a given component $C$ of $\delh\widehat U$ add up to 0,  the chain $\sum_i \mu_i k_i'$ is closed and represents the class  $[\alpha]\in H_1(\widehat U; \R)$. 

This proves that 
\begin{equation*}
 [\alpha] \cdot [\sigma_a]   = \sum_i \mu_i  \sigma_a(k_i) = \sum_i \mu_i  \sigma_a(k_i') =   \bigl \langle  [\omega_a] - [\omega_{a+1}], [\alpha]   \bigr\rangle. \qedhere
\end{equation*}
\end{proof}

The combination of Lemmas \ref{lem:LengthIsEvaluationOmega} and \ref{lem:EvaluateOmegaIsIntersection} completes the proof of Theorem~\ref{thm:ShearingAndLength}.
\end{proof}

\begin{cor}
\label{cor:TransMeasureHasPositiveIntersection}
Let $\mu$ be a non-trivial transverse measure for the orientation cover $\widehat\lambda$, and let $[\mu] \in H_1(\widehat U; \R)$ be its associated homology class as in \S {\upshape\ref{subsect:SpaceTangentCycles}}. Then,
$$
 [\mu] \cdot [\sigma_a^\rho] >0
$$
for each component $\sigma_a^\rho \in \CC(\widehat \lambda, \slits; \R) \cong H_1(\widehat U, \delv \widehat U; \R)$ of the shearing cycle $\sigma^\rho \in \CC(\lambda,\slits; \widehat\R^{n-1})$ of a Hichin character $\rho \in \Hit(S)$.  
\end{cor}

\begin{proof}
This is an immediate consequence of Theorem~\ref{thm:ShearingAndLength} and Proposition~\ref{prop:MeasureHasPositiveLengths}. 
\end{proof}

 \section{Parametrizing Hitchin components}
 \label{bigsect:Param}
 
 In \S \ref{bigsect:TriangleInv} and \S \ref{subsect:ShearingCycle}, we associated certain invariants to a Hitchin character $\rho \in \Hit(S)$. 
 
 The first type of invariants are the triangle invariants $\tau_{abc}^\rho(s)$, defined as $s$ ranges over the slits of $\lambda$ and $a$, $b$, $c\geq 1$ range over all integers such that $a+b+c=n$.  Noting that there are $\frac{(n-1)(n-2)}2$ such triples $(a,b,c)$ and $12(g-1)$ slits of $S-\lambda$, we can combine all these invariants into a single map
 $$
  \Hit(S) \to \R^{6(g-1)(n-1)(n-2)}.
 $$
 
 The second invariant is the shearing cycle $\sigma^\rho \in \mathcal C(\lambda, \slits;\widehat\R^n)$, which provides a map
 $$
  \Hit(S) \to  \mathcal C(\lambda, \slits;\widehat\R^n) \cong \R^{18(g-1)(n-1)}. 
 $$

Combining these two maps, we define
$$
\Phi   \colon \Hit(S) \to \R^{6(g-1)(n-1)(n-2)} \times \mathcal C(\lambda, \slits ;\widehat\R^n)\cong \R^{6(g-1)(n+1)(n-1)},
$$
which sends each Hitchin character $\rho \in \Hit(S)$ to its triangle invariants and its shearing cycle. We will show that $\Phi$ induces a homeomorphism between $\Hit(S)$ and an open convex polyhedral cone  $\mathcal P$  contained in a linear subspace of $ \R^{6(g-1)(n+1)(n-1)}$.

\begin{lem}
\label{lem:InvariantsContinuous}
The above map 
$$
\Phi \colon \Hit(S) \to \R^{6(g-1)(n-1)(n-2)} \times \mathcal C(\lambda, \slits ;\widehat\R^n)
$$
is continuous.
\end{lem}
\begin{proof} The key property is that the flag curve $\mathcal F_\rho \colon \partial_\infty \widetilde S \to \Flag$ depends continuously on the Hitchin homomorphism $\rho\colon \pi_1(S) \to \PSL$, and is uniformly H\"older continuous as $\rho$ ranges over a compact subset of the space of homomorphisms $ \pi_1(S) \to \PSL$. These two properties  follow from the application to the setup of \S \ref{subsect:FlagCurve} of the classical structural stability theorems for Anosov flows, and  H\"older continuity properties for their stable and unstable foliations; see for instance \cite[\S18--19]{KatHas}. 

The continuity property immediately shows  that the triangle invariants $\tau_{abc}^\rho(s)$ depend continuously on $\rho$. 

The case of the shearing cycle $\sigma^\rho \in \mathcal C(\lambda, \slits;\widehat\R^n)$ requires an additional argument, because its construction relies on the slithering maps $\Sigma_{gg'}\colon \R^n \to \R^n$. The uniform H\"older continuity property makes the estimates used in the construction of  slithering maps in \S \ref{subsect:Slithering}  uniform, and guarantees uniform convergence in this construction. It follows that, for any two leaves $g$, $g'$ of $\widetilde \lambda$, the slithering map $\Sigma_{gg'}$ depends continuously on $\rho$. After this, the continuous dependence of the flag map $\mathcal F_\rho$ on $\rho$ is enough to prove that $\sigma^\rho$ depends continuously on $\rho$. 
\end{proof}

\subsection{Constraints between invariants}
\label{subsect:ConstraintsInvariants}

There are clear constraints on the image of $\Phi$. The first one is the following consequence of Lemma~\ref{lem:SymmetriesTripleRatios}, which we have already encountered in Lemma~\ref{lem:RotateTriangleInvariants}.

\medskip
\noindent\textsc{Triangle Rotation Condition:} If the spikes of the component $T$ of $S-\lambda$ are indexed as  $s$, $s'$, $s''$ in counterclockwise order around $T$, then$$
\tau_{abc}^\rho(s) = \tau_{bca}^\rho(s') = \tau_{cab}^\rho(s'').
$$
\smallskip

The second constraint comes from the quasi-additivity property of the  shearing cycle $\sigma^\rho$. Recall that the lack of additivity  of the $a$--th component $\sigma_a^\rho \in \CC(\widehat\lambda, \slits ; \R^{n-1})$ of $\sigma^\rho \in \CC(\lambda, \slits ; \widehat \R^{n-1}) \subset \CC( \widehat \lambda, \slits; \R^{n-1})$ is measured by its boundary $\partial\sigma^\rho_a$, which  associates a number $\sigma^\rho_a(\widehat s)\in \R$ to each spike $\widehat s$ of the orientation cover $\widehat \lambda$ of the geodesic lamination $\lambda$. The spikes $\widehat s$ can be \emph{positive} of \emph{negative}, according to whether the canonical orientation of the leaves of $\widehat\lambda$ orients the two leaves that are adjacent to $\widehat s$ towards $\widehat s$ or away from $\widehat s$. 

The following constraint comes from the computation of $\partial \sigma_a^\rho$ provided by Lemmas~\ref{lem:ShearingCycleIsCycle} and \ref{lem:ExpressThetaTriangleInvariants}. 

\medskip
\noindent\textsc{Shearing Cycle Boundary Condition:} For every positive slit $ s^+$ of $\widehat\lambda$ projecting to a slit $s$ of $\lambda$, 
$$
\partial \sigma^\rho_a( s^+) =\sum_{b+c=n-a} \tau^\rho_{abc}(s) .
$$
\medskip

Note that this property for positive slits,  combined with the equivariance property of $\sigma^\rho \in \CC(\lambda, \slits ; \widehat \R^{n-1}) \subset \CC( \widehat \lambda, \slits; \R^{n-1})$ with respect to the covering involution of the cover $\widehat\lambda \to \lambda$, determines  $\partial \sigma^\rho_a$ on negative slits. More precisely,
$$
\partial \sigma^\rho_a( s^-) =-\sum_{b+c=a} \tau^\rho_{(n-a)bc}(s)
$$
for every negative slit $ s^-$ of $\widehat\lambda$ projecting to a slit $s$ of $\lambda$.

The last condition is provided by Corollary~\ref{cor:TransMeasureHasPositiveIntersection}. 

\medskip
\noindent\textsc{Positive Intersection Condition:} 
$$
[\mu ]\cdot [\sigma^\rho_a] >0
$$
for every transverse measure $\mu$ for $\widehat\lambda$, where $[\mu] \in H_1(\widehat U; \R)$ and $[\sigma_a^\rho] \in H_1(\widehat U, \delv \widehat U; \R)$ are the homology classes respectively defined by $\mu\in \CC(\widehat \lambda; \R)$ and by the $a$--th component $\sigma_a^\rho \in \CC(\widehat\lambda, \slits; \R)$ of the shearing cycle $\sigma^\rho \in \CC(\lambda; \widehat \R^{n-1}) \subset \CC( \widehat \lambda, \slits; \R^{n-1})$, and where $\cdot$ denotes the algebraic intersection in~$\widehat U$. 
\medskip

Let $\mathcal P $ be the set of pairs $(\tau, \sigma)$ such that
\begin{enumerate}
\item $\tau$ is a function associating a number $\tau_{abc}(s)\in \R$ to each triple of integers $a$, $b$, $c\geq 1$ with $a+b+c=n$, and to each slit $s$ of $\lambda$;
\item $\sigma \in \CC(\lambda, \slits; \widehat \R^{n-1})$ is a tangent cycle for $\lambda$ valued in the coefficient bundle $\widehat\R^{n-1}$ and relative to the slits of $\lambda$; in particular, $\sigma$ is defined by $n-1$ relative tangent cycles $\sigma_a \in \CC(\widehat\lambda, \slits; \R)$;
\item $\tau$ and $\sigma$ satisfy the above Triangle Rotation Condition, Shearing Cycle Boundary Condition and Positive Intersection Condition. 
\end{enumerate}

We will call a function $\tau \in  \R^{6(g-1)(n-1)(n-2)} $ as in {\setcounter{enumi}{1}\labelenumi} a \emph{triangle data function}. It is \emph{rotation invariant} when it satisfies the Triangle Rotation Condition.

\begin{prop}
\label{prop:ComputePolytopeDim}
The space $\mathcal P$ is an open convex polyhedral cone in a $2(g-1)(n^2-1)$--dimensional subspace of $\R^{6(g-1)(n-1)(n-2)} \times  \CC(\lambda, \slits; \widehat \R^{n-1})$. 
\end{prop}

\begin{proof} The transverse measures for the geodesic lamination $\widehat\lambda$ form a positive cone over a finite-dimensional simplex \cite{Kato, Papa}. It therefore suffices to check the Positive Intersection Condition on the vertices of this simplex (corresponding to ergodic measures). This reduces the Positive Intersection Condition to finitely many linear inequalities. As a consequence, $\mathcal P$ is an open convex polyhedral cone in the linear subspace of $ \R^{6(g-1)(n-1)(n-2)} \times \mathcal C(\lambda, \slits ;\widehat\R^n)$ defined by the Triangle Rotation Condition and the Shearing Cycle Boundary Condition. We need to compute its dimension, which will require a few lemmas. 

The Triangle Rotation Condition divides the dimension of the space of triangle data functions by $3$, in the sense that the space of rotation invariant triangle data functions $\tau \in \R^{6(g-1)(n-1)(n-2)}$  is  isomorphic to $\R^{2(g-1)(n-1)(n-2)}$. Indeed, if we pick a spike $s_j$ for each triangle component $T_j$ of $S-\lambda$, such a rotation invariant $\tau$ is completely determined by the $2(g-1)(n-1)(n-2)$ numbers $\tau_{abc}(s_j)$. We will use this observation to denote by $\R^{2(g-1)(n-1)(n-2)}$ the space of all rotation invariant triangle data functions $\tau$. 

Consider the linear subspace $\mathcal L \subset  \R^{2(g-1)(n-1)(n-2)} \times  \CC(\lambda, \slits; \widehat\R^{n-1})$ consisting of all  pairs $(\tau, \sigma)$ where $\tau$ is a rotation invariant triangle data function, where $\sigma$ is a twisted tangent cycle for $\lambda$ relative to its slits, and where $\tau$ and $\sigma$ satisfy the Shearing Cycle Boundary Condition.

To analyze $\mathcal L$, we introduce a new vector space $\CC(\slits; \R^{n-1})$, consisting of all functions $\theta \colon \{ \text{slits of } \lambda\} \to \R^{n-1}$. For $a=1$, $2$, \dots, $n-1$, we denote the $a$--th component of such a $\theta\in \CC(\slits; \R^{n-1})$ by  $\theta_a \colon\{ \text{slits of } \lambda\} \to \R$. The definition of the space $\mathcal L$ can then be expressed in terms of two maps $\partial\colon  \CC(\lambda, \slits; \widehat\R^{n-1}) \to  \CC(\slits; \R^{n-1})$ and $\Theta \colon \R^{2(g-1)(n-1)(n-2)} \to \CC(\slits; \R^{n-1})$. 

The first map $\partial\colon  \CC(\lambda, \slits; \widehat\R^{n-1}) \to  \CC(\slits; \R^{n-1})$ is the usual boundary map, and associates to a relative cycle $\sigma \in  \CC(\lambda, \slits; \widehat\R^{n-1})$ the restriction $\partial \sigma \colon  \{ \text{positive slits of } \widehat\lambda\} = \{ \text{slits of } \lambda\} \to \R^{n-1}$ of its boundary $\partial \sigma$ to positive slits of the orientation cover $\widehat \lambda$. (Recall that this restriction completely determines $\partial \sigma$ by definition of twisted relative tangent cycles, as $\partial \sigma_a(s^-) =- \partial\sigma_{n-a}(s^+)$ when the negative slit $s^-$ of $\widehat\lambda$ projects to the same slit of $\lambda$ as the positive slit $s^+$.)

The second map $\Theta \colon \R^{2(g-1)(n-1)(n-2)} \to \CC(\slits; \R^{n-1})$ associates to each rotation invariant triangle data function $\tau \in \R^{2(g-1)(n-1)(n-2)}$ the function $\theta^\tau \colon \{ \text{slits of } \lambda\} \to \R^{n-1}$ defined by the property that
$$
\theta_a^\tau (s) = \sum_{b+c=n-a} \tau_{abc}(s) \in\R
$$
for every slit $s$ of $\lambda$ and every $a=1$, $2$, \dots, $n-1$.

Then the subspace $\mathcal L$ consists of all pairs  $( \tau, \sigma) \in \R^{2(g-1)(n-1)(n-2)}   \times   \CC(\lambda, \slits; \widehat\R^{n-1})$ such that  $\partial \sigma = \Theta(\tau)$ in $ \CC(\slits; \R^{n-1})$. 

\begin{lem}
\label{lem:ImageBoundaryMap}
The image of  $\partial\colon  \CC(\lambda, \slits; \widehat\R^{n-1}) \to  \CC(\slits; \R^{n-1})$ consists of all $\theta \in \CC(\slits; \R^{n-1})$ such that
$$
\sum_{s\text{ slit of } \lambda} \theta_a(s) = \sum_{s\text{ slit of } \lambda} \theta_{n-a}(s) 
$$
for every $a=1$, $2$, \dots, $n-1$. This image has codimension $\lfloor \frac{n-1}2 \rfloor$ in $\CC(\slits; \R^{n-1}) \cong \R^{12(g-1)(n-1)}$.
\end{lem}

\begin{proof} This is an immediate consequence of the homological interpretation of twisted relative tangent cycles in \S \ref{subsect:TwistedRelTangentCycles}, and more precisely of the isomorphism $\CC(\lambda, \slits; \widehat\R^{n-1}) \cong H_1( U, \delv U; \widetilde \R^{n-1})$ constructed there. 

This construction is well behaved with respect to the boundary maps $\partial$ in the following sense. There is a unique isomorphism $\CC(\slits;\R^{n-1} )\cong H_0(\delv U; \widetilde \R^{n-1})$ defined as follows: this isomorphism  associates to $\theta\in \CC(\slits; \R^{n-1})$ the element of $H_0 (\delv U; \widetilde \R^{n-1}) \subset H_0(\delv \widehat U; \R^{n-1}) $ that assigns to each component of $\delv \widehat U$ facing a positive slit $s^+$ the multiplicity $\theta(s) \in \R^{n-1}$ associated by $\theta$ to the projection $s$ of $s^+$ (and assigns  multiplicity $-\theta_{n-a}(s)$ to the component of $\delv \widehat U$ facing a negative slit $s^-$ projecting to $s$).  Then, for these isomorphisms  $\CC(\lambda, \slits; \widehat\R^{n-1}) \cong H_1( U, \delv U; \widetilde \R^{n-1})$ and $\CC(\slits;\R^{n-1} )\cong H_0(\delv U; \widetilde \R^{n-1})$, the boundary homomorphism $\partial \colon \CC(\lambda, \slits; \widehat\R^{n-1}) \to \CC(\slits;\R^{n-1} )$ corresponds to the homological boundary $\partial \colon   H_1( U, \delv U; \widetilde \R^{n-1}) \to H_0(\delv U; \widetilde \R^{n-1})$. 

Lemma~\ref{lem:ImageBoundaryMap} is then  an immediate consequence of the long exact sequence
$$
\dots \to H_1( U, \delv U; \widetilde \R^{n-1}) \to H_0(\delv U; \widetilde \R^{n-1})\to  H_0( U; \widetilde \R^{n-1})\to H_0( U, \delv U; \widetilde \R^{n-1}) ,
$$
using the properties that, because  $\widehat U$ is connected and $\delv \widehat U$ is non-empty,  $\dim H_0( U; \widetilde \R^{n-1}) = \lfloor \frac{n-1}2 \rfloor$ and $ H_0( U, \delv U; \widetilde \R^{n-1}) =0$. 
\end{proof}

\begin{lem}
\label{lem:ImageTheta}
For $n>3$, the image of  $\Theta \colon \R^{2(g-1)(n-1)(n-2)} \to \CC(\slits; \R^{n-1})$ consists of all $\theta \in \CC(\slits; \R^{n-1})$ such that
\begin{align*}
 \theta_{n-1}(s_1) &= 0 \\
\text{and }
\theta_1(s_1) &= \sum_{a=2}^{n-2} ({\textstyle \frac{a-1}{n-3}} -1) \theta_a(s_1) +   \sum_{a=2}^{n-2} {\textstyle \frac{a-1}{n-3}}\, \theta_a(s_2)+  \sum_{a=2}^{n-2} {\textstyle \frac{a-1}{n-3}}\, \theta_a(s_3)
\end{align*}
whenever $s_1$, $s_2$ and $s_3$ are the three spikes of the same component $T$ of $S-\lambda$. In particular, the image of $\Theta$ has dimension $ 12(g-1)(n-3)$. 

When $n=3$, the image of  $\Theta \colon \R^{4(g-1)} \to \CC(\slits; \R^{2})$ consists of all $\theta \in \CC(\slits; \R^{2})$ such that
\begin{align*}
 \theta_{2}(s_1) &= 0 \\
\text{and }
\theta_1(s_1) &= \theta_1(s_2) = \theta_1(s_3) 
\end{align*}
whenever $s_1$, $s_2$ and $s_3$ are the three spikes of the same component $T$ of $S-\lambda$. In particular, the image of $\Theta$ then has dimension $ 4(g-1)$. 
\end{lem}

\begin{proof} By definition, if $\theta^\tau = \Theta(\tau)$ for a rotation invariant function $\tau\in \R^{2(g-1)(n-1)(n-2)}$ , then $\theta_{n-1}^\tau (s) = \sum_{b+c=1} \tau_{(n-1)bc}(s)=0$ for every slit $s$  since all indices $b$, $c$ are supposed to be at least 1. 

Less trivially, if $n>3$ and if $s_1$, $s_2$, $s_3$  are the three spikes of a same component $T$ of $S-\lambda$, in this order counterclockwise around $T$, 
\begin{align*}
\sum_{a=2}^{n-2} {\textstyle \frac{a-1}{n-3}} \, \theta_a^\tau (s_1) &+   \sum_{a=2}^{n-2} {\textstyle \frac{a-1}{n-3}}\, \theta_a^\tau(s_2)+  \sum_{a=2}^{n-2} {\textstyle \frac{a-1}{n-3}}\, \theta_a^\tau(s_3)\\
&= 
\sum_{a=1}^{n-2}{\textstyle \frac{a-1}{n-3}} \kern -4pt \sum_{b+c=n-a}  \kern -8pt \tau_{abc} (s_1) + 
\sum_{a=1}^{n-2}{\textstyle \frac{a-1}{n-3}} \kern -4pt \sum_{b+c=n-a} \kern -8pt  \tau_{abc} (s_2) + 
\sum_{a=1}^{n-2}{\textstyle \frac{a-1}{n-3}} \kern -4pt \sum_{b+c=n-a} \kern -8pt \tau_{abc} (s_3) \\
&= 
\sum_{a=1}^{n-2}{\textstyle \frac{a-1}{n-3}} \kern -4pt \sum_{b+c=n-a}\kern -8pt  \tau_{abc} (s_1) + 
\sum_{b=1}^{n-2}{\textstyle \frac{b-1}{n-3}} \kern -4pt \sum_{a+c=n-b} \kern -8pt \tau_{abc} (s_1) + 
\sum_{c=1}^{n-2}{\textstyle \frac{c-1}{n-3}} \kern -4pt \sum_{a+b=n-c} \kern -8pt \tau_{abc} (s_1) \\
&= \sum_{a,b,c} \bigl( {\textstyle \frac{a-1}{n-3}} + {\textstyle \frac{b-1}{n-3}} + {\textstyle \frac{c-1}{n-3} }\bigr)   \tau_{abc} (s_1)
= \sum_{a,b,c}   \tau_{abc} (s_1)= \sum_{a=1}^{n-1} \theta_a^\tau(s_1)
\end{align*}
where the second equality uses the rotation invariance of $\tau$. It follows that 
$$
\theta_1^\tau(s_1) 
= \sum_{a=2}^{n-2} \bigl( {\textstyle \frac{a-1}{n-3}} -1)  \theta_a^\tau(s_1) +   \sum_{a=2}^{n-2} {\textstyle \frac{a-1}{n-3}}\, \theta_a^\tau(s_2)+  \sum_{a=2}^{n-2} {\textstyle \frac{a-1}{n-3}}\, \theta_a^\tau(s_3).
$$

As a consequence, any function $\theta= \Theta(\tau)$ in the image of $\Theta$ satisfies the relations of Lemma~\ref{lem:ImageTheta}. 

Conversely, as  $a$ ranges from $2$ to $n-2$ and $s$ ranges over all slits of $\lambda$, the functions $\tau \mapsto \theta_a^\tau(s)$ are linearly independent over the space $\R^{2(g-1)(n-1)(n-2)}$ of rotation invariant triangle data functions $\tau$. Indeed, this follows from a simple computation focusing on the coefficients of the terms $\tau_{1bc}(s)$ and $\tau_{2bc}(s)$ in any linear relation between these functions. 

The dimension computation then follows from the fact that $\lambda$ has $12(g-1)$ slits. This completes the proof of Lemma~\ref{lem:ImageTheta} in the case considered, when $n>3$. 

The proof is much simpler when $n=3$, as the triangle data function $\tau$ assigns only one number $\tau_{111}(s)$ to each slit $s$. This makes the argument in this case completely straightforward.  
\end{proof}

\begin{lem}
\label{lem:IntersectionImagesBoundaryMapTheta}
The intersection $\mathrm{im}(\partial) \cap \mathrm{im}(\Theta)$ of the images $\mathrm{im}(\partial) = \partial \bigl(   \CC(\lambda, \slits; \widehat\R^{n-1}) \bigr)$ and $ \mathrm{im}(\Theta) = \Theta(\R^{2(g-1)(n-1)(n-2)})$ has dimension $ 12(g-1)(n-3) - \lfloor \frac{n-1}2 \rfloor $ if $n>3$, and $4g-5$ if $n=3$. 
\end{lem}
\begin{proof}
This is an immediate consequence of the characterization of these images in Lemmas~\ref{lem:ImageBoundaryMap} and \ref{lem:ImageTheta}.  Indeed, one very easily checks that the restrictions of the $\lfloor \frac{n-1}2 \rfloor $ relations of Lemma~\ref{lem:ImageBoundaryMap} to the image $\mathrm{im}(\Theta)$  are  linearly independent. 
\end{proof}

We now return to the subspace $\mathcal L \subset \R^{2(g-1)(n-1)(n-2)}   \times   \CC(\lambda, \slits; \widehat\R^{n-1})$, consisting of all pairs  $( \tau, \sigma)$ such that  $\partial \sigma = \Theta(\tau)$ in $ \CC(\slits; \R^{n-1})$.  The maps $\Theta$ and $\partial$ combine to give a  linear map $ \mathcal L \to  \CC(\slits; \R^{n-1})$,  whose image is $\mathrm{im}(\partial) \cap \mathrm{im}(\Theta)$ and whose kernel is the direct sum of $\ker \Theta$ and $\ker \partial$. Note that $\ker\partial $ is just the space $\CC(\lambda; \widehat \R^{n-1})$ of closed tangent cycles. Therefore, by combining Lemma~\ref{lem:IntersectionImagesBoundaryMapTheta},  Lemma~\ref{lem:ImageTheta} and  Proposition~\ref{prop:ComputeTwistedTgentCycles},
\begin{align*}
\dim \mathcal L &= \dim \mathrm{im}(\partial) \cap \mathrm{im}(\Theta) + \dim \ker \Theta + \dim \ker \partial\\
&= 12(g-1)(n-3) - {\textstyle\lfloor \frac{n-1}2 \rfloor}  \\
&\qquad\qquad\qquad+  2(g-1)(n-1)(n-2) -12(g-1)(n-3)  \\
&\qquad\qquad\qquad+  6(g-1)(n-1) + {\textstyle\lfloor \frac{n-1}2 \rfloor } \\
&= 2(g-1)(n^2-1)
\end{align*}
when $n>3$. 

When $n=3$ the same argument gives that 
$$
\dim \mathcal L=  ( 4g-5 ) + 0 + ( 12g-11) = 16(g-1),
$$
which is equal to $2(g-1)(n^2-1)$ in this case as well. 

Since $\mathcal P$ is an open convex polyhedral cone in the space $\mathcal L$, this concludes the proof of Proposition~\ref{prop:ComputePolytopeDim}.
\end{proof}

\begin{cor}
\label{cor:InvariantsLocalHomeo}
The map $\Phi\colon \Hit(S) \to \mathcal P$ is a local homeomorphism. 
\end{cor}
\begin{proof}
The map $\Phi$ is continuous by Lemma~\ref{lem:InvariantsContinuous}, and injective by Corollary~\ref{cor:HitchinDeterminedByInvariants}. By the Invariance of Domain Theorem, it is therefore a local homeomorphism since $\Hit(S)$ and $\mathcal P$ have the same dimension by Proposition~\ref{prop:ComputePolytopeDim}. 
\end{proof}

\subsection{An estimate from the Positive Intersection Condition} 
\label{subsect:PositiveIntersectionRevisited}

This section is devoted to an estimate that will be crucial to prove that the above map $\Phi\colon \Hit(S) \to \mathcal P$ is a global homeomorphism. 

In the universal cover $\widetilde S$ of $S$, we want to introduce a measure of the topological complexity of the components $T$ of the complement $\widetilde S - \widetilde \lambda$ of the preimage $\widetilde\lambda$ of the maximal geodesic lamination $\lambda$. For this, we choose a  train track neighborhood $U$ of $\lambda$, with preimage $\widetilde U$ in $\widetilde S$. 

We also select an oriented arc $\widetilde k$ tightly transverse to $\widetilde \lambda$ in $\widetilde S$; recall that this means that $\widetilde k$ is transverse to the leaves of $\widetilde\lambda$ and that, for each component $T$ of $\widetilde S - \widetilde\lambda$, the intersection $T\cap \widetilde k$ is either empty, or an arc containing an endpoint of $\widetilde k$, or an arc joining two distinct components of $\partial T$. As in \S \ref{subsect:RelTgtCyclesDifferentView}, using Proposition~\ref{prop:TrainTrackMaxGeodLam}, we can arrange by a homotopy respecting $\widetilde\lambda$ that $\widetilde k$ is contained in $\widetilde U$. 

Let $T$ be a component of $\widetilde S - \widetilde\lambda$ that meets $\widetilde k$, and does not contain any of the endpoints of $\widetilde k$. Then $\widetilde k \cap T$ consists of a single arc since $\widetilde k$ is tightly transverse to $\widetilde\lambda$, and can be joined to the complement $T- \widetilde U$ by a path contained in $T$. We define  the  \emph{divergence radius} $r(T)\geq 1$ of $T$ with respect to $\widetilde U$ and $\widetilde k$ as  the minimum number of edges of $\widetilde U$ that are met by a path  joining $\widetilde k \cap T$ to the complement $T - \widetilde U$ in $T$. 

\begin{lem}
\label{lem:BoundedDivergenceRadius}
For every integer $r_0$, the number of triangles $T$ with divergence radius $r(T) = r_0$ is uniformly bounded, independently of $r_0$. 
\end{lem}
\begin{proof} Instead of counting the components $T$ of $\widetilde S - \widetilde\lambda$ meeting $\widetilde k$, it is easier to count the components of $\widetilde k - \widetilde \lambda$. Cutting $\widetilde k$ into smaller arcs if necessary, we can assume without loss of generality that  $\widetilde k$ is sufficiently short that it projects to an arc $k$ embedded in $S$. Then there is a natural correspondence between the components of $\widetilde k - \widetilde \lambda$ and those of  $k-\lambda$. For each component $d$ of $k-\lambda$, let $T_d$ be the component of $\widetilde S - \widetilde \lambda$ that contains the component of $\widetilde k- \widetilde \lambda$ corresponding to $d$, and define $r(d) = r(T_d)$. We need to show that the number of components $d$ of $ k -  \lambda$ with $r(d)=r_0$ is uniformly bounded.

As $e$ ranges over all edges of the train track neighborhood $U$, the components of $e-\lambda$ form a family of rectangles $R_i$ whose union is equal to $U - \lambda$. In particular, this decomposes $U-\lambda$ in two pieces:
\begin{enumerate}
\item the union of the finitely many rectangles $R_i$ that meet the boundary $\partial U$;
\item $12(g-1)$ infinite chains of rectangles $R_{i_1} \cup R_{i_2} \cup \dots \cup R_{i_k} \cup \cdots$, where each $R_{i_k}$ shares with $R_{i_{k+1}}$ a side contained in a tie of $U$, that form the spikes of $U-\lambda$. 
\end{enumerate}
Compare Proposition~\ref{prop:TrainTrackMaxGeodLam} and Figure~\ref{fig:TrainTrackMaxGeodLam}. 

If $d$ is a component of $k-\lambda$ whose divergence radius $r(d)$ is equal to 1, then it meets one of the finitely many rectangles $R_i$ of (1) above. 
The number of components of $k-\lambda$ meeting a given rectangle $R_i$ is uniformly bounded, by a constant depending on the minimum distance between $\widetilde k$ and its iterates under the action of $\pi_1(S)$. Therefore, there are only finitely many components of $k-\lambda$ with divergence radius 1.

If $d$ is a component of $k-\lambda$ with $r(d)>1$, it is contained in one of the spikes  $R_{i_1} \cup R_{i_2} \cup \dots \cup R_{i_k} \cup \cdots$ as in (2) above. In fact,  $d$ meets the $(r(d)-1)$--th rectangle $R_{i_{r(d)-1}}$ of this spike by definition of the divergence radius $r(d)$. Since the number of components of $k-\lambda$ meeting each $R_i$ is uniformly bounded, and since there are only $12(g-1)$ spikes, it follows that for $r_0>1$ the number of components $d$ of $k-\lambda$ with $r(d)=r_0$ is uniformly bounded. 
\end{proof}

To explain the divergence radius terminology, consider the two sides of $T$ that meet $\widetilde k$. These two leaves of $\widetilde\lambda$ follow the same train route in $\widetilde U$ over a length of approximately $r(T)$ edges (up to a bounded error term) before diverging at some switch of $\widetilde U$. 

The side of the oriented arc $\widetilde k$ where this divergence occurs will greatly matter. There are two possibilities for the two sides of $T$ meeting $\widetilde k$: Either they are asymptotic on the left-hand side of $\widetilde k$, or they are asymptotic on the right-hand side. We will say that $T$ \emph{points to the left} of $\widetilde k$ in the first case, and \emph{points to the right} in the second case. 

Finally, remember that $\widehat \lambda$ denotes the orientation cover of $\lambda$, and that the covering map $\widehat\lambda \to \lambda$ uniquely extends to a cover $\widehat U \to U$ for some train track neighborhood $\widehat U$ of $\widehat \lambda$. . 

Let $T_0$ be the component of $\widetilde S - \widetilde \lambda$ containing the negative endpoint of $\widetilde k$. Using the point of view of \S \ref{subsect:RelTgtCyclesDifferentView}, a relative tangent cycle $\sigma \in \CC(\widehat\lambda, \slits; \R)$ associates a number $\sigma(T_0, T) \in \R$ to each component $T$ of $\widetilde S - \widetilde \lambda$.

\begin{lem}
\label{lem:PositiveIntersectionCondnEstimate}
Suppose that the relative tangent cycle $\sigma \in \CC(\widehat\lambda, \slits; \R)\cong H_1(\widehat U, \delv \widehat U; \R)$ satisfies the following Positive Intersection Property: $[\mu] \cdot [\sigma]>0$ for every transverse measure $\mu$ for $\widehat \lambda$, defining a homology class $[\mu] \in H_1(\widehat U ; \R)$. Then, there exists a constant $C>0$ such that, for all but finitely many components $T$ of $\widetilde S - \widetilde \lambda$ meeting $\widetilde k$, 
\begin{itemize}
\item $\sigma(T_0, T) \geq C r(T)$ if $T$ points to the right of $\widetilde k$;
\item $\sigma(T_0, T) \leq -C r(T)$ if $T$ points to the left of $\widetilde k$.
\end{itemize}

\end{lem}

\begin{proof} 
Pick a tie $k_e$ in each edge $e$ of the train track neighborhood $\widehat U$. Then, for each transverse measure $\mu$ for $\widehat\lambda$, define
$$
\Vert \mu \Vert = \sum_e \mu(k_e)
$$
where the sum is over all edges $e$ of $\widehat U$. This defines a norm $\Vert \ \Vert$ on  the space $\mathcal M(\widehat\lambda) \subset \CC(\widehat \lambda; \R)$ of transverse measures for $\widehat\lambda$.
 The  space of transverse measures of norm 1 is compact for the weak${}^*$ topology, and there consequently exists a number $\epsilon>0$ such that $[\mu] \cdot [\sigma]\geq \epsilon $ for every transverse measure $\mu$ with $\Vert \mu \Vert =1$. We will show that the conclusion of the lemma holds for every $C<\epsilon$.

For this, we use a proof by contradiction. Suppose that the property does not hold. Then, there exists  a sequence of distinct components  $T_n$ of $\widetilde S - \widetilde \lambda$ meeting $\widetilde k$ such that $\sigma(T_0, T_n)< C r(T_n)$ if $T_n$ points to the right of $\widetilde k$, and $\sigma(T_0, T_n)> -C r(T_n)$ if it points to the left. Passing to a subsequence if necessary, we can arrange that either all $T_n$ point to the right, or they all point to the left. 

Let us focus attention on the case where all $T_n$ point to the left, in which case $\sigma(T_0, T_n)>- C  r(T_n)$ for every $n$. The other case will be similar.

Let $\widetilde k_n$ be the subarc of $\widetilde k$ going from the negative endpoint of $\widetilde k$ to an arbitrary point of $\widetilde k \cap T_n$. Let $k_n$ be the projection of $\widetilde k_n \subset \widetilde U$ to $U$. Among the two lifts of $k_n$ to the cover $\widehat U$ of $U$, let $\widehat k_n$ be the one where the canonical orientation of the leaves of $\widehat\lambda$ points to the left for the orientation of $\widehat k_n$ coming from the orientation of $\widetilde k$. (We are here using the fact that $\widetilde k$ is tightly transverse to $\widetilde\lambda$.)
In particular, $\widehat k_n$ is tightly transverse to $\widehat\lambda$ in $\widehat U$, and $\sigma(T_0, T) = \sigma(\widehat k_n)$ by the construction of \S \ref{subsect:RelTgtCyclesDifferentView}.

\begin{figure}[htbp]

\SetLabels
( .45* .1) $ \widehat U$ \\
(.05 * .145) $\delh \widehat U $ \\
(.93 * .49) $ \delh \widehat U$ \\
( .3* .21) $ l_0$ \\
( .243*.21 ) $ t_0$ \\
( .65* .49) $l_n $ \\
(.86 * .445) $ t_n$ \\
( .375* .374) $\widehat k_n'' $ \\
(.365 *.9 ) $ \widehat k$ \\
( .8* .75) $\widehat U $ \\
( -.01* .65) $\widehat\lambda $ \\
\endSetLabels
\centerline{\AffixLabels{\includegraphics{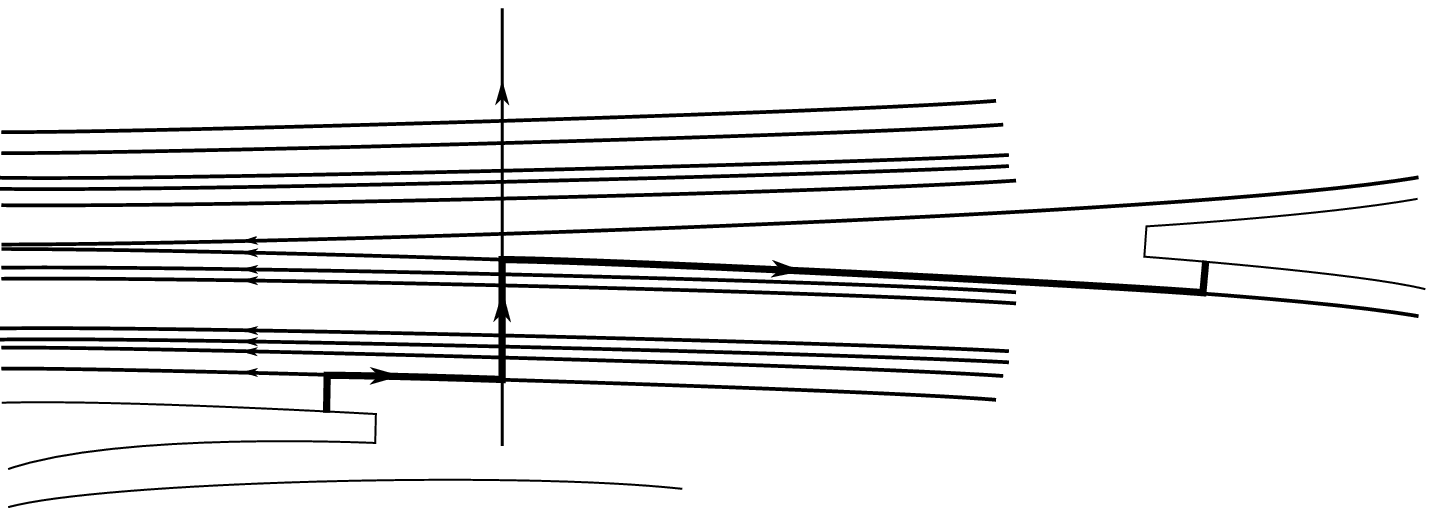}}}
\caption{}
\label{fig:PositiveIntersectionCondition}
\end{figure}

Let $[\widehat k_n] \in H_1(\widehat U, \delh \widehat U; \R)$ be the relative homology class associated to $\widehat k_n$ as in the proof of Proposition~\ref{prop:RelativeHomologyTangentCycles}. Namely, $[\widehat k_n]$ is represented by an arc $\widehat k_n' \subset \widehat U$ with $\partial\widehat k_n' \subset \delh \widehat U$ that is made up of the following five pieces: the arc $\widehat k_n''$ obtained from $\widehat k_n$ by removing the two components of $\widehat k_n-\widehat\lambda$ that contain its endpoints; two arcs $l_0$ and $l_n$ in the leaves of $\widehat\lambda$ that contain the endpoints of $\widehat k_n''$; two arcs $t_0$ and $t_n$ contained in ties of $\widehat U$, with one endpoint in  the horizontal boundary $\delh \widehat U$, with the other endpoint in $\widehat\lambda$, and whose interior is disjoint from $\widehat\lambda$. We choose the indexing so that $l_n$ joins the positive endpoint of $\widehat k_n''$ to the negative endpoint of $t_n$, and $l_0$ joins the positive endpoint of $t_0$ to the negative endpoint of $\widehat k_n''$. In addition, we can arrange that $t_0$ and $l_0$ are independent of $n$. See Figure~\ref{fig:PositiveIntersectionCondition}.

By Step 2 of the proof of Proposition~\ref{prop:RelativeHomologyTangentCycles}, the homology classes $[\sigma] \in  H_1(\widehat U, \delv \widehat U; \R)$ and $[\widehat k_n] = [\widehat k_n'] \in  H_1(\widehat U, \delh \widehat U; \R)$ are such that
$$
  [\widehat k_n']  \cdot [\sigma]= \sigma( \widehat k_n)  = \sigma(T_0, T_n).
$$

By definition of the divergence radius $r(T_n)$, the arc $l_n$ crosses approximately $r(T_n)$ edges of $\widehat U$ (counted with multiplicity). Because the triangles $T_n$ are all distinct, $r(T_n)$ tends to infinity as $n$ tends to $\infty$ by Lemma~\ref{lem:BoundedDivergenceRadius}. Passing to a subsequence if necessary, the standard weak${}^*$ compactness argument provides a nontrivial transverse measure $\mu$ for $\widehat\lambda$ such that 
$$
\int_k \mu = \lim_{n\to \infty}  {\textstyle\frac1{r(T_n)}} \# k\cap l_n
$$
for every arc $k$ transverse to $\widehat\lambda$, where $\# k \cap l_n$ demotes the number of points of $k \cap l_n$. In addition, $\Vert\mu\Vert=1$ by definition of the norm $\Vert \ \Vert$. 

Note that $\widehat k_n' - l_n$ has uniformly bounded length. In addition, the orientation of $l_n$ coming from the orientation of $\widehat k_n'$ is opposite the canonical orientation of the leaf of $\widehat\lambda$ that contains it. Therefore, 
$$
\lim_{n\to \infty} {\textstyle\frac1{r(T_n)}} [\widehat k_n'] = - [\mu] 
$$
in $ H_1(\widehat U, \delh \widehat U; \R)$. 
Intersecting with the class $[\sigma]\in H_1(\widehat U, \delv \widehat U; \R)$ defined by $\sigma \in \CC(\widehat\lambda, \slits; \R)$ then gives
$$
[\mu]\cdot [\sigma] = -\lim_{n\to \infty} {\textstyle\frac1{r(T_n)}} [\widehat k_n'] \cdot [\sigma] = 
-\lim_{n\to \infty} {\textstyle\frac1{r(T_n)}} \sigma(T_0, T_n) \leq C
$$
since $\sigma(T_0, T_n)> -C  r(T_n)$ by hypothesis. 

Therefore, we have constructed a  transverse measure $\mu$ for $\widehat\lambda$ such that $[\mu]\cdot [\sigma] \leq C$ and $\Vert\mu\Vert=1$. But this contradicts our hypothesis that $C<\epsilon \leq [\mu]\cdot [\sigma]$ for any such $\mu$, and provides the contradiction sought when all $T_n$ point to the left of $\widetilde k$.

The argument is similar when all $T_n$ point to the right. The only difference is that the transverse measure $\mu$ then constructed has associated homology class $[\mu] = + \lim_{n\to \infty} {\textstyle\frac1{r(T_n)}} [\widehat k_n'] $ in $ H_1(\widehat U, \delh \widehat U; \R)$, because the orientation of $l_n$ now coincides with the canonical orientation of the leaf of $\widehat\lambda$ containing it. Since the inequality $\sigma(T_0, T_n)<C  r(T_n)$ is also reversed, this again provides a  transverse measure $\mu$ for $\widehat\lambda$ such that $[\mu]\cdot [\sigma] <C<\epsilon$ and $\Vert\mu\Vert=1$, concluding the proof in this case as well. 
\end{proof}

\begin{comp}
\label{comp:PositiveIntersectionCondnEstimate}
The conclusion of Lemma~{\upshape\ref{lem:PositiveIntersectionCondnEstimate}} holds when $\sigma$ is replaced by any $\sigma'$ in a small neighborhood of $\sigma$ in $\CC(\widehat\lambda, \slits; \R)$. 
\end{comp}

\begin{proof}
By compactness of the space of transverse measures $\mu$ with $\Vert\mu\Vert=1$, we can choose $\epsilon>0$ so that $[\mu] \cdot [\sigma']\geq \epsilon $ for every $\sigma' \in \CC(\widehat\lambda, \slits; \R)$ sufficiently close to $\sigma$ and every transverse measure $\mu$ with $\Vert \mu \Vert =1$. Then the proof shows that the conclusion of  Lemma~\ref{lem:PositiveIntersectionCondnEstimate}  holds for any such $\sigma'$ and $C<\epsilon$. 
\end{proof}

\subsection{Realization of invariants, and parametrization of $\Hit(S)$} 
\label{subsect:RealizeInvariants}

At the beginning of \S \ref{bigsect:Param}, we introduced the map
 $$
\Phi  \colon \Hit(S) \to  \mathcal P \subset \R^{6(g-1)(n-1)(n-2)} \times \CC(\lambda, \slits;\widehat\R^n) 
 $$
that associates its triangle invariants and shearing cycle  to a Hitchin character. We showed in \S \ref{subsect:ConstraintsInvariants} that the image of $\Phi$ is contained in the convex polyhedral cone $\mathcal P$ defined by the Triangle Rotation Condition, the Shearing Cycle Boundary Condition, and the Positive Intersection Condition. We also showed in Corollary~\ref{cor:InvariantsLocalHomeo} that $\Phi \colon \Hit(S) \to \mathcal P$ is  a local homeomorphism. 

\begin{prop}
\label{prop:InvariantsProper}
The map $\Phi \colon \Hit(S) \to \mathcal P$ is proper.
\end{prop}
\begin{proof}
We need to prove the following property: Let $(\rho_i)_{i\in \N}$ be a sequence in $\Hit(S)$ such that $\bigl ( \Phi(\rho_i) \bigr)_{i\in\N} = \bigl ((\tau^{\rho_i}, \sigma^{\rho_i}) \bigr)_{i\in\N}$ converges to a point $(\tau^\infty, \sigma^\infty) \in \mathcal P$; then the sequence $(\rho_i)_{i\in \N}$ admits a converging subsequence. 

For this, we will revisit our proof that a Hitchin character is determined by its triangle invariants and its shearing cycle, as in \S \ref{subsect:ParamInjective}. In that proof, we showed that the fundamental group $\pi_1(S)$ is generated by elements $\gamma$ of the type described in Lemma~\ref{lem:NiceGeneratorsForPi1}, and then proved that
$$
\rho_i(\gamma) = \biggl(\kern 15pt
\overleftarrow{\prod_{\kern -13pt T\in \mathcal T_{g_0(\gamma h_0)}\kern -13pt}} \kern 7pt
\Bigl(
\Theta_{E_0F_0}^{\sigma^{\rho_i} (T_0, T)} \circ \widehat\Sigma_T^i \circ  \Theta_{E_0F_0}^{-\sigma^{\rho_i} (T_0, T)}
\Bigr)
\biggr)^{-1}
\circ  \Theta^{\sigma^{\rho_i}(T_0, \gamma T_0)}_{E_0F_0}
\circ \phi_0^i \in \PGL
$$
with the notation of Lemma~\ref{lem:GroupActionAndInvariants} (except that $\widehat\Sigma_T^i$ and $ \phi_0^i$ were respectively called $\widehat\Sigma_T'$ and $ \phi_0$ there).

\begin{lem}
\label{lem:ShearLeftRightEstimate}
There exists a constant $C$, independent of $T$,  such that
$$
\bigl\Vert \Theta_{E_0F_0}^{\sigma^{\rho_i} (T_0, T)} \circ \widehat\Sigma_T^i \circ  \Theta_{E_0F_0}^{-\sigma^{\rho_i} (T_0, T)} - \Id_{\R^n} \bigr\Vert
\leq
C \max_a \mathrm e^{-(n-1)\sigma_a^{\rho_i} (T_0, T)} 
$$
if $T$ points to the right  between $T_0$ and $\gamma T_0$ (as seen from $T_0$), and 
$$
\bigl\Vert \Theta_{E_0F_0}^{\sigma^{\rho_i} (T_0, T)} \circ \widehat\Sigma_T^i \circ  \Theta_{E_0F_0}^{-\sigma^{\rho_i} (T_0, T)} - \Id_{\R^n} \bigr\Vert
\leq
C \max_a \mathrm e^{(n-1)\sigma_a^{\rho_i} (T_0, T)} 
$$
if $T$ points to the left. 
\end{lem}

\begin{proof} Choose for $\R^n$ a basis in which the $a$--th term belongs to the line $E_0^{(a)}\cap F_0^{(n-a+1)}$. Then, by definition, the matrix of $ \Theta_{E_0F_0}^{\sigma^{\rho_i} (T_0, T)} $  in this basis is diagonal, with diagonal entries  $\E^{u_1}$, $\E^{u_2}$, \dots, $\E^{u_n}$
where   $u_1$, $u_2$, \dots, $u_n$ are  uniquely determined by the properties that $ u_{a} - u_{a+1}=  \sigma^{\rho_i}_a (T_0, T)$ and $\sum_{a=1}^n u_a=0$.

Consider for instance the case where $T$ points to the left. Then the map $\widehat\Sigma_T^i $ respects the flag $E_0$, and acts by the identity on each of the lines $E_0^{(a)}/E_0^{(a-1)}$. Therefore, in the above basis for $\R^n$,  the matrix $A$ of $\widehat\Sigma_T^i $ is upper triangular with all diagonal entries equal to 1. 

By construction, the map $\widehat\Sigma_T^i $ is completely determined by, and depends continuously on, the triangle invariants $\tau_{abc}^{\rho_i}(s)$ associated to the slit $s$ of $\lambda$ that is the projection of the spike of $T$ delimited by the two components of $\partial T$ that separate $T_0$ from $\gamma T_0$. Since these triangle invariants converge to $\tau_{abc}^\infty(s)$, we conclude that each $ab$--entry $A_{ab}$ of the matrix $A$ is uniformly bounded by a constant $C$.  We already observed that $A_{ab}=0$ if $a>b$ and $A_{aa}=1$. 

Multiplying matrices, we conclude that for $a<b$ the $ab$--entry  of the matrix of $\Theta_{E_0F_0}^{\sigma^{\rho_i} (T_0, T)} \circ \widehat\Sigma_T^i \circ  \Theta_{E_0F_0}^{-\sigma^{\rho_i} (T_0, T)} - \Id_{\R^n}$ is equal to $A_{ab} \E^{u_a-u_b}$ and bounded by
\begin{align*}
|A_{ab} |\E^{u_a-u_b} & \leq C \E^{u_a-u_b}  = C \E^{-\sum_{c=a}^{b-1} (u_{c+1}-u_c)} = C \E^{+\sum_{c=a}^{b-1} \sigma^{\rho_i}_c (T_0, T)} \\
&\leq C \max_c  \E^{(n-1)\sigma^{\rho_i}_c (T_0, T)} .
\end{align*}
The other entries of this matrix are 0 since $A_{ab}=0$ if $a>b$, and since $A_{aa}=1$. 

This proves the estimate required when the triangle $T$ points to the left.

The proof is almost identical when $T$ points to the right, except that the matrix $A$ is now lower diagonal. 
\end{proof}

We now use the property that the limit $(\tau^\infty, \sigma^\infty)\in \R^{6(g-1)(n-1)(n-2)} \times \CC(\lambda, \slits;\widehat\R^n)  $ actually belongs to the polyhedron $ \mathcal P$, and more precisely the fact that the relative tangent cycle $ \sigma^\infty \in\CC(\lambda, \slits;\widehat\R^n) $ satisfies the Positive Intersection Condition. 

\begin{lem}
\label{lem:BoundedByPositiveIntersection}
For $\gamma \in \pi_1(S)$ as above, the $\rho_i(\gamma)\in \PSL$ are bounded independently of $i$. 
\end{lem}

\begin{proof}
Because $ \sigma^\infty $ satisfies the Positive Intersection Condition, the combination of Lemma~\ref{lem:PositiveIntersectionCondnEstimate}, Complement~\ref{comp:PositiveIntersectionCondnEstimate} and Lemma~\ref{lem:ShearLeftRightEstimate} provides  constants $C$, $D>0$ such that, in the expression
$$
\rho_i(\gamma) = \biggl(\kern 15pt
\overleftarrow{\prod_{\kern -13pt T\in \mathcal T_{g_0(\gamma h_0)}\kern -13pt}} \kern 7pt
\Bigl(
\Theta_{E_0F_0}^{\sigma^{\rho_i} (T_0, T)} \circ \widehat\Sigma_T^i \circ  \Theta_{E_0F_0}^{-\sigma^{\rho_i} (T_0, T)}
\Bigr)
\biggr)^{-1}
\circ  \Theta^{\sigma^{\rho_i}(T_0, \gamma T_0)}_{E_0F_0}
\circ \phi_0^i,
$$
the contribution of  each triangle $T$ is such that 
$$
\bigl\Vert \Theta_{E_0F_0}^{\sigma^{\rho_i} (T_0, T)} \circ \widehat\Sigma_T^i \circ  \Theta_{E_0F_0}^{-\sigma^{\rho_i} (T_0, T)} - \Id_{\R^n} \bigr\Vert
\leq
C \E^{-D r(T)},
$$
for the divergence radius $r(T)$ defined in \S \ref{subsect:PositiveIntersectionRevisited}. In addition, for every integer $r_0\geq 1$, Lemma~\ref{lem:BoundedDivergenceRadius} shows that the number of triangles $T$ such that $r(T)=r_0$ is bounded independently of $r_0$. It follows that the product 
$$
\overleftarrow{\prod_{\kern -13pt T\in \mathcal T_{g_0(\gamma h_0)}\kern -13pt}} \kern 7pt
\Bigl(
\Theta_{E_0F_0}^{\sigma^{\rho_i} (T_0, T)} \circ \widehat\Sigma_T^i \circ  \Theta_{E_0F_0}^{-\sigma^{\rho_i} (T_0, T)}
\Bigr)
$$
converges and is uniformly bounded. 

By construction, the remaining terms $\Theta^{\sigma^{\rho_i}(T_0, \gamma T_0)}_{E_0F_0}$
and $ \phi_0^i$ are completely determined by, and depends continuously on, the triangle and shear invariants of $\rho_i$. Since these invariants converge, it follows that these two terms are also uniformly bounded. 
\end{proof}

Lemma~\ref{lem:BoundedByPositiveIntersection} shows that the sequence  $\bigl(\rho_i(\gamma)\bigr)_{i\in \N}$  admits a converging subsequence in $\PSL$. Doing this for all  $\gamma$ in the finite set of generators for  $\pi_1(S)$ provided by Lemma~\ref{lem:NiceGeneratorsForPi1}, we conclude that the sequence $(\rho_i)_{i\in \N}$ admits a converging subsequence in $\Hit(S)$. 

Therefore, every sequence of $\Hit(S)$ whose image under $\Phi$ converges in the polyhedron $\mathcal P$ admits a converging subsequence in $\Hit(S)$. This proves that the map $\Phi \colon \Hit(S) \to \mathcal P$ is proper, and concludes the proof of Proposition~\ref{prop:InvariantsProper}. 
\end{proof}

\begin{thm}
\label{thm:InvariantsHomeomorphism}
The map $\Phi \colon \Hit(S) \to \mathcal P$ is a homeomorphism from the Hitchin component $\Hit(S)$ to the polyhedron $\mathcal P \subset  \R^{6(g-1)(n-1)(n-2)} \times \CC(\lambda, \slits;\widehat\R^n)$.
\end{thm}

\begin{proof}
The map $\Phi$ is a local homeomorphism by Corollary~\ref{cor:InvariantsLocalHomeo}, and proper by Proposition~\ref{prop:InvariantsProper}. Since $\Phi$ is injective by Corollary~\ref{cor:HitchinDeterminedByInvariants} and since the convex polytope $\mathcal P$ is connected, this proves that $\Phi$ is a homeomorphism. 
\end{proof}

 \begin{rem}
 The formulas of \S \ref{subsect:ParamInjective}, in particular Lemma~\ref{lem:GroupActionAndInvariants}, provide an explicit construction for the inverse map $\Phi^{-1} \colon  \mathcal P \to \Hit(S) $. The boundedness estimates that we just used in the proof of Lemma~\ref{lem:BoundedByPositiveIntersection} show that the infinite products involved in these formulas do converge. This immediately proves that this inverse map $\Phi^{-1}$ is real analytic. 
 
 It can be shown that the forward map $\Phi$ is also analytic, using the fact \cite{BCLS} that the flag curve $\mathcal F_\rho \colon \partial_\infty \widetilde S \to \Flag$ depends real analytically on the homomorphism $\rho$. However, this is beyond the scope of this article. 
\end{rem}

\subsection{Constraints among triangle invariants, and on shearing cycles}
\label{subsect:ConstraintInvariantsBis}
 The Shearing Cycle Boundary Condition does more than connecting the boundary of  the shearing cycle $\sigma^\rho$ of a Hitchin character $\rho \in \Hit(S)$ to  its triangle invariants $\tau_{abc}^\rho(s)$. It also puts constraints between the triangle invariants themselves, and restricts the twisted relative tangent cycles that can occur as shearing cycles of Hitchin characters. As a complement to Theorem~\ref{thm:InvariantsHomeomorphism}, this section is devoted to emphasizing these somewhat unexpected phenomena, which we already encountered in Lemmas~\ref{lem:ImageBoundaryMap} and \ref{lem:ImageTheta}.

\begin{cor}
\label{cor:RestrictionTriangleInvariants}
A rotation invariant triangle data function  $\tau \in \R^{2(g-1)(n-1)(n-2)}$ is the triangle invariant $\tau^\rho$ of a Hitchin character $\rho \in \Hit(S)$  if and only if 
$$
\sum_{s\text{ slit of }\lambda} \ \sum_{b+c=n-a} \tau_{abc}(s) = \sum_{s\text{ slit of }\lambda} \ \sum_{b+c=a} \tau_{(n-a)bc}(s) 
$$
for every $a=1$, $2$, \dots, $n-1$. 

As a consequence, the triangle invariants of Hitchin characters form a linear subspace of codimension $\lfloor \frac{n-1}2 \rfloor$ in the space $  \R^{2(g-1)(n-1)(n-2)}$ of all rotation invariant triangle data functions.
\end{cor}

\begin{proof} Theorem~\ref{thm:InvariantsHomeomorphism} shows that $\tau$ is the triangle invariant of a Hitchin character if and only if there exists a relative cycle $\sigma\in \CC(\lambda, \slits; \widehat\R^{n-1})$ such that the pair $(\tau, \sigma)$ satisfies the Shearing Boundary Condition, and such that $\sigma $ satisfies the Positive Intersection Condition. 

The proof of Proposition~\ref{prop:ComputePolytopeDim}, and in particular  Lemmas~\ref{lem:ImageBoundaryMap} and \ref{lem:IntersectionImagesBoundaryMapTheta}, takes care of the first constraint. More precisely, with the notation of that proof, there exists $\sigma\in \CC(\lambda, \slits; \widehat\R^{n-1})$ such that $(\tau, \sigma)$ satisfies the Shearing Boundary Cycle Condition if and only if $\Theta(\tau)$ belongs to the image $\mathrm{im}(\partial) $. Lemma~\ref{lem:ImageBoundaryMap} shows that this is equivalent to the condition stated in Corollary~\ref{cor:RestrictionTriangleInvariants}, while Lemma~\ref{lem:IntersectionImagesBoundaryMapTheta} shows that $\Theta^{-1} \bigl( \mathrm{im}(\partial) \bigr)$ has codimension $\lfloor \frac{n-1}2 \rfloor$ in  $  \R^{2(g-1)(n-1)(n-2)}$. 

The only thing left to prove is that the Positive Intersection Condition has no impact on this property. Namely: If there exists $\sigma\in \CC(\lambda, \slits; \widehat\R^{n-1})$ such that $(\tau, \sigma)$ satisfies the Shearing Cycle Boundary Condition, the relative tangent cycle $\sigma$ can be chosen so that, in addition, it satisfies the Positive Intersection Condition. 

For this, we will use the existence of a closed twisted  tangent cycle $\sigma_0 \in \CC(\lambda; \widehat\R^{n-1})$ that satisfies the Positive Intersection Condition. An easy way to construct such a tangent cycle is to consider the shearing cycle $\sigma_0 = \sigma^{\rho_0} \in  \CC(\lambda, \slits; \widehat\R^{n-1})$ of a Hitchin character $\rho_0 \in \mathrm{Hit}_2(S) \subset \Hit(S)$ coming from a discrete homomorphism $\rho \colon \pi_1(S) \to \mathrm{PSL}_2(\R) \subset \PSL$. All triangle invariants $\tau_{abc}^{\rho_0}(s)$ of such a Hitchin character are equal to 0; the easiest way to see this is to apply Lemma~\ref{lem:SymmetriesTripleRatios} and to observe that, for every triangle component of $\widetilde S - \widetilde\lambda$ with vertices $\widetilde s$, $\widetilde s'$ and $\widetilde s''$, there is an element of $\PGL$ coming from an element of $\mathrm{PGL}_2(\R)$ that fixes the flag $\F_{\rho_0}(\widetilde s) \in \Flag$ and exchanges $\F_{\rho_0}(\widetilde s)$ and $\F_{\rho_0}(\widetilde s)$. It therefore follows from the Shearing Cycle Boundary Condition that $\partial \sigma_0=0$, namely that $\sigma_0$ is closed. 

If the rotation invariant triangle data function $\tau \in \R^{2(g-1)(n-1)(n-2)}$ satisfies the conditions of Corollary~\ref{cor:RestrictionTriangleInvariants}, we just showed that there exists $\sigma \in  \CC(\lambda, \slits; \widehat\R^{n-1})$ such that $(\tau, \sigma)$  satisfies the Shearing Cycle Boundary Condition. For $c>0$ sufficiently large, $\sigma + c \sigma_0$ satisfies the Positive Intersection Condition since this property holds for $\sigma_0$ and since the space of transverse measures for $\widehat\lambda$ is finite-dimensional \cite{Kato, Papa}. In addition, the pair $(\tau, \sigma + c \sigma_0)$ satisfies the Shearing Cycle Boundary Condition since $\partial(\sigma + c\sigma_0) = \partial \sigma$, and the Triangle Rotation Condition by choice of  $\tau$. As a consequence, Theorem~\ref{thm:InvariantsHomeomorphism} provides a Hitchin character $\rho\in \Hit(S)$ whose triangle invariant  $\tau^\rho$  is $\tau$, and whose shearing cycle $\sigma^\rho$ is equal to  $ \sigma + c \sigma_0$.
\end{proof}

Lemma~\ref{lem:IntersectionImagesBoundaryMapTheta} and Proposition~\ref{prop:ComputeTwistedTgentCycles} similarly give the following characterization of the shearing cycles of Hitchin characters. 

\begin{cor}
\label{cor:RestrictionShearingCycle}
Suppose that $n>3$. 
For a twisted relative tangent cycle $\sigma \in \CC(\lambda, \slits; \widehat\R^{n-1}) $ and for $a=1$, $2$, \dots, $n-1$, let $\partial\sigma_a$ be the $a$--th component of its boundary $\partial\sigma \colon \{ \text{slits of } \widehat\lambda \} \to \R^{n-1}$. Then, $\sigma$ is the shearing cycle $\sigma^\rho$ of a Hitchin character $\rho \in \Hit(S)$ if and only if $\sigma$ satisfies the Positive Intersection Condition and 
\begin{align*}
&\partial \sigma_{n-1}(s_1^+) = 0 \\
\text{and }
&\partial\sigma_1(s_1^+) = \sum_{a=2}^{n-2} ({\textstyle \frac{a-1}{n-3}} -1) \partial\sigma_a(s_1^+) +   \sum_{a=2}^{n-2} {\textstyle \frac{a-1}{n-3}}\, \partial\sigma_a(s_2^+)+  \sum_{a=2}^{n-2} {\textstyle \frac{a-1}{n-3}}\, \partial\sigma_a(s_3^+)
\end{align*}
whenever $s_1^+$, $s_2^+$ and $s_3^+$ are positive slits of the orientation cover $\widehat \lambda$ that project to the three spikes of the same component $T$ of $S-\lambda$. 

As a consequence, the shearing cycles of Hitchin characters form an open convex  polyhedral cone in a linear subspace of codimension $24(g-1)$ of $ \CC(\lambda, \slits; \widehat\R^{n-1})\cong \R^{18(g-1)(n-1)} $. \qed
\end{cor}

\begin{cor}
\label{cor:RestrictionShearingCycle3and2}
When $n=3$, a twisted relative tangent cycle $\sigma \in \CC(\lambda, \slits; \widehat\R^{2}) $   is the shearing cycle $\sigma^\rho$ of a Hitchin character $\rho \in \mathrm{Hit}_3(S)$ if and only if $\sigma$ satisfies the Positive Intersection Condition and 
\begin{align*}
\partial \sigma_2(s_1^+) &= 0 \\
\text{and }
\partial \sigma_1(s_1^+) &= \partial \sigma_1(s_2^+) =\partial \sigma_1(s_3^+) =0
\end{align*}
whenever $s_1^+$, $s_2^+$ and $s_3^+$ are positive slits of the orientation cover $\widehat \lambda$ that project to the three spikes of the same component $T$ of $S-\lambda$. 
As a consequence, the shearing cycles of Hitchin characters form an open convex  polyhedral cone in a subspace of codimension $20(g-1)$ of $ \CC(\lambda, \slits; \widehat\R^{2})\cong \R^{36(g-1)}$. 

When $n=2$, a twisted relative tangent cycle $\sigma \in \CC(\lambda, \slits; \widehat\R) $   is the shearing cycle $\sigma^\rho$ of a Hitchin character $\rho \in \mathrm{Hit}_2(S)$ if and only if $\sigma$ is closed and satisfies the Positive Intersection Condition. \qed
\end{cor}

We conclude this article by giving, in the next two sections, two brief applications of the machinery developed in this article. In particular, these applications require the full generality of geodesic laminations (as opposed to the much simpler case of geodesic laminations with finitely many leaves considered in \cite{BonDre}).

\section{The action of pseudo-Anosov homeomorphisms on the Hitchin component}
\label{bigsect:PseudoAnosov}

Let $\phi \colon S \to S$ be a pseudo-Anosov homeomorphism of the surface $S$. We can use our parametrization of  $\Hit(S)$ to show that the action of $\phi$ on the Hitchin component $\Hit(S)$ is concentrated in a relatively small factor of $\Hit(S)$. This section is only intended as an illustration of the possible applications of the main results of the article;  we are consequently limiting its scope to avoid making an already long article much longer.

The pseudo-Anosov property  of $\phi$ is usually expressed in terms of transverse measured foliations on the surface $S$ \cite{ThuBAMS, FLP}. It will be more convenient to use the point of view of  \cite{CasBlei}, so that the homeomorphism $\phi \colon S \to S$  is (isotopic to) a pseudo-Anosov homeomorphism if  there exist a  geodesic lamination $\lambda^{\mathrm s}$, a transverse measure $\mu^{\mathrm s}$ for $\lambda^{\mathrm s}$, and a number $R>1$ such that, after an isotopy of $\phi$:
\begin{enumerate}
\item each component of the complement of the topological support $\lambda^{\mathrm s}$ of $\mu^{\mathrm s}$ is a topological disk;
\item $\phi(\lambda^{\mathrm s}) = \lambda^{\mathrm s}$; 
\item the pull back $\phi^*(\mu^{\mathrm s})$ of the transverse measure $\mu^{\mathrm s}$ is equal to $R\mu^{\mathrm s}$. 
\end{enumerate}

The homomorphism $\phi \colon S \to S$ acts on the character variety $\mathcal X_{\PSL}(S)$ as $\rho \mapsto \phi_* \circ  \rho$, where $\phi_* \colon \pi_1(S) \to \pi_1(S)$ is any homomorphism induced by $\phi$ (by choosing a path joining the base point to its image under $\phi$). When $\rho\in \mathcal X_{\PSL}(S)$ comes from a Teichm\"uller character of $\mathrm{Hit}_2(S)$, it is immediate that so does $\rho \circ \phi_*$. By connectedness, it follows that the action $\rho \mapsto \rho \circ \phi_*$ respects the Hitchin component $\Hit(S)$. 

Replacing $\phi$ by one of its powers does not significantly change its dynamics. 

\begin{lem}
\label{lem:GoodPowerPseudoAnosov}
There exists an integer $k>0$ and a maximal geodesic lamination $\lambda^+$ containing  $\lambda^{\mathrm s}$ such that $\phi^k(\lambda^+) = \lambda^+$ after isotopy of $\phi^k$. In addition, $\phi^k$ can be chosen so that it respects each slit of $\lambda^+$.
\end{lem}
\begin{proof}
The homeomorphism $\phi$ permutes the finitely many slits of $\lambda^{\mathrm s}$. Therefore, there exists $k$ such that $\phi^k$ respects each slit.

Let $\lambda^+$ be any maximal geodesic lamination containing $\lambda^{\mathrm s}$. Because each component of $S-\lambda^{\mathrm s}$ is a topological disk, or more precisely an ideal polygon, $\lambda^+$ is obtained from $\lambda^{\mathrm s}$  by adding finitely many diagonal leaves joining spikes of these polygons. Since $\phi^k$ respects each slit of $\lambda^{\mathrm s}$, namely each spike of $S-\lambda^{\mathrm s}$, it can easily be isotoped to respect these diagonal leaves (as well as $\lambda^{\mathrm s}$). By construction, $\phi^k$  respects  each slit of $\lambda^+$.
\end{proof}

We can now use the maximal geodesic lamination $\lambda^+$ to construct a parametrization of the Hitchin component $\Hit(S)$ by the polytope $ \mathcal P \subset \R^{6(g-1)(n-1)(n-2)} \times \CC(\lambda^+, \slits; \widehat \R^{n-1})$ as in Theorem~\ref{thm:InvariantsHomeomorphism}. 

Because $\phi^k$ respects the geodesic lamination $\lambda^+$, it acts on $ \CC(\lambda^+, \slits;  \widehat \R^{n-1})$ as follows. Lift $\phi$ to a homeomorphism $\widetilde \phi \colon \widetilde S \to \widetilde S$ of the universal cover $\widetilde S$; in particular, $\widetilde \phi^k$ respects the pre-image $\widetilde\lambda^+$ of $\lambda^+$. Then, using the point of view of \S \ref{subsect:RelTgtCyclesDifferentView}, define $\phi_{\bullet}^k \colon  \CC(\lambda^+, \slits;  \widehat \R^{n-1}) \to \CC(\lambda^+, \slits;  \widehat \R^{n-1}) $ by the property that $\phi_\bullet^k (\alpha)(T, T')= \alpha \bigl( \widetilde \phi^k(T), \widetilde \phi^k(T') \bigr)$ for any two components $T$, $T'$ of $\widetilde S - \widetilde\lambda^+$. 

\begin{prop}
\label{prop:ActionPseudoAnosov}
For the homeomorphism 
$$\Phi \colon\Hit(S) \to \mathcal P \subset \R^{6(g-1)(n-1)(n-2)} \times \CC(\lambda^+, \slits; \widehat \R^{n-1})$$ 
provided by Theorem~{\upshape\ref{thm:InvariantsHomeomorphism}}, the action of $\phi^k$ on $\Hit(S)$ corresponds to the restriction to $\mathcal P$ of the product of the identity $\Id_{\R^{6(g-1)(n-1)(n-2)}}$ and of the action of $\phi^k$ on $ \CC(\lambda^+, \slits; \widehat \R^{n-1})$. 
\end{prop}

\begin{proof} For $\rho \in \Hit(S)$, we need to compare the triangle invariants $\tau_{abc}^{\rho \circ\phi_*^k}(s)$ and the  shearing cycle $\sigma^{\rho \circ\phi_*^k} \in \CC(\lambda^+, \slits; \widehat \R^{n-1})$ of $\rho \circ\phi_*^k$ to those of $\rho$. 

Lift $\phi$ to a homeomorphism $\widetilde \phi \colon \widetilde S \to \widetilde S$ of the universal cover $\widetilde S$, which is equivariant with respect to $\phi_* \colon \pi_1(S) \to \pi_1(S)$ in the sense that $\widetilde\phi(\gamma x) = \phi_*(\gamma)\widetilde \phi(x)$ for every $x\in \widetilde S$ and $\gamma \in \pi_1(S)$. The flag maps $\F_{\rho}$ and $\F_{\rho \circ \phi_*^k} \colon \partial_\infty \widetilde S \to \Flag$ are then related by the property that $\F_{\rho \circ \phi_*^k} = \F_\rho \circ \widetilde \phi^k$. Going back to the definitions of these invariants and remembering that $\phi^k$ respects each slit of $\lambda^+$, it immediately follows that $\rho$ and $\rho \circ \phi_*^k$ have the same triangle invariants $\tau_{abc}^{\rho \circ\phi_*^k}(s)=\tau_{abc}^{\rho}(s)$, and that $\sigma^{\rho \circ\phi_*^k}= \phi_\bullet^k(\sigma^{\rho})$. 
\end{proof}

This is better described in terms of the map  $\pi \colon \Hit(S) \to  \R^{6(g-1)(n-1)(n-2)} $ corresponding to the projection of $\Hit(S)\cong  \mathcal P $ to the first factor of $ \R^{6(g-1)(n-1)(n-2)} \times \CC(\lambda^+, \slits; \widehat \R^{n-1})$. Namely, $\pi$ associates its triangle invariants $\tau_{abc}^{\rho}(s)$ to a Hitchin character $\rho \in \Hit(S)$. The image $\mathcal L = \pi\bigl( \Hit(S) \bigr)$ is the vector space of dimension $2(g-1)(n-1)(n-2)- \lfloor \frac{n-1}2 \rfloor$ determined by Corollary~\ref{cor:RestrictionTriangleInvariants}.  This defines a fibration $\pi \colon \Hit(S) \to \mathcal L$, where the fiber $\pi^{-1}(\tau)$ above each $\tau \in \mathcal L$ is a convex polyhedral cone of dimension $3(g-1)(n-1)+  \lfloor \frac{n-1}2 \rfloor$ in $ \CC(\lambda^+, \slits; \widehat \R^{n-1})\cong \R^{18(g-1)(n-1)}$. 

Then, Proposition~\ref{prop:ActionPseudoAnosov} states that the action of $\phi^k$ on $\Hit(S)$ respects each fiber $\pi^{-1}(\tau)$, and acts on each of these polyhedral cones $\pi^{-1}(\tau) \subset  \CC(\lambda^+, \slits; \widehat \R^{n-1})$ by restriction of $\phi_{\bullet}^k \colon  \CC(\lambda^+, \slits;  \widehat \R^{n-1}) \to \CC(\lambda^+, \slits;  \widehat \R^{n-1}) $. 

In $U$ is a train track neighborhood of $\lambda^+$, the endomorphism $\phi_{\bullet}^k$ of $  \CC(\lambda^+, \slits;  \widehat \R^{n-1}) \cong H_1(U, \delv U; \widetilde \R^{n-1})$ can be explicitly explicitly described in terms of a classical object associated to the pseudo-Anosov homeomorphism $\phi$, namely the incidence matrix of $\phi$ with respect to the train track $U$ (see for instance \cite[Exp. 9-10]{FLP}). However, this would take us beyond the intended scope of this article.

\section{Length functions of measured laminations}
\label{bigsect:LengthMeasLam}

One of the motivations for this article is to extend to the Hitchin component the  differential calculus of lengths of simple closed curves that was developed for hyperbolic geometry in \cite{Thu0, Thu1, Bon97b, Bon96}.

For a Hitchin character $\rho\in\Hit(S)$, the length functions $\ell_1^\rho$, $\ell_2^\rho$, \dots, $\ell_{n-1}^\rho$, of \cite{Dre1} and \S \ref{subsect:LengthFunctions} can be restricted to Thurston's space $\ML(S)$ of measured geodesic laminations. There is just a little subtlety, which is that the geodesic currents discussed in \S \ref{subsect:LengthFunctions}  form  a completion of the set of homotopy classes of \emph{oriented} closed curves, whereas $\ML(S)$ completes the set of homotopy classes of \emph{unoriented} simple closed curves.

An unoriented simple closed curve $\gamma$ in $S$ defines two oriented curves $\gamma^*$ and $\gamma^{**}$, one for each orientation. Then there is a unique continuous embedding $\iota \colon \ML(S) \to \CC(S)$ that is homogeneous, in the sense that $\iota(t\mu) = t\iota(\mu)$ for every $\mu \in \ML(S)$ and every $t>0$, and such that $\iota(\gamma) = \frac12(\gamma^* + \gamma^{**})$ for every simple closed curve $\gamma \in \ML(S)$; see for instance \cite{Bon88}. Combining this embedding with $\ell^\rho_a \colon \CC(S) \to \R$ defines, for each $a=1$, $2$, \dots, $n-1$, a  length function $\ell^\rho_a \colon \ML(S) \to \R$. The definition, and in particular the introduction of the factor $\frac12$, is designed so that when $n=2$ the  function $\ell_1^\rho$ coincides with Thurston's length function $\ell^\rho \colon \ML(S) \to \R$ for the hyperbolic metric on $S$ associated to $\rho \in \mathrm{Hit}_2(S)$, which plays a fundamental r\^ole in hyperbolic geometry; see for instance  \cite{ThuBAMS, FLP, Thu0, Bon88, Mirza} for a few applications of this length function~$\ell^\rho$. 

Because $\ell_a^\rho(\gamma^{**}) = \ell_{n-a}^\rho(\gamma^*)$, the length functions $\ell_a^\rho$ and $\ell_{n-a}^\rho$ coincide on $\ML(S)$ so that, in practice, we have only $\lfloor \frac n2 \rfloor$  length functions $\ell^\rho_a \colon \ML(S) \to \R$.

The space $\ML(S)$ of measured geodesic laminations is homeomorphic to $\R^{6(g-1)}$, but admits no differentiable structure that is respected by the action of the mapping class group. As a consequence, we cannot use the standard concepts of differential calculus on this space.

However, $\ML(S)$ is naturally endowed with a  piecewise integral linear structure; this means that it admits an atlas locally modelling it over $\R^{6(g-1)}$ where the coordinate changes are piecewise linear and where the linear pieces of these coordinate changes have integer coefficients \cite{Thu0, PenH}. In particular, because  a piecewise linear map does have a tangent map, a consequence of the piecewise linear structure is that $\ML(S)$ admits a well-defined tangent space $T_\mu \ML(S)$ at each point $\mu\in \ML(S)$. 

Each tangent space $T_\mu \ML(S)$ is homeomorphic to $\R^{6(g-1)}$ and is homogeneous, in the sense that there is a well defined multiplication of tangent vectors by non-negative numbers, but it is not always a vector space. Indeed, there exists points $\mu\in \ML(S)$ where the tangent space $T_\mu \ML(S)$ admits no vector space structure which is respected by all coordinate charts; a typical example of such points  are the positive real multiples of simple closed curves, which are dense in $\ML(S)$. Conversely, at a measured geodesic lamination $\mu$ whose support is a maximal geodesic laminations, the piecewise integral linear structure does define a natural vector space structure on the tangent space $T_\mu \ML(S)$; these $\mu$ form a subset of full measure in $\ML(S)$.  See \cite{Thu1} for instance.

\begin{thm}[{\cite[\S 3.2]{Dre1}}]
\label{thm:DerivativeLengthMLexists}
For a Hitchin character $\rho\in \Hit (S)$ and for $a=1$, $2$, \dots, $\lfloor \frac n2 \rfloor$, the length function $\ell^\rho_a \colon \ML(S) \to \R$ admits a tangent map $T_\mu \ell^\rho_a \colon T_\mu \ML(S) \to \R$ at each $\mu\in \ML(S)$,  in the following sense. For $\mu\in \ML(S)$ and $v\in T_\mu\ML(S)$, let $t\mapsto \alpha_t$ be a curve in $ \mathcal{ML}(S)$ such that $\alpha_0=\mu$ and the  right-hand-side tangent derivative $\frac {d\ }{dt^+}\alpha_t{}_{|t=0}$ exists and is equal to $v$, then  $\frac {d\ }{dt^+} \ell_r^i(\alpha_t)_{|t=0} = T_{\mu} \ell^\rho_a( v )\in \R$.  \qed
\end{thm}

The  proof of Theorem~\ref{thm:DerivativeLengthMLexists}  relies on two key ingredients: the analytic interpretation \cite{Bon97a, Bon97b} of tangent vectors $v\in T_\mu \ML(S)$ as a certain type of  H\"older geodesic currents as in \S \ref{subsect:LengthFunctions}; and the continuity of the length function $\ell_a^\rho \colon \CH(S) \to \R$  for the H\"older topology, proved in \cite{Dre1}. In particular, $T_{\mu} \ell^\rho_a( v )$ is  equal to the $a$--th length $\ell_a^\rho(v)$ of the H\"older geodesic current $v\in \CH(S)$ associated to $v\in T_\mu\ML(S)$. 

The results of the current paper, and in particular Theorem~\ref{thm:ShearingAndLength},  provide a description of the tangent map $T_{\mu} \ell^\rho_a$ on the faces of $T_\mu \ML(S)$. 

This is based on a more combinatorial interpretation, also developed in  \cite{Bon97b, Bon97a}, of tangent vectors $v\in T_\mu \ML(S)$ as tangent cycles for geodesic laminations $\lambda$ containing the support $\lambda_\mu$ of $\mu$; these tangent cycles must satisfy a certain positivity condition (unrelated to the Positive Intersection Condition of \S \ref{subsect:ConstraintsInvariants}). This decomposes  the tangent space $ T_\mu \ML(S)$  into a family of cones $F_\lambda$, indexed by  geodesic laminations $\lambda$  containing the support $\lambda_\mu$ of $\mu$, where $F_\lambda$ consists of  those tangent vectors $v\in T_\mu \ML(S)$ that can be described as tangent cycles for $\lambda$. In particular, each $F_\lambda$ is naturally identified to  a convex polyhedral cone in the vector space $\CC(\lambda;\R)$ of all tangent cycles for $\lambda$, and the partial vector space structure induced on $F_\lambda$ by $\CC(\lambda;\R)$ is compatible with the piecewise  linear structure of $\ML(S)$. The $F_\lambda$  are the \emph{faces} of $ T_\mu \ML(S)$ for the piecewise  linear structure of $\ML(S)$. See \cite{Thu1} for a slightly different approach. 
  
 In the generic case where the support $\lambda_\mu$ of $\mu \in \ML(S)$ is maximal there is only one face  in $T_\mu \ML(S)$, namely $F_{\lambda_\mu}$. This face $F_{\lambda_\mu}$ is equal to the whole vector space $\CC(\lambda_\mu; \R)$ of tangent cycles for $\lambda_\mu$.

Because of the positivity condition involved in the interpretation of tangent vectors  $v\in T_\mu \ML(S)$ as tangent cycles for geodesic laminations, 
 it is quite possible that different geodesic laminations $\lambda$ and $\lambda'$ define the same face $F_\lambda= F_{\lambda'}$. 
The correspondence $\lambda \mapsto F_\lambda$ can be made bijective by restricting attention to chain recurrent geodesic laminations \cite{Thu1, Bon97a}. Instead, we will focus on the case where the geodesic lamination $\lambda$ is maximal, as it is better adapted to our purposes. Every geodesic lamination $\lambda'$ is contained in a maximal geodesic lamination $\lambda$, so that every face of  $T_\mu \ML(S)$ is contained in a face $F_\lambda$ associated to a maximal geodesic lamination $\lambda$. Note that, although $\lambda$ is maximal, the dimension of the associated face $F_\lambda$ may be significantly smaller than the dimension $6(g-1)$ of  $T_\mu \ML(S)$.

\begin{thm}
\label{thm:DerivativeLengthMLFormula}

The tangent map $T_\mu \ell^\rho_a \colon T_\mu \ML(S) \to \R$ is linear on each face of $T_\mu \ML(S)$. 

More precisely, if the face $F_\lambda \subset T_\mu \ML(S)$ is associated to a maximal geodesic lamination $\lambda$, if we interpret the tangent vector $v\in F_\lambda$ as a tangent cycle for  $\lambda$, and if $\sigma^\rho \in \CC(\lambda, \slits; \widehat\R^{n-1})$ is the shearing cycle of $\rho$, then 
$$  
T_{\mu} \ell^\rho_a( v ) = [\sigma_a^\rho] \cdot [v]
$$
where, as in {\upshape \S \ref{subsect:RelativeHomologyTangentCycles}} and {\upshape \S \ref{subsect:ShearingAndLength}}, the dot $\cdot$ denotes the algebraic intersection number in a train track neighborhood $\widehat U$ of the orientation cover $\widehat\lambda$ of $\lambda$, $[\sigma_a^\rho] \in H_1(\widehat U, \delv \widehat U; \R)$ is the $a$--th component of the twisted relative homology class $[\sigma^\rho] \in H_1(U, \delh U; \widetilde \R^{n-1}) \subset H_1(\widehat U, \delh \widehat U; \R^{n-1})$ defined by  $\sigma^\rho\in \CC(\lambda, \slits; \widehat \R^{n-1})$, and $[v] \in H_1(\widehat U; \R)$ is the homology class represented by $v\in \CC(\lambda; \R) \subset \CC(\widehat\lambda; \R)$. 

\end{thm}

\begin{proof} We already observed that $T_{\mu} \ell^\rho_a( v ) = \ell^\rho_a( v )$ where the right hand side interprets $v$ as a tangent cycle for $\lambda$ and involves the function $\ell_a^\rho \colon \CC(\widehat\lambda) \to \R$ introduced in \S \ref{subsect:LengthFunctions}. The formula  then occurs as a special case of  Theorem~\ref{thm:ShearingAndLength}.
\end{proof}


\bibliographystyle{amsalpha}
 \bibliography{HitchinLam} 
 \end{document}